\documentclass[12pt]{article}
\usepackage{amssymb}
\usepackage{mathrsfs}
\usepackage{amsmath}
\usepackage{amsthm}
\usepackage{amstext}
\usepackage{amsopn}
\usepackage{graphicx}
\usepackage{amssymb}
\usepackage{latexsym}
\usepackage{amsfonts,euscript}

\newtheorem{theorem}{Theorem}[section]

\newtheorem{definition}{Definition}[section]

\newtheorem{lemma}{Lemma}[section]
\newtheorem{proposition}{Proposition}[section]
\newtheorem{remark}{Remark}[section]
\numberwithin{equation}{section} \textwidth 150mm \textheight 222mm
\newcommand{\f}{\frac}
\newcommand{\R}{\mathbb {R}}
\newcommand{\Z}{\mathbb {Z}}
\newcommand{\un}{\underline}
\newcommand{\mc}{\mathcal}
\newcommand{\D}{\text{D}}
\newcommand{\di}{\text{d}}
\begin{document}
\date{}
\title{\bf Bistable Traveling Waves for Monotone Semiflows with Applications}
\author{Jian Fang$^{1,2}$ and Xiao-Qiang
Zhao$^{2}$\thanks {Corresponding author.
 E-mail address: zhao@mun.ca} \\
{\small $^1$ Department of Mathematics, Harbin Institute of Technology}\\
{\small Harbin 150001, China}\\
{\small $^2$ Department of Mathematics and Statistics,
Memorial University of Newfoundland}\\
{\small St. John's, NL A1C 5S7, Canada }}
 \maketitle

\begin{abstract}
This paper is devoted to the study of traveling waves for monotone
evolution systems of bistable type. Under an abstract setting, we
establish the existence of bistable traveling waves for discrete and
continuous-time monotone semiflows. This result is then extended to
the cases of periodic habitat and weak compactness, respectively. We
also apply the developed theory to four classes of evolution
systems.
\end{abstract}

\


\


\section{Introduction}
In this paper, we study traveling waves for monotone (i.e.,
order-preserving) semiflows $\{Q_t\}_{t\in\mc{T}}$ with the
bistability structure on some subsets of the space
$\mc{C}:=C(\mc{H},\mc{X})$ consisting of all continuous functions
from the habitat $\mc{H}$($=\R$ or $\Z$) to the Banach lattice
$\mc{X}$, where $\mc{T}=\Z^+$ or $\R^+$ is the set of evolution
times. Here the bistability structure is generalized from a number
of studies for various evolution equations. It means that the
restricted semiflow on $\mc{X}$ admits two ordered stable
equilibria, between which all others are unstable. We focus on the
existence of traveling waves connecting these two stable equilibria,
which are called bistable traveling waves. This setting allows us to
study not only autonomous and time-periodic evolution systems in a
homogeneous habitat (media), but also those in a periodic habitat.
Besides, the obtained results can be extended to the semiflows with
weak compactness on some subsets of the space $\mc{M}$ consisting of
all monotone functions from $\R$ to $\mc{X}$.

To explain the concept of the bistability structure, we recall some
related works on typical evolution equations. Fife and McLeod
\cite{FM77,FM81} proved the existence and global asymptotic
stability of monotone traveling waves for the following
reaction-diffusion equation:
\begin{equation}\label{Eq:FifeMcleod}
u_t=u_{xx}+u(1-u)(u-a),\quad x\in\R,\, t>0,
\end{equation}
where $a\in (0,1)$. Clearly, the restriction of system
\eqref{Eq:FifeMcleod} on $\mc{X}=\R$ is the ordinary differential
equation $u'= u(1-u)(u-a)$, which admits a unique unstable
equilibrium between two ordered and stable ones. The same property
is shared by the nonlocal dispersal equation in
\cite{BatesFRW,Coville,Yagisita} and the lattice equations  in
\cite{BatesC99,Zinner91,Zinner92}. Chen \cite{Chen97} studied a
general nonlocal evolution equation $u_t=\mc{A}(u(\cdot,t))$, which
also possess the above bistability structure. Some related
investigations on discrete-time equations can be found in
\cite{Lui83,CoutinhoFernandez}. For the time-periodic
reaction-diffusion equation $u_t=u_{xx}+f(t,u)$, the spatially
homogeneous equation is a time-periodic ordinary differential
equation. In this case, the equilibrium in the bistability structure
should be understood as the time-periodic solution. Under such
bistability assumption, Alikakos, Bates and Chen
\cite{AlikakosBatesChen} obtained the existence of bistable
time-periodic traveling waves. Recently, Yagisita \cite{Yagisita}
studied bistable traveling waves for discrete and continuous-time
semiflows on the space consisting of all left-continuous and
non-decreasing functions from $\R$ to $\mc{X}=\R$ under the
assumption that there is exactly one intermediate unstable
equilibrium. It should be mentioned that the result in
\cite{Yagisita} for continuous-time semiflows requires an additional
assumption on the existence of a pair of upper and lower solutions.

Note that the restrictions on $\mc{X}=\R$ of the afore-mentioned
systems are all scalar equations, and hence, there is only one
unstable equilibrium in between two stable ones. But in the case
where $\mc{X}=\R^n$, there may be multiple unstable equilibria. This
is one of the main reasons why some ideas and techniques developed
for scalar equations can not be easily extended to higher
dimensional systems. Volpert \cite{VVV} established the existence
and stability of traveling waves for the bistable reaction-diffusion
system $u_t=D \Delta u+f(u)$ by using topological methods, where $D$
is a positive definite diagonal matrix. Fang and Zhao
\cite{FangZhaoJDDE} further extended these results to the case where
$D$ is semi-positive definite via the vanishing viscosity approach.

Consider the following parabolic equation in a cylindrical domain
$\Sigma=\R\times \Omega$:
\begin{equation}\label{Eq:BereNiren}
\begin{cases}
u_t=\Delta u+\alpha(y)u_x+f(u),& x\in \R,
y=(y_1,\cdot\cdot\cdot,y_{n-1})\in \Omega,\, t>0,\\
\f{\partial u}{\partial \nu}=0  &\text{on $\R\times
\partial\Omega\times(0,+\infty)$},
\end{cases}
\end{equation}
where $f$ is of the same type as the nonlinearity in
\eqref{Eq:FifeMcleod} and $\Omega$ is a bounded domain with smooth
boundary in $\R^{n-1}$. Obviously, the restriction of the solution
semiflow of \eqref{Eq:BereNiren} on $\mc{X}=C(\bar \Omega, \R)$
gives rise to the following $x$-independent system:
\begin{equation}\label{Eq:BereNirenSpatialHomo}
\begin{cases}
u_t=\Delta_y u+f(u),&y\in \Omega,\, t>0,\\
\f{\partial u}{\partial \nu}=0  &\text{on
$\partial\Omega\times(0,+\infty)$}.
\end{cases}
\end{equation}
One can see from Matano \cite{Matano} (or Casten and Holland
\cite{CastenHolland}) that any nonconstant steady state of
\eqref{Eq:BereNirenSpatialHomo} is linearly unstable when the domain
$\Omega$ is convex. It follows that if $\Omega$ is convex, then
\eqref{Eq:BereNiren} admits the bistability structure: its
$x$-independent system has two (constant) linearly stable steady
states, between which all others are linearly unstable. In such a
case, Berestycki and Nirenberg \cite{BereNiren} obtained the
existence and uniqueness of bistable traveling waves. In the case
where $\Omega$ is an appropriate dumbbell-shaped domain, Matano
\cite{Matano} constructed a counterexample to show that
\eqref{Eq:BereNirenSpatialHomo} has stable non-constant steady
states, and Berestycki and Hamel \cite{BerHamelnonexistence} also
proved the nonexistence of traveling waves connecting two stable
constant steady states. For bistable traveling waves in time-delayed
reaction-diffusion equations, we refer to
\cite{Schaaf,SmithZhao1,MaWu07,WangLiRuan}. For such an equation
with time delay $\tau>0$, one can choose $\mc{X}=C([-\tau,0],\R)$ so
that its solution semiflow has the bistability structure.

Recently, there is an increasing interest in reaction diffusion
equations in periodic habitats. A typical example is
\begin{equation}\label{Eq:PeriodicDiff}
u_t=(d u_x)_x +f(u), \quad x\in \R,\, t>0,
\end{equation}
where $d\in C^1(\R,\R)$ is a positive periodic function with period
$r>0$. Define $\mc{Y}:=C([0,r],\R)$ and $C_{per}(\R,\R):=\{f\in
C(\R,\R):f(x+r)=f(x),\forall x\in \R\}$. It is easy to see that
\begin{equation*}
C(\R,\R)=\{f\in C(r\Z, \mc{Y}): f(ri)(r)=f(r(i+1))(0),\forall i\in
\Z\}:=\mc{K},
\end{equation*}
and that any element in $C_{per}(\R,\R)$ is a constant function in
$\mc{K}$. Thus, the solution semiflow of \eqref{Eq:PeriodicDiff} on
$C(\R,\R)$ can be regarded as a conjugate semiflow on $\mc{K}$, and
hence, the bistability structure should be understood as: the
restriction of the solution semiflow of \eqref{Eq:PeriodicDiff} on
$C_{per}(\R,\R)$ has two ordered $r$-periodic steady states, between
which all others are unstable. Assuming that the function $f$ is of
bistable type, Xin \cite{Xin} obtained the existence of spatially
periodic (pulsating) traveling wave as long as $d$ is sufficiently
close to a positive constant in a certain sense (see also
\cite{Xin93,XinReview}). However, whether the solution semiflow of
\eqref{Eq:PeriodicDiff} admits the bistability structure remains an
open problem. We will give an affirmative answer in section 6.3 and
further improve Xin's existence result. Meanwhile, a counterexample
will be constructed to show that the solution semiflow of
\eqref{Eq:PeriodicDiff} has no bistability structure in the general
case of varying $d(x)$. More recently, Chen, Guo and Wu
\cite{ChenGuoWu} proved the existence, uniqueness and stability of
spatially periodic traveling waves for one-dimensional lattice
equations in a periodic habitat under the bistability assumption.
There are also other types of bistable waves (see,
e.g.,\cite{BereHamelMatano,Shen99almost}). For monostable systems in
periodic habitats, we refer to
\cite{BerHamCPAM,BerHamNad1,BerHamNad2,GuoHamel,GuoWu,LiangZhaoJFA}
and references therein.

In general, there are multiple intermediate unstable equilibria in
between two stable ones in the case where the space $\mc{X}$ is high
dimensional. Meanwhile, it is possible for the given system to have
intermediate unstable time-periodic orbits in $\mc{X}$. These give
more difficulties to the study of bistable semiflows than monostable
ones, whose restricted systems on $\mc{X}$ have only one unstable
and one stable equilibria. To overcome these difficulties, we will
show that all these unstable equilibria and all points in these
periodic orbits are unordered in $\mc{X}$ under some appropriate
assumptions. With this in mind, a bistable system can be regarded as
the union of two monostable systems although such a union is not
unique. From this point of view, we establish a link between
monostable subsystems and the bistable system itself, which plays a
vital role in the propagation of bistable traveling waves. This link
is stated in terms of spreading speeds of monostable subsystems (see
assumption (A6)). For spreading speeds of various monostable
evolution systems, we refer to
\cite{AroWei2,BerHamNad1,GuoHamel,LiangZhao,LiangZhaoJFA,Lui1989,
Wein1982,Wein2002,ZhaoSurvey} and references therein.

In our investigation, we consider seven cases: (I)\, $\mc{T}=\Z^+$
and $\mc{H}=\R$;\, (II)\, $\mc{T}=\Z^+$ and $\mc{H}=\Z$;\, (III)\,
$\mc{T}=\R^+$ and $\mc{H}=\R$;\,  (IV)\, $\mc{T}=\R^+$ and $\mc{H}=
\Z$;\, (V)\, periodic habitat; \, (VI)\, weak compactness;\, (VII)\,
time periodic. For the case (I), we combine the above observations
for general bistable semiflows and Yagisita's perturbation idea in
\cite{Yagisita} to prove the existence of traveling waves. For the
case (III), we use the bistable traveling waves
$\phi_\pm(x+c_{\pm,s})$ of discrete-time semiflows
$\{(Q_s)^n\}_{n\ge 0}$ to approximate the bistable wave of the
continuous-time semiflow $\{Q_t\}_{t\ge 0}$. This new approach
heavily relies on an estimation of the boundedness of
$\f{1}{s}c_{\pm,s}$ as $s\to 0$, which is proved by the bistability
structure of the semiflow (see inequalities
\eqref{BistabilityInequ1} and (\ref{BistabilityInequ2})). It turns
out that our result does not require the additional assumption on
the existence of a pair of upper and lower solutions as in
\cite{Yagisita}. In the case (II), both the evolution time $\mc{T}$
and the habitat $\mc{H}$ are discrete, a traveling wave $\psi(i+cn)$
of $\{Q^n\}_{n\ge 0}$ cannot be well-defined in the usual way
because the wave speed $c$ and hence, the domain of $\psi$ is
unknown. So we define it to be a traveling wave of an associated map
$\tilde{Q}$. However, $\tilde{Q}$ has much weaker compactness than
$Q$. To overcome this difficulty, we establish a variant of Helly's
theorem for monotone functions from $\R$ to $\mc{X}$ in the
Appendix, which is also of its own interest. This discovery also
enables us to study monotone semiflows in a periodic habitat and
with weak compactness, respectively. Further, we can deal with the
case (IV) by the similar idea as in the case (III) because now
traveling waves in the case (II) are defined on $\R$. Traveling
waves for a time-periodic system can be obtained with the help of
the discrete-time semiflow generated by the associated Poincar\'e
map. Motivated by the discussions in \cite[Section 5]{LiangZhaoJFA},
we can regard a semiflow in a periodic habitat as a conjugate
semiflow in a homogeneous discrete habitat, and hence, we can employ
the arguments for the cases (II) and (IV) to establish the existence
of spatially periodic bistable traveling waves.

The rest of this paper is organized as follows. In section 2, we
present our main assumptions. Section 3 is focused on discrete-time,
continuous-time, and time-periodic compact semiflows on some subsets
of $\mc{C}$. In section 4, we extend our results to compact
semiflows in a periodic habitat. In section 5, we further
investigate semiflows with weak compactness. In section 6, we apply
the abstract results to four classes of evolution systems: a
time-periodic reaction-diffusion system, a parabolic system in a
cylinder, a parabolic equation with periodic diffusion, and a
time-delayed reaction-diffusion equation. A short appendix section
completes the paper.

\section{Notations and assumptions}

Throughout this paper, we assume that $\mc{X}$ is an ordered Banach
space with the norm $\|\cdot\|_{\mc{X}}$ and the cone $\mc{X}^+$.
Further, we assume that $\mc{X}$ is also a vector lattice with the
following monotonicity condition:
\begin{equation*}
|x|_{\mc{X}}\le |y|_{\mc{X}} \Rightarrow \|x\|_{\mc{X}}\le
\|y\|_{\mc{X}},
\end{equation*}
where $|z|_{\mc{X}}:=\sup\{z,-z\}$. Such a Banach space is called a
Banach lattice. We use $C(M, \R^d)$ to denote the set of all
continuous functions from the compact metric space $M$ to the
$d$-dimensional Euclidean space $\R^d$. We equip $C(M, \R^d)$ with
the maximum norm and the standard cone consisting of all nonnegative
functions. Then $C(M, \R^d)$ is a special Banach lattice, which will
be used in this paper. For more general information about Banach
lattices, we refer to the book \cite{Schaefer}.

Let the spatial habitat $\mc{H}$ be the real line $\mathbb{R}$ or
the lattice
\begin{equation*}
r\mathbb{Z}:=\{\cdot\cdot\cdot, -2r,-r,0,r,2r,\cdot\cdot\cdot\}
\end{equation*}
for some positive number $r$. For simplicity, we let $r=1$. We say a
function $\phi:\mc{H}\to \mc{X}$ is bounded if the set
$\{\|\phi(x)\|_{\mc{X}}:x\in\mc{H}\}\subset \mc{X}$ is bounded.
Throughout this paper, we always use $\mc{B}$ to denote the set of
all bounded functions from $\R$ to $\mc{X}$, and $\mc{C}$ to denote
the set of all bounded and continuous functions from $\mc{H}$ to
$\mc{X}$. Moreover, any element in $\mc{X}$ can be regarded as a
constant function in $\mc{B}$ and $\mc{C}$.

In this paper, we equip $\mc{C}$ with the compact open topology,
that is, a sequence $\phi_n$ converges to $\phi$ in $\mc{C}$ if and
only if  $\phi_n(x)$ converges to $\phi(x)$ in $\mc{X}$ uniformly
for $x$ in any bounded subset of $\mc{H}$. The following norm on
$\mc{C}$ can induce such topology:
\begin{equation}\label{NormDef}
\|\phi\|_{\mc{C}}=\sum_{k=1}^\infty \f{\max_{|x|\le
k}\|\phi(x)\|_{\mc{X}}}{2^k}, \quad \forall \phi\in\mc{C}.
\end{equation}
Clearly, if $\mc{H}=\mathbb{Z}$, then $\phi_n\to \phi$ with respect
to the compact open topology if and only if $\phi_n(x)\to \phi(x)$
for every $x\in \mathbb{Z}$.

We assume that $Int(\mc{X}^+)$ is not empty. For any $u,v\in
\mc{X}$, we write $u\ge v$ provided $u-v\in \mc{X}^+$, $u>v$
provided $u\ge v$ but $u\neq v$, and $u\gg v$ provided $u-v\in Int
(\mc{X}^+)$. A set $E\subset \mc{X}$ is said to be totally unordered
if any two elements (if exist) are unordered. For any
$\phi,\psi\in\mc{C}$, we write $\phi\ge \psi$ provided $\phi(x)\ge
\psi(x)$ for all $x\in\mc{H}$, $\phi>\psi$ provided $\phi\ge \psi$
but $\phi\neq \psi$, and $\phi\gg \psi$ provided $\phi(x)\gg
\psi(x)$ for all $x\in\mc{H}$. For any $\gamma\in \mc{X}$ with
$\gamma>0$, we define $\mc{X}_\gamma:=\{u\in \mc{X}: \gamma\ge u\ge
0\}$, $\mc{C}_\gamma:=\{\phi\in\mc{C}:\gamma\ge \phi\ge 0\}$ and
$\mc{B}_\gamma:=\{\phi\in\mc{B}:\gamma\ge \phi\ge 0\}$. For any
$\phi,\psi\in\mc{C}$, we write the interval $[\phi,\psi]_{\mc{C}}$
to denote the set $\{w\in \mc{C}: \phi\le w\le \psi\}$,
$[[\phi,\psi]]_{\mc{C}}$ to denote the set $\{w\in \mc{C}: \phi\ll
w\ll \psi\}$, and similarly, we can write the intervals
$[\phi,\psi]]_{\mc{C}}$ and $[[\phi,\psi]_{\mc{C}}$. And for any
$u\le v$ in $\mc{X}$, we can write the intervals
$[u,v]_{\mc{X}},[[u,v]]_{\mc{X}},[[u,v]_{\mc{X}}$ and
$[u,v]]_{\mc{X}}$ in a similar way.

Let $\beta\in Int(\mc{X}^+)$ and $Q$ be a map from $\mc{C}_\beta$ to
$\mc{C}_\beta$. Let $E$ be the set of all fixed points of $Q$
restricted on $\mc{X}_\beta$.

\begin{definition}\label{DefStabFix}
For the map $Q:\mc{X}_\beta\to \mc{X}_\beta$, a fixed point $\alpha
\in E$ is said to be strongly stable from below if there exist a
number $\delta>0$ and a unit vector $e\in Int (\mc{X}^+)$ such that
\begin{equation}\label{Eq:DefinitionStability}
Q[\alpha-\eta e]\gg \alpha-\eta e\quad \text{for any}\quad \eta\in
(0,\delta].
\end{equation}
Strong instability from below is defined by reversing the inequality
\eqref{Eq:DefinitionStability}. Similarly, we can define strong
stability (instability) from above.
\end{definition}

Given $y\in\mc{H}$, define the translation operator $T_y$ on
$\mc{B}$ by $T_y[\phi](x)=\phi(x-y)$. Assume that $0$ and $\beta$
are in $E$. We impose the following hypothesis on $Q$:
\begin{enumerate}
\item[(A1)]({\it Translation Invariance}) $T_y\circ Q [\phi]=Q\circ T_y [\phi],
\forall \phi\in\mc{C}_\beta,y\in\mathcal{H}$.
\item[(A2)]({\it
Continuity}) $Q:\mc{C}_\beta\to \mc{C}_\beta$ is continuous with
respect to the compact open topology.
\item[(A3)]({\it Monotonicity}) $Q$ is order preserving in the sense that $Q[\phi]\ge
Q[\psi]$ whenever $\phi\ge \psi$ in $\mc{C}_\beta$.
\item[(A4)]({\it Compactness})
$Q:\mc{C}_\beta\to \mc{C}_\beta$ is compact with respect to the
compact open topology.
\item[(A5)]({\it
Bistability}) Two fixed points $0$ and $\beta$ are strongly stable
from above and below, respectively, for the map $Q:\mc{X}_\beta\to
\mc{X}_\beta$, and the set $E\setminus\{0,\beta\}$ is totally
unordered.
\end{enumerate}

Note that the above bistability assumption is imposed on the
spatially homogeneous map $Q:\mc{X}_\beta\to \mc{X}_\beta$. We allow
the existence of other fixed points on the boundary of
$\mc{X}_\beta$ so that the theory is applicable to two species
competitive evolution systems. The non-ordering property of
$E\setminus\{0,\beta\}$ can be obtained by the strong instability of
all fixed points in this set if the semiflow is eventually strongly
monotone. More precisely, a sufficient condition for hypothesis (A5)
to hold is:
\begin{enumerate}
\item[(A5$'$)]({\it
Bistability}) $Q:\mc{X}_\beta\to \mc{X}_\beta$ is eventually
strongly monotone in the sense that there exists
$m_1\in\mathbb{Z}_+$ such that $Q^{m}[u]\gg Q^{m}[v]$ for all $m\ge
m_1$ whenever $u> v$ in $\mc{X}_\beta$. Further, for the map
$Q:\mc{X}_\beta\to \mc{X}_\beta$, two fixed points $0$ and $\beta$
are strongly stable from above and below, respectively, and each
$\alpha\in E\setminus\{0,\beta\}$ (if exists) is strongly unstable
from both below and above.
\end{enumerate}

The following figures illustrate the bistability structures in (A5)
and (A5$'$).
\begin{figure}[h]
\centering \includegraphics[width = 0.48\textwidth,height
=0.20\textwidth]{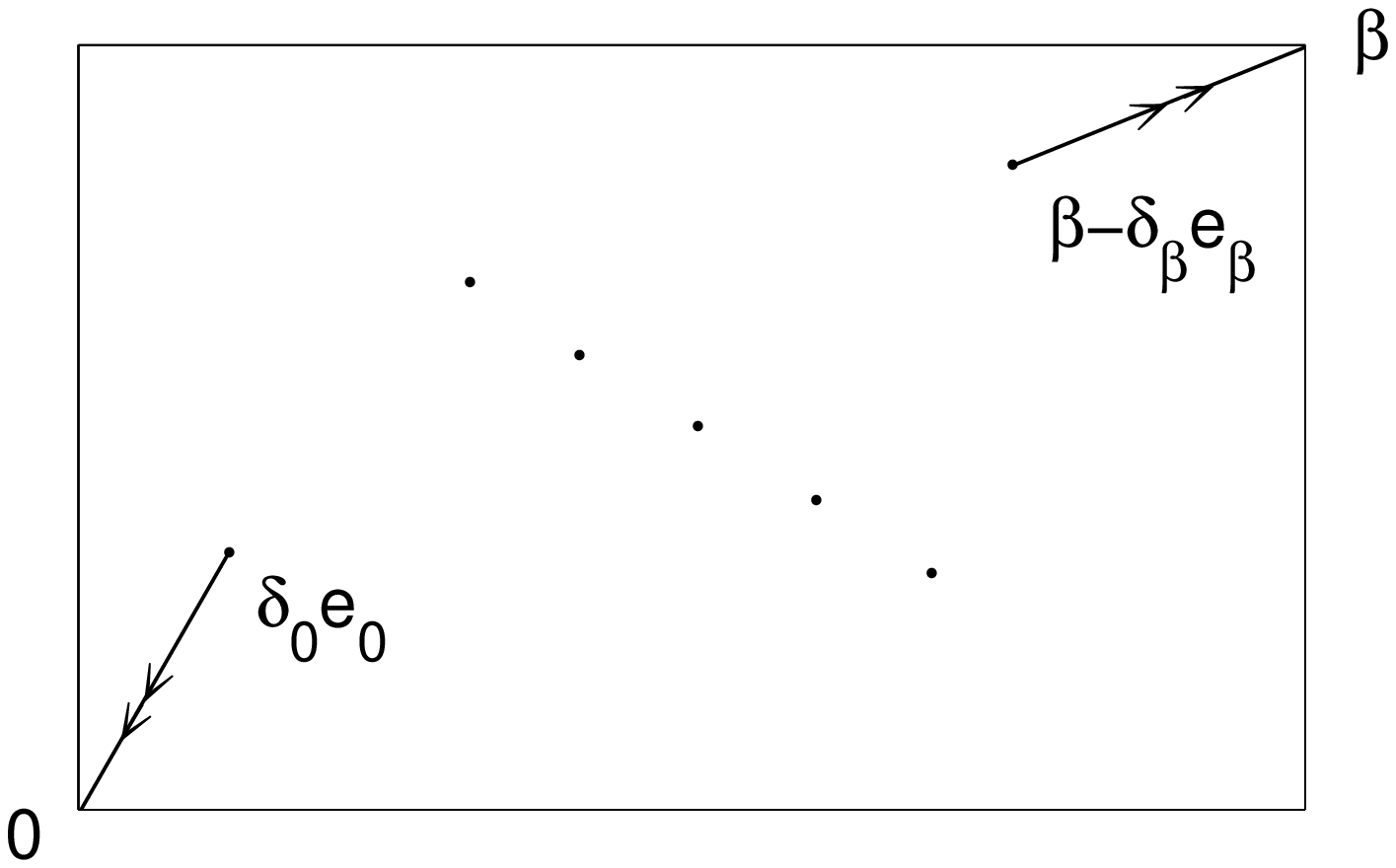}
\includegraphics[width = 0.49\textwidth,height
=0.20\textwidth]{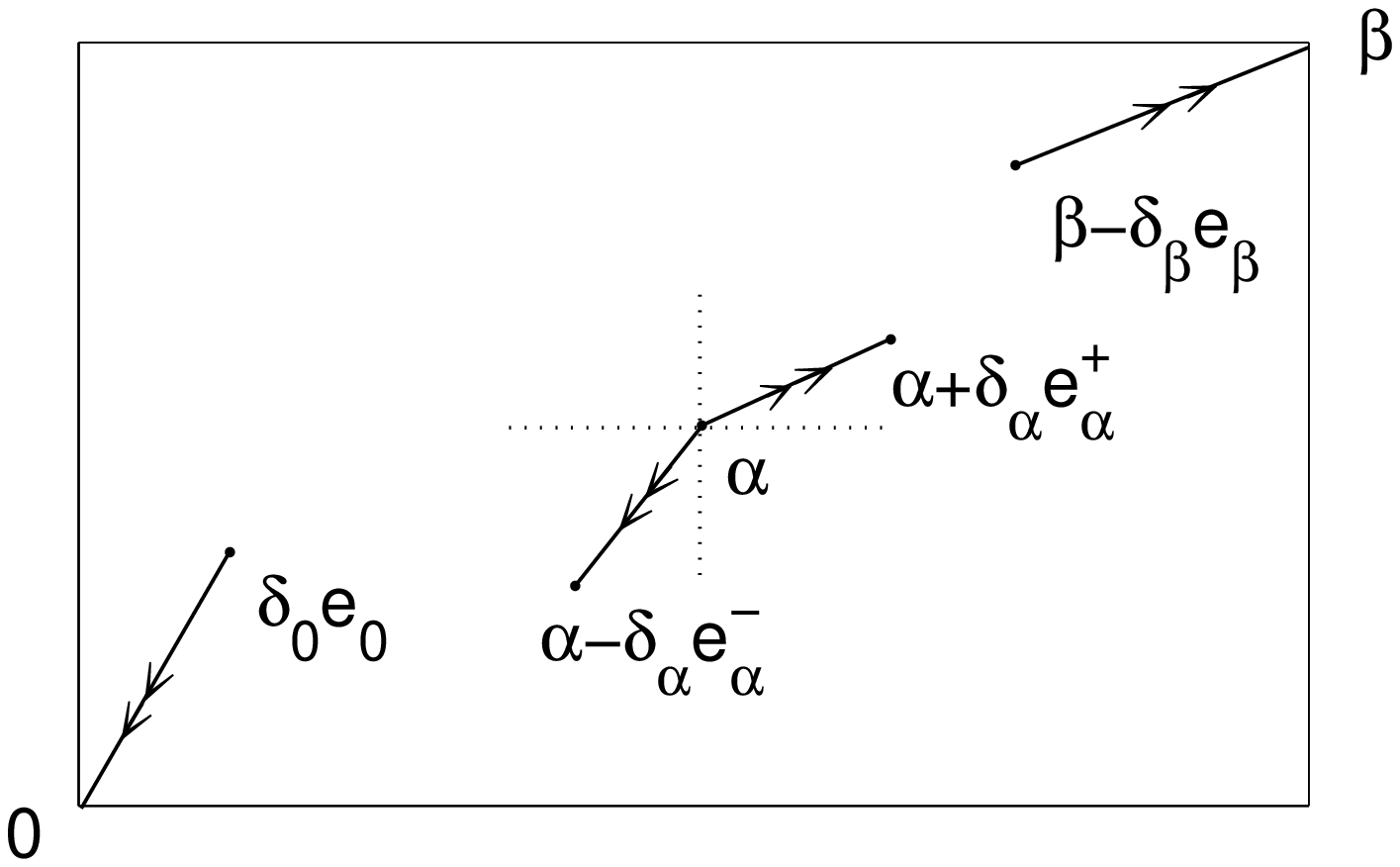} \caption{{\footnotesize (Left) The set $E$
satisfying (A5).\quad (Right) The set $E$ satisfying
(A5$'$).}}\label{fig1}
\end{figure}
Next we show that the assumption (A5$'$) implies (A5). In
applications, however, one may find other weaker sufficient
conditions than (A5$'$) for (A5) to hold.

\begin{proposition}\label{BistabilityMap}
If (A5 $'$) holds, then for any $\alpha_1,\alpha_2\in
E\setminus\{0,\beta\}$, we have $\alpha_1\not<\alpha_2$ and
$\alpha_2\not<\alpha_1$.
\end{proposition}
\begin{proof}
Without loss of generality, we only show $\alpha_1\not< \alpha_2$.
Assume, for the sake of contradiction, that $\alpha_1<\alpha_2$.
Then $\alpha_1=Q^{m_1}[\alpha_1]\ll Q^{m_1}[\alpha_2]=\alpha_2$.
Since $\alpha_1$ is strongly unstable from above, there exists
$\delta_{\alpha_1}>0$ and $e_{\alpha_1}\in Int(\mc{X}^+)$ such that
$u_0:=\alpha_1+\delta_{\alpha_1}e_{\alpha_1}\in[[\alpha_1,\alpha_2]]_{\mc{X}}$
and $Q[u_0]\gg u_0$. Define the recursion $u_{n+1}=Q[u_n],n\ge 0$.
Then $u_n$ is convergent to some $\alpha\in \mc{X}$ with
$\alpha_1\ll \alpha\le \alpha_2$ due to hypothesis (A4). By the
eventual strong monotonicity of $Q$, we see that
\begin{equation*}
u_n=Q^{m_1}[u_{n-m_1}]\ll Q^{m_1}[u_{n+1-m_1}]=u_{n+1}\ll
Q^{m_1}[\alpha]=\alpha, \forall n\ge m_1.
\end{equation*}
Since $\alpha$ is strongly unstable from below, we can find
$\delta_{\alpha}>0$ and $e_{\alpha}\in Int \mc{X}^+$ such that
$Q[\alpha-\delta e_\alpha]\ll \alpha-\delta e_\alpha,\forall
\delta\in(0,\delta_\alpha]$. Choose $n_1\ge m_1$ such that
$u_{n_1}\ge \alpha-\delta_\alpha e_\alpha$. Define
$\eta:=\sup\{\delta\in(0,\delta_\alpha]: u_{n_1}\le \alpha-\delta
e_\alpha\}$. Thus, $u_{n_1}\not\ll \alpha-\eta \delta_\alpha$. On
the other hand, we have
\begin{equation*}
u_{n_1}\ll u_{n_1+1}=Q[u_{n_1}]\le Q[\alpha-\eta e_\alpha]\ll
\alpha-\eta e_\alpha,
\end{equation*}
a contradiction.
\end{proof}

Due to assumption (A5), a bistable system $\{Q^n\}_{n\ge 0}$ can be
regarded as the union of two monostable systems. More precisely,
assuming that $\alpha\in E\setminus \{0,\beta\}$, we have two
monostable sub-systems: $\{Q^n\}_{n\ge 0}$ restricted on
$[0,\alpha]_{\mc{C}}$ and $[\alpha,\beta]_{\mc{C}}$, respectively.
With this in mind, next we construct an initial function
$\phi_\alpha^-$ so that we can define the leftward asymptotic speed
of propagation of $\phi_\alpha^-$, and hence, present our last
assumption.

Note that in (A5) we do not require $\alpha\gg 0$ or $\alpha\ll
\beta$. But (A5) is sufficient to guarantee that $\alpha$ and
$\beta$ can be separated by two neighborhoods in
$[\alpha,\beta]_{\mc{X}}$, and a similar claim is valid for $0$ and
$\alpha$ (see Lemma 3.1). In view of assumption (A5), we can find a
positive number $\delta_\beta>0$ and a unit vector $e_\beta\in
Int(\mc{X}^+)$ such that
\begin{equation}\label{deltabeta}
Q[\beta-\eta e_\beta]\gg \beta-\eta e_\beta,\quad \forall
\eta\in(0,\delta_\beta].
\end{equation}
Define
\begin{equation*}
\theta^-:=\sup\left\{\theta\in [0,1]:\theta
\alpha+(1-\theta)\beta\in [\beta-\delta_\beta
e_\beta,\beta]_\mc{X}\right\}.
\end{equation*}
Let
\begin{equation}\label{defvaalpha-}
v_\alpha^-=\theta^-\alpha+(1-\theta^-)\beta.
\end{equation}
Choose a nondecreasing initial function $\phi_\alpha^-\in
\mc{C}_\beta$ with the property that
\begin{equation}\label{propertiesAlpha-}
\phi_\alpha^-(x)=\alpha,\forall x\le -1, \quad \text{and}\quad
\phi_\alpha^-(x)=v_\alpha^-, \forall x\ge 0.
\end{equation}
It then follows from assumptions (A1)-(A2) and (A5) that
\begin{equation*}
\lim_{x\to
+\infty}Q[\phi_\alpha^-](x)=Q[\phi_\alpha^-(+\infty)](0)=Q[v_\alpha^-]\ge
Q[\beta-\delta_\beta e_\beta]\gg  \beta-\delta_\beta e_\beta,
\end{equation*}
and hence, there exits $\sigma>0$ such that
\begin{equation*}
Q[\phi_\alpha^-](x)\gg \beta-\delta_\beta e_\beta, \forall x\ge
\sigma-1.
\end{equation*}
Define a sequence $a_{n,\sigma}$ of points in $\mc{X}$ as follows:
\begin{equation*}
a_{n,\sigma}=Q^n[\phi_\alpha^-](\sigma n),\quad n\ge 1.
\end{equation*}
Then we have
\begin{equation*}
a_{2,\sigma}=Q^2[\phi_\alpha^-](2\sigma)=Q[Q[\phi_\alpha^-]
(\cdot+\sigma)](\sigma)\ge Q[\phi_\alpha^-](\sigma)=a_{1,\sigma}.
\end{equation*}
By induction, we see that $a_{n,\sigma}$ is nondecreasing in $n$.
Thus, assumption (A4) implies that $a_{n,\sigma}$ tends to a fixed
point $e$ with $e\ge a_{1,\sigma} \gg \beta -\delta_\beta e_\beta$.
Therefore, $e=\beta$.

By the above observation, we have
\begin{equation*}
\beta\ge \lim_{n\to \infty,x\ge \sigma n}Q^n[\phi_\alpha^-](x)\ge
\lim_{n\to \infty} Q^n[\phi_\alpha^-](\sigma n)= \lim_{n\to \infty}
a_{n,\sigma}=\beta,
\end{equation*}
and hence,
\begin{equation*}
(-\infty,-\sigma]\subset
\Lambda(\phi_\alpha^-):=\left\{c\in\R:\lim_{n\to \infty,x\ge
-cn}Q^n[\phi_\alpha^-](x)=\beta\right\}.
\end{equation*}
Define
\begin{equation}\label{defcmiunsstar}
c_-^*(\alpha,\beta):=\sup \Lambda (\phi_\alpha^-).
\end{equation}
Clearly, $c_-^*(\alpha,\beta)\in [-\sigma,+\infty]$ and
$(-\infty,c_-^*(\alpha,\beta))\subset\Lambda (\phi_\alpha^-)$. We
further claim that $c_-^*(\alpha,\beta)$ is independent of the
choice of $\phi_\alpha^-$ as long as $\phi_\alpha^-$ has the
property \eqref{propertiesAlpha-}. Indeed, for any given $\phi$ with
the property \eqref{propertiesAlpha-}, we have
\begin{equation*}
\phi_\alpha^-(x-1)\le \phi(x)\le \phi_\alpha^-(x+1),\quad\forall
x\in \mc{H}.
\end{equation*}
It then follows that for any $c\in \Lambda (\phi_\alpha^-)$ and
$\epsilon>0$,
\begin{eqnarray*}
\beta &&=\lim_{n\to \infty,x\ge -cn}Q^n[\phi_\alpha^-](x)=\lim_{n\to
\infty,x\ge -(c-\epsilon) n}Q^n[\phi_\alpha^-](x-1)
\nonumber\\
&&\le \lim_{n\to \infty,x\ge -(c-\epsilon) n}Q^n[\phi](x) \le
\lim_{n\to \infty,x\ge -(c-\epsilon)
n}Q^n[\phi_\alpha^-](x+1)\nonumber\\
&& = \lim_{n\to \infty,x\ge -c n}Q^n[\phi_\alpha^-](x)=\beta,
\end{eqnarray*}
which implies that $c-\epsilon \in \Lambda (\phi)$ and hence, $\sup
\Lambda (\phi_\alpha^-)=\sup \Lambda (\phi)$. For convenience, we
may call $c_-^*(\alpha,\beta)$ as the leftward asymptotic speed of
propagation of $\phi_\alpha^-$.

Following the above procedure, we can find $\delta_0>0,e_0\in
Int(\mc{X}^+)$ such that
\begin{equation}\label{delta0}
Q[\eta e_0]\ll \eta e_0, \forall \eta\in (0,\delta_0].
\end{equation}
Here we emphasis that $\delta_0,e_0$ above and
$\delta_\beta,e_\beta$ will play a vital role in the whole paper
because they describe the local stability of fixed points $0$ and
$\beta$. Similarly, we can define $\theta^+:=\sup\{\theta\in[0,1]:
\theta \alpha\in [0,\delta_0e_0]_\mc{X}\}$ and $v_\alpha^+:=\theta^+
\alpha$. Let $\phi_\alpha^+\in \mc{C}_\beta$ be a nondecreasing
initial function with the property that
\begin{equation*}
\phi_\alpha^+(x)=\alpha,\forall x\geq  1, \quad \text{and}\quad
\phi_\alpha^+(x)=v_\alpha^+, \forall x\leq 0.
\end{equation*}
Due to the same reason, we can define the number
\begin{equation}\label{Defcplusstar}
c_+^*(0,\alpha):=\sup\left\{c\in\R:\lim_{n\to \infty,x\le
cn}Q^n[\phi_\alpha^+](x)=0\right\},
\end{equation}
which is called the rightward asymptotic speed of propagation of
$\phi_\alpha^+$. As showed above, these two speeds are bounded
below, but may be plus infinity. To better understand these two
spreading speeds, we use Figure \ref{fig2}\,(Left) to explain them.

\begin{figure}[h]
\centering \includegraphics[width = 0.49\textwidth]{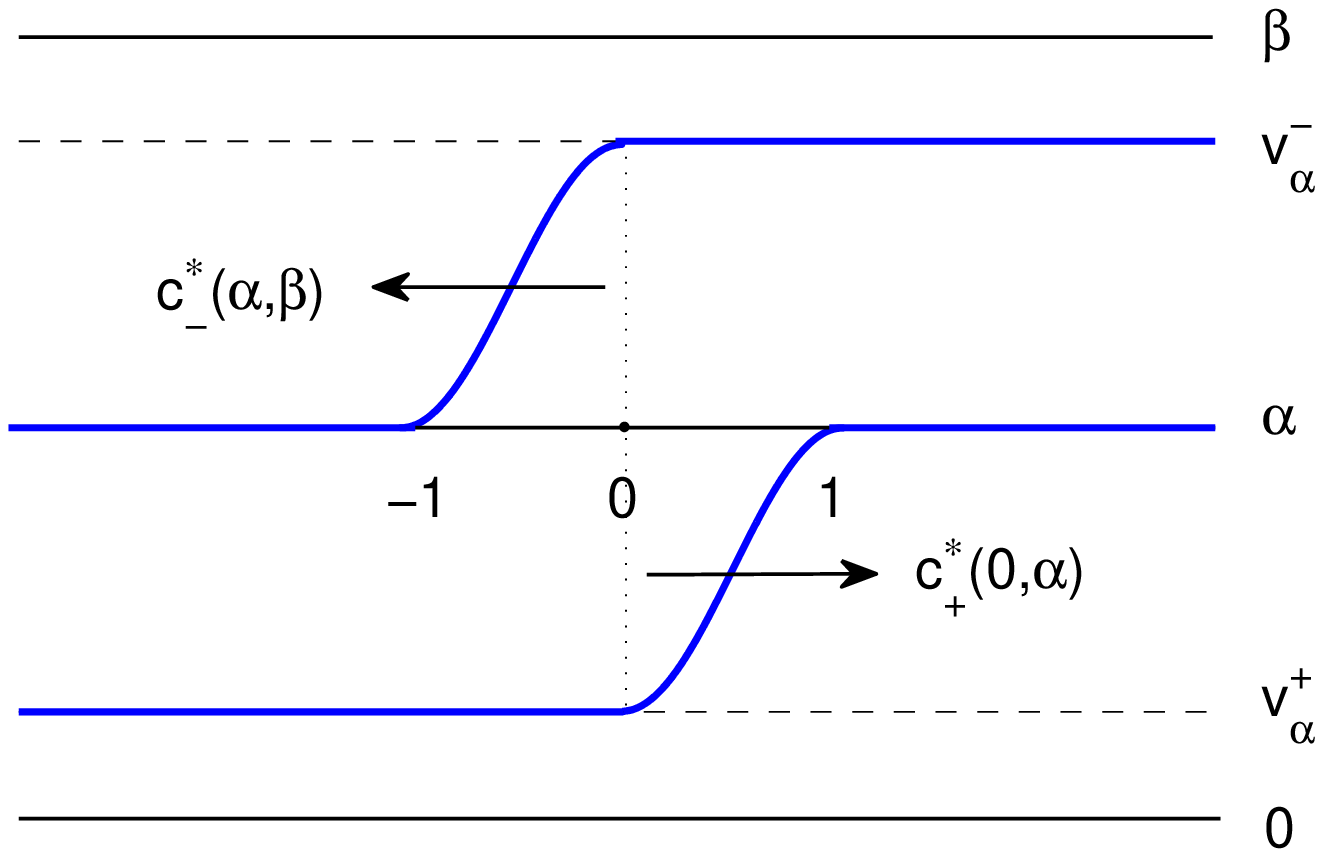}
\includegraphics[width = 0.49\textwidth]{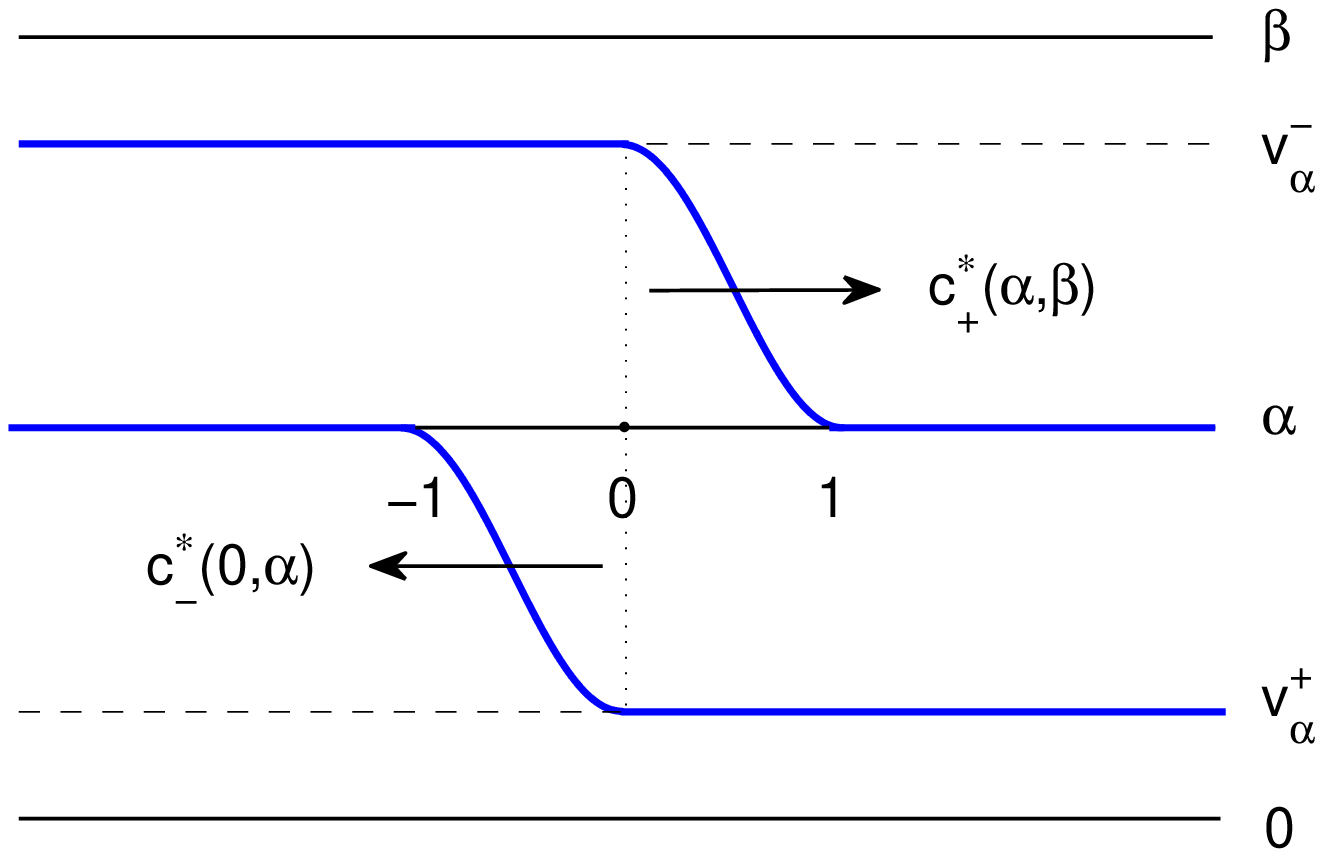}
\caption{{\footnotesize (Left) $c_-^*(\alpha,\beta)$ and
$c_+^*(0,\alpha)$.\quad (Right) $c_+^*(\alpha,\beta)$ and
$c_-^*(0,\alpha)$.}}\label{fig2}
\end{figure}

Now we are ready to state our last assumption on $Q$:
\begin{enumerate}
\item[(A6)] ({\it Counter-propagation}) For each $\alpha\in E\setminus\{0,\beta\}$,
$c_-^*(\alpha,\beta)+c_+^*(0,\alpha)>0$.
\end{enumerate}

Assumption (A6) assures that two initial functions in the left hand
side of Figure \ref{fig2} will eventually propagate oppositely
although one of these two speeds may be negative. It is interesting
to note that assumption (A6) is nearly necessary for the propagation
of a bistable traveling wave. Indeed, if a monotone evolution system
admits a bistable traveling wave, then it is usually unique (up to
translation) and globally attractive (see, e.g., Remark
\ref{stabilityremark}). This implies that the solution starting from
the initial data $\f{1}{2}(\phi_\alpha^+ +\phi_\alpha^-)$ converges
to a phase shift of the bistable wave. If
$c_-^*(\alpha,\beta)+c_+^*(0,\alpha)<0$, then the comparison
principle would force the solutions starting from $\phi_\alpha^\pm$
to split the bistable wave.

Comparing with the definition of spreading speeds (short for
asymptotic speeds of spread/propagation) for monostable semiflows
(see, e.g., \cite{AroWei2,LiangZhaoJFA}), one can find that the
leftward spreading speed of the monostable subsystem $\{Q^n\}_{n\ge
0}$ restricted on $[\alpha,\beta]_\mc{C}$ is shared by a large class
of initial functions, and in many applications, it equals
$c_-^*(\alpha,\beta)$. A similar observation holds for
$c_+^*(0,\alpha)$. Thus, for a specific bistable system, the
assumption (A6) can be verified by using the properties of spreading
speeds for monostable subsystems.

\begin{remark}
If we consider the non-increasing traveling waves, then we can
similarly define the numbers $c_+^*(\alpha,\beta)$ and
$c_-^*(0,\alpha)$ (See Figure \ref{fig2}(Right)). As such, (A6)
should be stated as $c_+^*(\alpha,\beta)+c_-^*(0,\alpha)>0$.
\end{remark}

\section{Semiflows in a homogeneous habitat}
We say a habitat is homogeneous for the semiflow
$\{Q_t\}_{t\in\mc{T}}$ on a metric space $\mc{E}\subset \mc{C}$ if
\begin{equation*}
Q_t[\phi](x-y)=Q_t[\phi(\cdot-y)](x),\quad \forall
\phi\in\mc{E},x,y\in\mc{H}, t\in \mc{T}.
\end{equation*}
In this section, we will establish the existence of bistable
traveling waves for the semiflow $\{Q_t\}_{t\in\mc{T}}$ on $\mc{E}$
in the following order: discrete-time semiflows in a continuous
habitat, discrete-time semiflows in a discrete habitat,
time-periodic semiflows, continuous-time semiflows in a continuous
habitat, and continuous-time semiflows in a discrete habitat.

\subsection{Discrete-time semfilows in a continuous habitat} In this
case, time $\mc{T}$ is discrete and habitat $\mc{H}$ is continuous:
$\mc{T}=\Z^+$ and $\mc{H}=\R$.  For convenience, we use $Q$ to
denote $Q_1$, and consider the semiflow $\{Q^n\}_{n\ge 0}$, where
$Q^n$ is the $n$-th iteration of $Q$.

\begin{definition}\label{DefDisCont}
$\psi(x+cn)$ with $\psi\in \mc{C}$ is said to be a traveling wave
with speed $c\in\mathbb{R}$ of the discrete semiflow $\{Q^n\}_{n\ge
0}$ if $Q^n[\psi](x)=\psi(x+cn),\forall x\in\R, n\ge 0$. We say that
$\psi$ connects $0$ to $\beta$ if $\psi(-\infty):=\lim_{x\to
-\infty}\psi(x)=0$ and $\psi(+\infty):=\lim_{x\to
+\infty}\psi(x)=\beta$.
\end{definition}

We first show that $0$ and $\beta$ are two isolated fixed points of
$Q$ in $\mc{X}_\beta$ if (A5) holds.

\begin{lemma}\label{IslatedEquilibria}
Let $\delta_0,e_0$ and $\delta_\beta, e_\beta$ be chosen such that
\eqref{delta0} and \eqref{deltabeta} hold, respectively. Then $E\cap
\mc{X}_{\delta_0e_0}=\{0\}$ and $E\cap [\beta-\delta_\beta
e_\beta,\beta]_\mc{X}=\{\beta\}$.
\end{lemma}
\begin{proof}
Assume, for the sake of contradiction, that $0\neq\alpha\in E\cap
\mc{X}_{\delta_0e_0}$. Define the number
$\bar{\delta}\in(0,\delta_0]$ by
\begin{equation*}
\bar{\delta}:=\inf\{\delta\in(0,\delta_0]: \alpha\in [0,\delta
e_0]_\mc{X}\}.
\end{equation*}
Then it follows that $\alpha \le \bar{\delta} e_0$ but
$\alpha\not\in [0,\bar{\delta}e_0]]_\mc{X}$. However, by the
monotonicity of $Q$ and the fact that $0$ is strongly stable, we
have
\begin{equation*}
\alpha=Q[\alpha]\le Q[\bar{\delta}e_0]\ll \bar{\delta}e_0.
\end{equation*}
This contradicts $\alpha\not\in [0,\bar{\delta}e_0]]_\mc{X}$. And
hence, $E\cap \mc{X}_{\delta_0e_0}=\{0\}$. Similarly, we have $E\cap
[\beta-\delta_\beta e_\beta,\beta]_\mc{X}=\{\beta\}$.
\end{proof}

Choose $\delta>0$ such that
\begin{equation}\label{delta}
\delta<\min\{\delta_0,\delta_\beta\}\quad \text{and} \quad \delta
e_0\ll\beta-\delta e_\beta.
\end{equation}
Assume that $\un{\psi}$ and $\bar{\psi}$ are two nondecreasing
functions in $C(\mathbb{R},\mc{X}_\beta)$ with the properties that
\begin{eqnarray*}
\un{\psi}(x)=
\begin{cases}
0,& x\le 0\\
\beta-\delta e_\beta, & x\ge 1
\end{cases}
\quad \text{and} \quad \bar{\psi}(x)=
\begin{cases}
\delta e_0, & x\le -1\\
\beta, & x\ge 0.
\end{cases}
\end{eqnarray*}
Clearly $\un{\psi}\le \bar{\psi}$. And we have the following
observation.

\begin{lemma}\label{UpperLowerSolu1}
Assume that $Q$ satisfies (A1)-(A3) and (A5). Then there exists a
positive rational number $\bar{c}$ such that for any $c\ge \bar{c}$,
we have
\begin{equation*}
Q[\un{\psi}](x)\ge \un{\psi}(x-c)\quad \text{and}\quad
Q[\bar{\psi}](x)\le \bar{\psi}(x+c)\quad \text{for any $x\in\R $}.
\end{equation*}
\end{lemma}
\begin{proof}
Assume that $x_n \to +\infty$ be an increasing sequence in
$\mathbb{R}$. Then the sequence $\psi_n:=\un{\psi}(\cdot+x_n)$
converges to $\beta-\delta e_\beta$ in $\mc{C}_\beta$ since
$\un{\psi}(x)=\beta-\delta e_\beta,\forall x\ge 1$. It then follows
from (A1)-(A2) and (A5) that
\begin{equation*}
Q[\un{\psi}](+\infty)=\lim_{n\to \infty}
Q[\un{\psi}](x_n)=\lim_{n\to
\infty}Q[\un{\psi}(\cdot+x_n)](0)=Q[\beta-\delta e_\beta]\gg
\beta-\delta e_\beta.
\end{equation*}
Therefore, there exists a positive $y_0\in\mathbb{R}$ such that
$Q[\un{\psi}](y_0)\ge \beta-\delta e_\beta$. Note that
$Q[\un{\psi}](x)$ is nondecreasing in $x$. Then for any $c\ge y_0$
we have
\begin{equation*}
Q[\un{\psi}](x)\ge Q[\un{\psi}](y_0)\ge \beta-\delta e_\beta \ge
\un{\psi}(x-c) ,\forall x\ge y_0
\end{equation*}
and
\begin{equation*}
Q[\un{\psi}](x)\ge 0=\un{\psi}(0) \ge \un{\psi}(x-y_0)\ge
\un{\psi}(x-c),\forall x< y_0,
\end{equation*}
which means $Q[\un{\psi}](x) \ge \un{\psi}(x-c),\forall c\ge y_0$.
Similarly, we have
\begin{equation*}
Q[\bar{\psi}](-\infty)=\lim_{n\to \infty}
Q[\un{\psi}](-x_n)=\lim_{n\to
\infty}Q[\bar{\psi}(\cdot-x_n)](0)=Q[\delta e_0]\ll \delta
e_0=\bar{\psi}(-\infty),
\end{equation*}
and hence, there exists $z_0>0$ such that $Q[\bar{\psi}](x)\le
\bar{\psi}(x+c),\forall c\ge z_0$. Choosing $\bar{c}=
\max\{y_0,z_0\}$, we complete the proof.
\end{proof}

Let $\kappa_n:=\f{n+\bar{c}}{n}$. Clearly, $\kappa_n,\forall n\ge
1$, is a rational number. For any $\xi\in\R$, define the map
$A_\xi:\mc{B}\to\mc{B}$ by $A_\xi[\phi](x)=\phi(\xi x),\forall
x\in\mathbb{R}$. Define $\un{\psi}_n,\bar{\psi}_n\in \mc{C}_\beta$
by
\begin{equation*}
\un{\psi}_n(x)=\un{\psi}\left(x-(n+\bar{c})\right)\quad
\text{and}\quad
\bar{\psi}_n(x)=\bar{\psi}\left(x+(n+\bar{c})\right).
\end{equation*}

\begin{lemma}\label{FixedPoint1}
Assume that $Q$ satisfies (A1)-(A5). Then for each $n\in\mathbb{N}$,
$G_n:=Q\circ A_{\kappa_n}$ has a fixed point $\phi_n$ in
$\mc{C}_\beta$ such that $\phi_n$ is nondecreasing and
$\un{\psi}_n\le \phi_n\le \bar{\psi}_n$.
\end{lemma}
\begin{proof}
We first show that $\un{\psi}_n\le G_n[\un{\psi}_n]$. Indeed, when
$x< n$ we have
\begin{equation*}
\un{\psi}_n(x+\bar{c})\le \un{\psi}_n(n+\bar{c})=\un{\psi}(0)=0\le
A_{\kappa_n}[\un{\psi}_n](x);
\end{equation*}
when $x\ge n$ we have
\begin{equation*}
A_{\kappa_n}[\un{\psi}_n](x)=\un{\psi}_n(\kappa_n
x)=\un{\psi}_n(x+\f{\bar{c}}{n}x)\ge \un{\psi}_n(x+\bar{c}),
\end{equation*}
and hence, $\un{\psi}_n(x+\bar{c})\le A_{\kappa_n}[\un{\psi}_n](x)$
for all $x\in\R$. Consequently, by the monotonicity of $Q$ and
$\un{\psi}(x)\le Q[\un{\psi}](x+\bar{c})$(see Lemma
\ref{UpperLowerSolu1}) we obtain
\begin{equation*}
\un{\psi}_n(x)\le Q[\un{\psi}_n](x+\bar{c})\le Q\circ
A_{\kappa_n}[\un{\psi}_n ](x)=G_n[\un{\psi}_n ](x).
\end{equation*}
Similarly, we have $\bar{\psi}_n\ge G_n[\bar{\psi}_n]$. It then
follows that
\begin{equation}\label{LimitFunctionIncreasing}
\un{\psi}_n\le G_n^k[\un{\psi}_n ]\le G_n^k[\bar{\psi}_n]\le
\bar{\psi}_n,\quad \forall k\in \mathbb{N}.
\end{equation}
For any $k\ge 1$, we have
\begin{equation}\label{UsingCompactness1}
G_n^k[\un{\psi}_n]=G_n\circ G_n^{k-1}[\un{\psi}_n]\in
G_n[\mc{C}_\beta].
\end{equation}
Since $G_n$ is order preserving and $\un{\psi}_n(x)$ is
nondecreasing in $x$, we know that $G_n^k[\un{\psi}_n](x)$ is
nondecreasing both in $k$ and $x$. Recall that $G_n$ is compact due
to assumption (A4). It then follows that $G_n^k[\un{\psi}_n]$
converges in $\mc{C}_\beta$. Denote the limit by $\phi_n$. By
inequality \eqref{LimitFunctionIncreasing}, we also get
$\un{\psi}_n\le \phi_n\le \bar{\psi}_n$. Moreover, $\phi_n(x)$ is
also nondecreasing due to Proposition \ref{BasicProp}(2). And
obviously,
\begin{equation*}
\phi_n=\lim_{k\to \infty} G_n^{k+1}[\un{\psi}_n]=G_n [\lim_{k\to
\infty}G_n^k[\un{\psi}_n]]=G_n [\phi_n].
\end{equation*}
This completes the proof.
\end{proof}

The following lemma reveals a relation between the wave speeds of
monostable traveling waves in the sub-monostable systems and the
numbers defined in \eqref{defcmiunsstar} and \eqref{Defcplusstar}.
\begin{lemma}\label{MiniWaveSpeedDSF}
Let $c_-^*(\alpha,\beta)$ and $c_+^*(0,\alpha)$ be defined as in
\eqref{defcmiunsstar} and \eqref{Defcplusstar}. Assume that $Q$
satisfies (A3). Then the following statements are valid:
\begin{enumerate}
\item[(1)]If $\psi(x+ct)$ is a monotone traveling wave connecting $\alpha$ to $\beta$
of the discrete semiflow $\{Q^n\}_{n\ge 1}$, then the speed $c\ge
c_-^*(\alpha,\beta)$.
\item[(2)]If $\psi(x+ct)$ is a monotone traveling wave connecting $0$ to $\alpha$
of the discrete semiflow $\{Q^n\}_{n\ge 1}$, then the speed $c\le
-c_+^*(0,\alpha)$.
\end{enumerate}
\end{lemma}
\begin{proof}
We only prove the statement (1) since the proof for (2) is similar.
In view of Lemma \ref{IslatedEquilibria}, we see that $\alpha$ and
$\beta$ can be separated by balls $\mc{N}(\alpha,\gamma)$ and
$\mc{N}(\beta,\gamma)$ with radius $\gamma<\delta_\beta/2$ in the
metric space $[\alpha,\beta]_\mc{X}$. Then $\alpha<u, \forall u\in
\mc{N}(\beta,\gamma)$. We write $u\in \mc{N}(\beta,\gamma)$ as the
form $u=\beta+v$. Recall the definition of $v_\alpha^-$ in
\eqref{defvaalpha-}, it then follows that
\begin{equation*}
v_\alpha^- = \theta^-\alpha+(1-\theta^-)\beta<\theta^- (\beta+v)
+(1-\theta^-)\beta= \beta+\theta^- v,
\end{equation*}
which implies that
\begin{equation*}
v_\alpha^-< w, \quad \forall w\in \mc{N}(\beta,\theta^- \gamma).
\end{equation*}
Since $\psi(-\infty)=\alpha$ and $\psi(+\infty)=\beta$, there must
exist a nondecreasing initial function $\phi_\alpha^-$ with property
\eqref{propertiesAlpha-} such that $\phi_\alpha^- \le \psi$. Assume,
for the sake of contradiction, that $c<c_-^*(\alpha,\beta)$. Choose
a rational number $\f{q}{p}\in(c,c_-^*(\alpha,\beta))$ with $p,q\in
\Z$. It then follows from \eqref{defcmiunsstar} that
\begin{eqnarray*}
\beta &=&\lim_{n\to \infty}Q^{pn}[\phi_\alpha^-](-\f{q}{p}\times
pn)\le \lim_{n\to \infty}Q^{pn}[\psi](-qn)\nonumber\\
&=&\lim_{n\to \infty} \psi(-q n+cpn)=\psi(-\infty)=0,
\end{eqnarray*}
a contradiction. Thus, we have $c\ge c_-^*(\alpha,\beta)$.
\end{proof}

Now we are ready to prove the main result of this subsection.
\begin{theorem}\label{ThDisCont}
Assume that $Q$ satisfies (A1)-(A6). Then there exists
$c\in\mathbb{R}$ such that the discrete semiflow $\{Q^n\}_{n\ge 1}$
admits a non-decreasing traveling wave with speed $c$ and connecting
$0$ to $\beta$ .
\end{theorem}
\begin{proof}
We spend three steps to complete the proof. Firstly, we construct
$\phi_+,\phi_-\in\mc{C}_\beta$, $c_+\le c_-\in\mathbb{R}$ such that
\begin{equation*}
Q[\phi_+](x)=\phi_+(x+c_+)\quad \text{and}\quad
Q[\phi_-](x)=\phi_-(x+c_-)
\end{equation*}
with
\begin{equation*}
\phi_-(0)\in(0,\delta e_0]_\mc{X}\quad \text{and}\quad \phi_+(0)\in
[\beta-\delta e_\beta,\beta)_\mc{X}.
\end{equation*}
Indeed, let $\phi_n$ be obtained in Lemma \ref{FixedPoint1}. Since
$0\ll \bar{\psi}(-1)=\delta e_0\ll \un{\psi}(1)=\beta-\delta
e_\beta\ll \beta$ and $\un{\psi}_n\le \phi_n\le\bar{\psi}_n$, we
have
\begin{equation*}
\bar{\psi}(-1)=\bar{\psi}_n(-1-(n+\bar{c}))
\ge\phi_n(-1-(n+\bar{c}))
\end{equation*}
and
\begin{equation*}
\un{\psi}(1)=\un{\psi}_n(1+(n+\bar{c})) \le\phi_n(1+(n+\bar{c})).
\end{equation*}
Now we define $a_n,b_n\in\mathbb{R}$ as follows:
\begin{equation*}
a_n:=\sup_{x\in\mathbb{R}}\{\phi_n(x)\in [0,\delta
e_0]_{\mc{X}}\},\quad b_n:=\inf_{x\in\mathbb{R}}\{\phi_n(x)\in
[\beta-\delta e_\beta,\beta]_{\mc{X}}\}.
\end{equation*}
It then follows that
\begin{equation*}
-1-(n+\bar{c})\le a_n\le b_n\le 1+(n+\bar{c})
\end{equation*}
and
\begin{equation*}
\phi_n(a_n)\le \delta e_0\le \beta-\delta e_\beta\le \phi_n(b_n).
\end{equation*}
Define $\phi_{-,n}(x):=\phi_n(x+a_n)$ and
$\phi_{+,n}(x):=\phi_n(x+b_n)$. Then
\begin{equation*}
\phi_{-,n}=\phi_n(\cdot+a_n)
=G_n[\phi_n](\cdot+a_n)=Q[\phi_n(\kappa_n\cdot)](\cdot+a_n)=Q[\phi_n(\kappa_n(\cdot+a_n))]\in
Q[\mc{C}_\beta].
\end{equation*}
Similarly, $\phi_{+,n}=Q[\phi_n(\kappa_n(\cdot+b_n))]\in
Q[\mc{C}_\beta]$. Thus, there exists a subindex (still denoted by
$n$), two nondecreasing functions $\phi_-,\phi_+\in \mc{C}_\beta$
and $\xi_-,\xi_+\in[-1,1]$ with $\xi_-\le \xi_+$ such that
\begin{equation*}
\lim_{n\to \infty}\f{a_{n}}{n}=\xi_-,\quad \lim_{n\to
\infty}\f{b_{n}}{n}=\xi_+,\quad \lim_{n\to \infty}\phi_{-,n}=\phi_-
\quad \text{and}\quad \lim_{n\to \infty}\phi_{+,n}=\phi_+.
\end{equation*}
Obviously, $\phi_-(0)=\lim_{n\to \infty}\phi_{n}(a_{n})$ and
$\phi_+(0)=\lim_{n\to \infty}\phi_{n}(b_{n})$. By the definitions of
$a_n$ and $b_n$, we immediately have $\phi_-(0)\neq 0$ and
$\phi_+(0)\neq \beta$, and hence $0<\phi_-(0)\le
\bar{\psi}(-1)=\delta e_0$ and $\beta-\delta e_\beta =
\un{\psi}(1)\le \phi_+(0)<\beta$. Define $c_-:=-\bar{c}\xi_-$ and
$c_+:=-\bar{c}\xi_+$. Obviously, $c_-\ge c_+$ because $\xi_-\le
\xi_+$. Now we want to only prove $Q[\phi_-](x)=\phi_-(x+c_-)$
because the proof of the other one is similar. Note that the
following limit is uniform for $x$ in any bounded subset $M\subset
\mathbb{R}$
\begin{equation*}
\lim_{n\to\infty}\kappa_{n}(x+a_n)-a_n =\lim_{n\to\infty}
\left(x+\bar{c}\cdot\f{x+a_n}{n}\right) =x-c_-.
\end{equation*}
It then follows that for any $\in\mathbb{R}$, we have
\begin{eqnarray}\label{Eq:PhiMinus}
&&\phi_-(x+c_-)\nonumber\\
&&=\lim_{n\to \infty}\phi_{-,n}(x+c_-)=\lim_{n\to
\infty}\phi_n(x+c_-+a_n)=\lim_{n\to \infty}G_n[\phi_n](x+c_-+a_n)\nonumber\\
&&=\lim_{n\to \infty} Q[\phi_n(\kappa_n\cdot)](x+c_-+a_n)=\lim_{n\to
\infty}Q[
\phi_n(\kappa_n(\cdot+a_n))](x+c_-)\nonumber\\
&&=\lim_{n\to \infty}Q[
\phi_{-,n}(\kappa_n(\cdot+a_n)-a_n)](x+c_-)=Q[\phi_-](x),
\end{eqnarray}
where the last equality is obtained from Proposition
\ref{SequenceInC}(2) and the continuity of $Q$.

Secondly, we prove that $\phi_\pm(x)$ obtained in the first step
have the following properties:
\begin{enumerate}
\item[(i)] $\phi_-(-\infty)=0$ and $\phi_+(+\infty)=\beta$;

\item[(ii)] $\phi_-(+\infty)$ and $\phi_+(-\infty)$ are ordered.
\end{enumerate}
Indeed, let $x_n \to +\infty$ be an increasing sequence in
$\mathbb{R}$. Note that $\phi_-(x_n)=Q[\phi_-(\cdot-c_-+x_n)](0)\in
Q[\mc{C}_\beta](0)$, which is precompact in $\mc{X}_\beta$. It then
follows that there exists a subindex $\{n_l\}$ and $v\in
\mc{X}_\beta$ such that $\lim_{l\to \infty}\phi_-(x_{n_l})=v$,
which, together with the fact that $\phi_-$ is nondecreasing and
proposition \ref{SequenceInC}(1), implies that
$\phi_-(+\infty):=\lim_{x\to +\infty}\phi_-(x)=v$. Besides, from
\eqref{Eq:PhiMinus} we see that $\phi_+(+\infty)\in \mc{X}_\beta$ is
a fixed point of $Q$. Similar results hold for $\phi_-(-\infty)$ and
$\phi_+(\pm\infty)$. Recall that $\phi_-(-\infty)\le \phi_-(0)\le
\delta e_0$ and $\phi_+(+\infty)\ge \phi_+(0)\ge \beta-\delta
e_\beta$, which, together with the choice of $\delta$, implies that
$\phi_-(-\infty)=0$ and $\phi_+(+\infty)=\beta$. Further, since any
two real numbers are ordered, we see that there exist sequences
\begin{equation*}
\{n\}_{n\ge 0}\supset \{n_{1m}\}_{m\ge 1}\supset \{n_{2m}\}_{m\ge
2}\supset \cdot\cdot\cdot\supset \{n_{km}\}_{m\ge 1}\cdot \cdot
\cdot
\end{equation*}
such that for each $k\ge 1$,
\begin{equation*}
k+a_{n_{km}}\le-k+b_{n_{km}},\forall m\ge 1 \quad \text{or}\quad
k+a_{n_{km}}\ge-k+b_{n_{km}},\forall m\ge 1.
\end{equation*}
Define $\Gamma_1:=\{k\in\mathbb{N}:
k+a_{n_{km}}\le-k+b_{n_{km}},\forall m\ge 1 \}$ and $\Gamma_2:=
\mathbb{N}\setminus \Gamma_1$. Then either $\Gamma_1$ or $\Gamma_2$
has infinitely many elements. If $\Gamma_1$ does, then there holds
\begin{equation*}
\phi_{-,n_{km}}(k)=\phi_{n_{km}}(k+a_{n_{km}})\le
\phi_{n_{km}}(-k+b_{n_{km}})=\phi_{+,n_{km}}(-k),\forall
k\in\Gamma_1,m\in\mathbb{N}.
\end{equation*}
This implies that $\phi_-(k)\le \phi_+(-k),\forall k\in\Gamma_1$,
and hence, $\phi_-(+\infty)\le \phi_+(-\infty)$. If $\Gamma_2$ has
infinitely many elements, then we have $\phi_-(+\infty)\ge
\phi_+(-\infty)$ by a similar argument. Thus, $\phi_-(+\infty)$ and
$\phi_+(-\infty)$ must be ordered in $\mc{X}_\beta$.

Finally, we prove that either $\phi_-$ or $\phi_+$ connects $0$ to
$\beta$. Indeed, we have shown in the second step that
$\phi_-(+\infty)$ and $\phi_+(-\infty)$ are ordered. It then follows
from the bistability assumption (A5) that there are only three
possibilities:
\begin{enumerate}
\item[(i)]$\beta=\phi_-(+\infty)\ge \phi_+(-\infty)$;

\item[(ii)]$\phi_-(+\infty)\ge \phi_+(-\infty)=0$;

\item[(iii)]$\phi_-(+\infty)=\alpha=\phi_+(-\infty)$ for some
$\alpha\in E\setminus\{0,\beta\}$.
\end{enumerate}
We further claim that the possibility (iii) cannot happen.
Otherwise, Lemma \ref{MiniWaveSpeedDSF} implies that $c_+\ge
c_-^*(\alpha,\beta)$ and $c_-\le -c_+^*(0,\alpha)$. Since $c_-\ge
c_+$, it then follows that
\begin{equation*}
0\ge c_+ +(-c_-)\ge c_-^*(\alpha,\beta)+c_+^*(0,\alpha),
\end{equation*}
which contradicts assumption (A6). Thus, either (i) or (ii) holds,
and hence, we complete the proof.
\end{proof}

\subsection{Discrete-time semiflows in a discrete habitat}

In this case, both time $\mc{T}$ and habitat $\mc{H}$ are discrete:
$\mc{T}=\Z^+$ and $\mc{H}=\Z$. Without confusion, we consider the
semiflow $\{Q^n\}_{n\ge 0}$ in a metric space $\mc{E}\subset
\mc{C}$. Since the habitat is discrete, we cannot use the definition
of traveling waves with a unknown speed as in Definition
\ref{DefDisCont}. This is because the wave profile $\psi(x)$ may not
be well-defined for all $x\in\R$. So we start with the modification
of the definition of traveling waves in a discrete habitat.

\begin{definition}\label{DefDisDis}
$\psi(x+cn)$ with $\psi\in \mc{B}$ is said to be a traveling wave
with speed $c\in\mathbb{R}$ of the discrete semiflow $\{Q^n\}_{n\ge
0}$ if there exists a countable set $\Gamma\subset \R$ such that
$Q[\psi(\cdot+x)](i)=\psi(i+x+c),\forall i\in\Z, x\in\R\setminus
\Gamma$.
\end{definition}

By Definition \ref{DefDisDis} and Proposition \ref{SetProp1}, it
follows that there exists $x_0\in \R$ such that
$Q^n[\psi(\cdot+x_0)](i)=\psi(i+x_0+cn),\forall i\in\Z, n\ge0$.
Define $\phi(x):=\psi(x+x_0),\forall x\in\R$. Then, with a little
abuse of notation, we have $Q^n[\phi](i)=\phi(i+cn),\forall i\in\Z,
n\ge0$. Such a definition of traveling waves is motivated by the
idea employed in the proof of Theorem \ref{ThDisDis}.

Let $\beta\gg 0$ be a fixed point of $Q$. Define
$\tilde{Q}:\mc{B}_\beta\to \mc{B}_\beta$ by
\begin{equation*}
\tilde{Q}[\phi](x)=Q[\phi(\cdot+x)](0),\quad \forall x\in \R.
\end{equation*}
Then we see from  \cite[Lemma 2.1]{LiangZhao} that $\tilde{Q}$
satisfies (A1)-(A3) and (A5) with $Q=\tilde{Q}$ and
$\mc{C}_\beta=\mc{B}_\beta$ if $Q$ itself satisfies (A1)-(A3) and
(A5). Further, if $Q$ satisfies (A4), then the set
$\tilde{Q}[\mc{B}_\beta](x)\subset \mc{X}_\beta$ is precompact for
any $x\in\R$.

For $\tilde{Q}:\mc{B}_\beta\to \mc{B}_\beta$, we have similar
results as in Lemma \ref{UpperLowerSolu1} and \ref{FixedPoint1}.
\begin{lemma}\label{UpperLowerSolu2}
Assume that $Q$ satisfies (A1)-(A3) and (A5). Then there exists a
positive rational number $\bar{c}$ such that for any $c\ge \bar{c}$,
we have
\begin{equation*}
\tilde{Q}[\un{\psi}](x)\ge \un{\psi}(x-c)\quad \text{and}\quad
\tilde{Q}[\bar{\psi}](x)\le \bar{\psi}(x+c)\quad \text{for any
$x\in\R $}.
\end{equation*}
\end{lemma}

\begin{lemma}\label{FixedPoint2}
Assume that $Q$ satisfies (A1)-(A5). Then for each $n\in\mathbb{N}$,
$\tilde{G}_n:=\tilde{Q}\circ A_{\kappa_n}$ has a fixed point
$\tilde{\phi}_n$ in $\mc{B}_\beta$ such that $\tilde{\phi}_n$ is
nondecreasing and $\un{\psi}_n\le \tilde{\phi}_n\le \bar{\psi}_n$.
\end{lemma}
\begin{proof}
By the same arguments as in the proof of Lemma \ref{FixedPoint1}, we
can obtain a similar inequality as \eqref{LimitFunctionIncreasing}:
\begin{equation*}
\un{\psi}_n\le \tilde{G}_n^k[\un{\psi}_n ]\le
\tilde{G}_n^k[\bar{\psi}_n]\le \bar{\psi}_n,\quad \forall
k\in\mathbb{N}.
\end{equation*}
Define $w_{n,1}:=\un{\psi}_n$ and
$w_{n,k+1}:=\tilde{G}_n[w_{n,k}],k\ge 1$. Then
\begin{equation}\label{Zeq1}
w_{n,k+1}(x)=\tilde{Q}\circ
A_{\kappa_n}[w_{n,k}](x)=\tilde{Q}[w_{n,k}(\kappa_n\cdot)](x)=Q[w_{n,k}(\kappa_n(\cdot+x))](0).
\end{equation}
Note that $Q[\mc{C}_\beta]$ is compact and $w_{n,k}$ is
nondecreasing in $k$. It then follows that for any fixed
$x\in\mathbb{R}$, $w_{n,k}(x)$ converges in $\mc{X}_\beta$. Denote
the limit by $\tilde{\phi}_n(x)$. Then $\tilde{\phi}_n(x)$ is
nondecreasing in $x\in\mathbb{R}$ and $\un{\psi}_n\le
\tilde{\phi}_n\le \bar{\psi_n}$. Taking $k\to \infty$ in
\eqref{Zeq1}, we arrive at
$\tilde{\phi}_n(x)=Q[\tilde{\phi}_n(\kappa_n(\cdot+x))](0)$.
Consequently,
\begin{equation*}
\tilde{\phi}_n=\tilde{Q}[\tilde{\phi}_n(\kappa_n\cdot)]=\tilde{Q}\circ
A_{\kappa_n}[\tilde{\phi}_n]=\tilde{G}_n[\tilde{\phi}_n].
\end{equation*}
This completes the proof.
\end{proof}

To overcome the difficulty due to the lack of compactness for
$\tilde{Q}$, we will use the properties of monotone functions
established in the Appendix to show the convergence of a sequence in
$\tilde{Q}[\mc{B}_\beta]$.
\begin{theorem}\label{ThDisDis}
Assume that $\mc{X}=C(M,\R^d)$ and $Q$ satisfies (A1)-(A6). Then
there exists $c\in\mathbb{R}$ such that the semiflow $\{Q^n\}_{n\ge
1}$ on $\mc{C}_\beta$ admits a nondecreasing traveling wave
$\psi(x+cn)$ with speed $c$ and connecting $0$ to $\beta$. Further,
$\psi$ is either left or right continuous.
\end{theorem}
\begin{proof}
As in the proof of Theorem \ref{ThDisCont}, we define
\begin{equation*}
\tilde{a}_n:=\sup_{x\in\R}\{\tilde{\phi}_n(x)\in [0,\delta
e_0]_{\mc{X}}\},\quad
\tilde{b}_n:=\inf_{x\in\R}\{\tilde{\phi}_n(x)\in[\beta-\delta
e_\beta,\beta]_{\mc{X}}\}.
\end{equation*}
Then $-1-(n+\bar{c})\le \tilde{a}_n\le \tilde{b}_n\le
1+(n+\bar{c})$. Note that for any $x\in\R$, we have
\begin{equation*}
\tilde{\phi}_n(x)=\tilde{G}_n[\tilde{\phi}_n](x)
=\tilde{Q}[\tilde{\phi}_n(\kappa_n\cdot)](x)=Q[\tilde{\phi}_n
(\kappa_n(\cdot+x))](0)\in Q[\mc{C}_\beta](0).
\end{equation*}
Since $Q[\mc{C}_\beta](0)$ is precompact in $\mc{X}_\beta$, it then
follows that for any $x\in\R$, $\tilde{\phi}_n(x^-):=\lim_{y\uparrow
x} \tilde{\phi}_n(y)$ and $\tilde{\phi}_n(x^+):=\lim_{y\downarrow x}
\tilde{\phi}_n(y)$ both exist. And hence, by the definitions of
$\tilde{a}_n$ and $\tilde{b}_n$, we have
\begin{equation*}
\tilde{\phi}_n(\tilde{a}_n^-)\le \delta e_0\le \beta-\delta
e_\beta\le \tilde{\phi}_n(\tilde{b}_n^+),
\end{equation*}
but
\begin{equation*}
\tilde{\phi}_n(\tilde{a}_n^+)\not\in [0,\delta e_0]]_{\mc{X}}\quad
\text{and}\quad \tilde{\phi}_n(\tilde{b}_n^-)\not\in [[\beta-\delta
e_\beta,\beta]_{\mc{X}}.
\end{equation*}
Define $\tilde{\phi}_{-,n}(x):=\tilde{\phi}_n(x+\tilde{a}_n)$ and
$\tilde{\phi}_{+,n}(x):=\tilde{\phi}_n(x+\tilde{b}_n)$. Then
\begin{equation*}
\tilde{\phi}_{-,n}(0^-)\le \delta e_0\le \beta-\delta e_\beta\le
\tilde{\phi}_{+,n}(0^+),
\end{equation*}
but
\begin{equation*}
\tilde{\phi}_{-,n}(0^+)\not\in [0,\delta e_0]]_{\mc{X}}\quad
\text{and}\quad \tilde{\phi}_{+,n}(0^-)\not\in [[\beta-\delta
e_\beta,\beta]_{\mc{X}}.
\end{equation*}
Since $\tilde{\phi}_n=\tilde{G}_n[\tilde{\phi}_n]$, we have
\begin{equation*}
\tilde{\phi}_{-,n}(x)=\tilde{G}_n[\tilde{\phi}_n] (x+\tilde{a}_n)
=Q[\tilde{\phi}_n(\kappa_n(\cdot+\tilde{a}_n+x))](0)\in
Q[\mc{C}_\beta](0).
\end{equation*}
Similarly,
$\tilde{\phi}_{+,n}(x)=Q[\tilde{\phi}_n(\kappa_n(\cdot+\tilde{b}_n+x))](0)\in
Q[\mc{C}_\beta](0)$. Let $\mathbb{Q}$ be the set of all rational
numbers, and  $\{x_l\}_{l\ge 1}\subset \mathbb{Q}$ be an increasing
sequence converging to $x$. Using
$\tilde{\phi}_n=\tilde{G}_n[\tilde{\phi}_n]$ again, we see that for
any $i\in\Z$ and $l\ge 1$,
\begin{equation*}
\tilde{\phi}_n(\kappa_n(i+\tilde{a}_n+x_l))
=Q[\tilde{\phi}_n(\kappa_n(\cdot
+\kappa_n(i+\tilde{a}_n+x_l)))](0)\in Q[\mc{C}_\beta](0).
\end{equation*}
Similarly, $\tilde{\phi}_n(\kappa_n(i+\tilde{b}_n+x_l))\in
Q[\mc{C}_\beta](0)$. Note that $Q[\mc{C}_\beta](0)$ is precompact in
$\mc{X}_\beta$ and that $\mathbb{Q}$ is countable. It then follows
that there exists a subindex (still denoted by $\{n\}$) and
$\xi_-\le \xi_+\in\mathbb{R}$ such that $\lim_{n\to
\infty}\tilde{a}_{n}/n=\xi_-, \lim_{n\to
\infty}\tilde{b}_{n}/n=\xi_+$ and that for any $x\in\mathbb{Q}$,
$i\in\Z$ and $l\ge 1$, sequences $\tilde{\phi}_{\pm,n}(x),
\tilde{\phi}_n(\kappa_n(i+\tilde{a}_n+x_l))$ and
$\tilde{\phi}_n(\kappa_n(i+\tilde{b}_n+x_l))$ converge in
$\mc{X}_\beta$. And hence, the following limits
\begin{equation*}
\lim_{l\to \infty}\lim_{n\to
\infty}\tilde{\phi}_{-,n}(x_l)=\lim_{l\to \infty}\lim_{n\to
\infty}Q[\tilde{\phi}_n(\kappa_n(\cdot+\tilde{a}_n+x_l))](0)
\end{equation*}
and
\begin{equation*}
\lim_{l\to \infty}\lim_{n\to
\infty}\tilde{\phi}_{+,n}(x_l)=\lim_{l\to \infty}\lim_{n\to
\infty}Q[\tilde{\phi}_n(\kappa_n(\cdot+\tilde{b}_n+x_l))](0)
\end{equation*}
both exist. This means the limits
\begin{equation*}
\lim_{y\in\mathbb{Q},y\uparrow x}\lim_{n\to
\infty}\tilde{\phi}_{\pm,n}(y)\quad\text{and}\quad
\lim_{y\in\mathbb{Q},y\downarrow x}\lim_{n\to
\infty}\tilde{\phi}_{\pm,n}(y)
\end{equation*}
exist for all $x\in\R$. Define
\begin{equation*}
\hat{\phi}_-(x):=
\begin{cases}
\lim_{n\to \infty}\tilde{\phi}_{-,n}(x),& x\in\mathbb{Q}\\
\lim_{y\in\mathbb{Q},y\uparrow x}\lim_{n\to
\infty}\tilde{\phi}_{-,n}(x), & x\in \R\setminus \mathbb{Q},
\end{cases}
\end{equation*}
and
\begin{equation*} \hat{\phi}_+(x):=
\begin{cases}
\lim_{n\to \infty}\tilde{\phi}_{+,n}(x),& x\in\mathbb{Q}\\
\lim_{y\in\mathbb{Q},y\downarrow x}\lim_{n\to
\infty}\tilde{\phi}_{+,n}(x), & x\in \R\setminus \mathbb{Q}.
\end{cases}
\end{equation*}
Clearly, $\hat{\phi}_\pm$ are nondecreasing functions in
$\mc{B}_\beta$ and for any $x\in \R\setminus \mathbb{Q}$,
$\hat{\phi}_\pm(x^\pm)$ all exist. Hence, we see from Theorem
\ref{HellyTh}, that there exists a countable subset $\Gamma_1$ of
$\R$ such that $\tilde{\phi}_{\pm,n}(x)$ converges to
$\hat{\phi}_\pm(x)$ for all $x\in\R\setminus \Gamma_1$. Define
\begin{equation*}
\tilde{\phi}_-(x):=\lim_{y\in\mathbb{Q},y\uparrow x}\lim_{n\to
\infty}\tilde{\phi}_{-,n}(y),\quad \forall x\in \R.
\end{equation*}
and
\begin{equation*}
\tilde{\phi}_+(x):=\lim_{y\in\mathbb{Q},y\downarrow x}\lim_{n\to
\infty}\tilde{\phi}_{+,n}(y),\quad \forall x\in \R.
\end{equation*}
Thus, $\tilde{\phi}_- (x)$ is left continuous and
$\tilde{\phi}_+(x)$ is right continuous. Note that
$\tilde{\phi}_\pm(x)=\hat{\phi}_\pm(x)$ for all $x\in \R\setminus
\mathbb{Q}$. It then follows that $\tilde{\phi}_{\pm,n}(x)$
converges to $\tilde{\phi}_\pm(x)$ for $x\in \R\setminus \Gamma_2$,
where $\Gamma_2:=\mathbb{Q}\cup \Gamma_1$ is also countable.

Let $y_k\in\R\setminus \Gamma_2$ be an increasing sequence
converging to $0$ and $z_k\in \R\setminus \Gamma_2$ be an increasing
sequence converging to $1$, respectively. Note that
\begin{equation*}
\tilde{\phi}_-(0)=\lim_{k\to \infty}\tilde{\phi}_- (y_k)=\lim_{k\to
\infty}\lim_{n\to \infty}\tilde{\phi}_{-,n} (y_k) \le \delta e_0,
\end{equation*}
and
\begin{equation*}
\tilde{\phi}_-(1)=\lim_{k\to \infty}\tilde{\phi}_- (z_k)=\lim_{k\to
\infty}\lim_{n\to \infty}\tilde{\phi}_{-,n} (z_k)\not\in[0,\delta
e_0]]_{\mc{X}}.
\end{equation*}
Similarly, we have $\tilde{\phi}_+(0)\ge \beta-\delta e_\beta$ but
$\tilde{\phi}_+(-1)\not\in[[\beta-\delta e_\beta,\beta]_{\mc{X}}$.
Define $c_-:=-\bar{c}\xi_-$ and $c_+:=-\bar{c}\xi_+$. Obviously,
$c_-\ge c_+$ since $\xi_-\le \xi_+$. Now we want to prove
$Q[\tilde{\phi}_-(\cdot+x)](0)=\tilde{\phi}_-(x+c_-), \, \forall
x\in\R\setminus \Gamma_2$. Note that
\begin{equation*}
\lim_{n\to\infty}\kappa_{n}(x+\tilde{a}_n)-\tilde{a}_n
=\lim_{n\to\infty} \left(x+\bar{c}\cdot\f{x+\tilde{a}_n}{n}\right)
=x-c_-.
\end{equation*}
It then follows that
\begin{eqnarray*}
\tilde{\phi}_-(x+c_-)&&=\lim_{n\to
\infty}\tilde{\phi}_{-,n}(x+c_-)=\lim_{n\to
\infty}\tilde{\phi}_n(x+c_-+\tilde{a}_n)=\lim_{n\to \infty}\tilde{G}_n[\tilde{\phi}_n](x+c_-+\tilde{a}_n)\nonumber\\
&&=\lim_{n\to \infty}
\tilde{Q}[\tilde{\phi}_n(\kappa_n\cdot)](x+c_-+\tilde{a}_n)=\lim_{n\to
\infty}\tilde{Q}[
\tilde{\phi}_n(\kappa_n(\cdot+a_n))](x+c_-)\nonumber\\
&&=\lim_{n\to \infty}\tilde{Q}[
\tilde{\phi}_{-,n}(\kappa_n(\cdot+\tilde{a}_n)-\tilde{a}_n)](x+c_-)\nonumber\\
&&=\lim_{n\to \infty} Q[
\tilde{\phi}_{-,n}(\kappa_n(\cdot+x+c_-+\tilde{a}_n)-\tilde{a}_n)](0).
\end{eqnarray*}
In view of Proposition \ref{Convergence}, we obtain that
$\tilde{\phi}_-(x+c_-)=Q[\tilde{\phi}_-(\cdot+x)](0)$
$x\in\R\setminus \Gamma_2$. A similar result also holds for
$\tilde{\phi}_+$.

Now, the same argument as in the proof of Theorem \ref{ThDisCont}
completes the proof.
\end{proof}

\subsection{Time-periodic semiflows}

Let $\omega\in \mc{T}$ be a positive number, where $\mc{T}=\R^+$ or
$\Z^+$. Recall that a family of mappings $\{Q_t\}_{t\in \mc{T}}$ is
said to be an $\omega$-time periodic semiflow on a metric space
$\mc{E}\subset\mc{C}$ provided that it has the following properties:
\begin{enumerate}
\item[(i)]$Q_0[\phi]=\phi,\,  \forall \phi\in \mc{E}$.
\item[(ii)]$Q_t\circ Q_\omega[\phi]=Q_{t+\omega}[\phi],\, \forall t\ge
0,\, \phi\in \mc{E}$.
\item[(iii)]$Q_t[\phi]$ is continuous jointly in $(t,\phi)$ on
$[0,\infty)\times \mc{E}$.
\end{enumerate}
The mapping $Q_\omega$ is called the Poincar\'e map associated with
this periodic semiflow.

\begin{definition}\label{DefPerTime}
\begin{enumerate}
\item[(i)]In the case where $\mc{H}=\R$, $U(t,x+ct)$ is said to be
an $\omega$-time periodic traveling wave with speed $c$ of the
semiflow $\{Q_t\}_{t\in \mc{T}}$ if $Q_t[U(0,\cdot)](x)=U(t,x+ct)$
and $U(t,x)=U(t+\omega,x)$ for all $t\in \mc{T},x\in \R$.
\item[(ii)]In the case where $\mc{H}=\Z$, $U(t,x+ct)$ is said to be
an $\omega$-time periodic traveling wave with speed $c$ of the
semiflow $\{Q_t\}_{t\in \mc{T}}$ if there exists a countable subset
$\Gamma\subset \R$ such that $Q_t[U(0,\cdot+x)](0)=U(t,x+ct)$ for
all $t\in \mc{T},x\in \R$ and $U(t,x)=U(t+\omega,x)$ for all $t\in
\mc{T},x\in \R\setminus\Gamma$.
\end{enumerate}
\end{definition}

\begin{theorem}\label{ThTimePeri}
Let $\beta(t)$ be a strongly positive $\omega$-time periodic orbit
of $\{Q_t\}_{t\in\mc{T}}$ restricted on $\mc{X}$. Assume that
$Q:=Q_\omega$ satisfies hypotheses (A1)-(A6) with $\beta=\beta(0)$.
Then $\{Q_t\}_{t\in\mc{T}}$ admits a traveling wave $U(t,x+ct)$ with
$U(t,-\infty)=0$ and $U(t,+\infty)=\beta(t)$ uniformly for $t\in
\mc{T}$. Furthermore, $U(t,x)$ is nondecreasing in $x\in\R$.
\end{theorem}
\begin{proof}
Case 1. $\mc{H}=\R$. Since the map $Q_\omega$ satisfies (A1)-(A6),
there exits $c\in \R$ and a nondecreasing function $\phi\in \mc{C}$
connecting $0$ to $\beta(0)$ such that
$Q_\omega[\phi](x)=\phi(x+c\omega)$. Clearly, $T_{c\omega}Q_\omega
[\phi]=\phi$. Define $U(t,x):=T_{ct} Q_t[\phi](x)$. Then we have
$U(t,x+ct)=Q_t[\phi](x)=Q_t[U(0,\cdot)](x)$, and
\begin{equation*}
U(t+\omega,x)=T_{ct+c\omega}Q_{t+\omega}[\phi](x)=T_{ct} Q_t
T_{c\omega} Q_\omega[\phi](x)=T_{ct} Q_t [\phi](x)=U(t,x).
\end{equation*}
Note that $Q_t[\beta(0)]=\beta(t)$ and that $\phi$ is nondecreasing
and connecting $0$ to $\beta(0)$. It then follows that
$U(t,-\infty)=0$ and $U(t,+\infty)=\beta(t)$.

Case 2. $\mc{H}=\Z$. Since the map $Q_\omega$ satisfies (A1)-(A6),
there exits $c\in \R$, a countable subset $\Gamma\subset \R$ and a
nondecreasing function $\phi\in \mc{B}$ connecting $0$ to $\beta(0)$
such that $\tilde{Q}_\omega[\phi](x)=\phi(x+c\omega),\forall x\in
\R\setminus \Gamma$. Clearly, $T_{c\omega}\tilde{Q}_\omega
[\phi](x)=\phi(x),\forall x\in \R\setminus \Gamma$. Define
$U(t,x):=T_{ct} \tilde{Q}_t[\phi](x)$. Thus, we have
\begin{equation*}
U(t,x+ct)=\tilde{Q}_t[\phi](x)=\tilde{Q}_t[U(0,\cdot)](x)=Q_t[U(0,\cdot+x)](0),
\, \, \forall x\in \R,
\end{equation*}
and
\begin{equation*}
U(t+\omega,x)=T_{ct+c\omega}\tilde{Q}_{t+\omega}[\phi](x)=T_{ct}
\tilde{Q}_t T_{c\omega} \tilde{Q}_\omega[\phi](x)=T_{ct} \tilde{Q}_t
[\phi](x)=U(t,x),
\end{equation*}
for all $x\in \R\setminus \Gamma$. Note that
$Q_t[\beta(0)]=\beta(t)$ and that $\phi$ is nondecreasing and
connecting $0$ to $\beta(0)$. It then follows that $U(t,-\infty)=0$
and $U(t,+\infty)=\beta(t)$.
\end{proof}

\subsection{Continuous-time semiflows in a continuous habitat}
In this subsection, we consider continuous-time semiflows in the
continuous habitat $\mc{H}=\R$. Recall that a family of mappings
$\{Q_t\}_{t\ge 0}$ is said to be a semiflow on a metric space
$\mc{E}\subset\mc{C}$ provided that $Q_t:\mc{E}\to \mc{E}$ satisfies
the following properties:
\begin{enumerate}
\item[(1)]$Q_0[\phi]=\phi,\, \forall \phi\in \mc{E}$.
\item[(2)]$Q_t\circ Q_s[\phi]=Q_{t+s}[\phi],\, \forall  t,s\ge 0, \,
\phi\in \mc{E}$.
\item[(3)]$Q_t[\phi]$ is continuous jointly in $(t,\phi)$ on
$[0,\infty)\times \mc{E}$.
\end{enumerate}

Before moving to the study of traveling waves of the semiflow
$\{Q_t\}_{t\ge 0}$, we first investigate the spatially homogeneous
system, that is, the system restricted on $\mc{X}$. Let $\beta\gg 0$
be an equilibrium in $\mc{X}$. For each $t>0$, we use $\Sigma_t$ to
denote the set of all fixed points of the map $Q_t$ restricted on
$\mc{X}_\beta$. Clearly, the equilibrium set of the semiflow is
$\Sigma:=\cap_{t>0}\Sigma_t$, which is a subset of $\Sigma_t$ for
any $t>0$. The subsequent results indicates that the instability of
intermediate equilibria of the semiflow implies the nonordering
property of all intermediate fixed points of each time-$t$ map.
\begin{proposition}\label{BistabilitySemiflow}
For any given $t>0$, if the map $Q_t$ satisfies the bistability
assumption (A5$'$) with $E=\Sigma$, then $Q_t$ satisfies (A5) with
$E=\Sigma_t$.
\end{proposition}
\begin{proof}
Let $t_0>0$ be given. We first show that any two points $u\in
\Sigma\setminus \{0,\beta\}$ and $v\in \Sigma_{t_0}\setminus
\{0,\beta\}$ are unordered. Assume, for the sake of contradiction,
that $u$ and $v$ are ordered. Without loss of generality, we also
assume that $u<v$. Then the eventual strong monotonicity implies
that $u\ll v$. Since $u$ is strongly unstable from above, there
exist a unit vector $e\in Int(\mc{X}^+)$ and a number $\delta>0$
such that $Q_{t_0}[u+\delta e]\gg u+\delta e$ with $u+\delta e\in
[[u,v]]_{\mc{X}}$. From \cite[Theorem 1.2.1]{Smith}, we see that
$(Q_{t_0})^n[u+\delta e]$ is eventually strongly increasing and
converges to some $\alpha\in \Sigma$. Note that
$\alpha\in[[u,v]_{\mc{X}}$ is strongly  unstable from below. Hence,
by the same arguments as in the proof of Proposition
\ref{BistabilityMap}, we obtain a contradiction.

Next we show the set $\Sigma_{t_0}\setminus\Sigma$ is unordered. For
this purpose, we see from the first step that it suffices to prove
that for any two ordered elements $u<v$ in
$\Sigma_{t_0}\setminus\Sigma$, $[u,v]_{\mc{X}}\cap \Sigma
\neq\emptyset$. Indeed, by the eventual strong monotonicity, we have
$u\ll v$. Then, we can choose a sequence $\{u_n\}$ on the segment
connecting $u$ and $v$ such that $u\ll u_n \ll u_{n+1}\ll v, \quad
\forall n\geq 1$. By \cite[Theorem 1.3.7]{Smith}, it follows that $
\omega (u)\leq \omega (u_n)\leq \omega (u_{n+1})\leq \omega(v),
\forall n\geq 1$. Clearly, we have $\omega (u)=\{Q_tu: \, t\in
[0,t_0]\}$ and $\omega (v)=\{Q_tv: \, t\in [0,t_0]\}$, and hence
$u\leq \omega (u_n)\leq v,\, \forall n\geq 1$. Note that
$\cup_{n\geq 1}\omega (u_n)$ is contained in the compact set
$\overline{Q_{t_0}[\mc{X}_\beta]}$. In the compact metric space
consisting of all nonempty compact subsets of
$\overline{Q_{t_0}[\mc{X}_\beta]}$ with Hausdorff distance $d_H$,
the sequence $\{\omega (u_n):\, n\ge 1\}$ has a convergent
subsequence. Without loss of generality, we assume that for some
nonempty compact set $\varpi\subset
\overline{Q_{t_0}[\mc{X}_\beta]}$, $\lim_{n\to \infty}d_H(\omega
(u_n), \varpi)=0$. Since each $\omega (u_n)$ is invariant for the
semiflow $\{Q_t\}_{t\geq 0}$, so is the compact set $\varpi$, that
is, $Q_t\varpi=\varpi, \, \forall t\geq 0$. For any given $x,y\in
\omega$, there exist two sequences of points $x_n,y_n\in\varpi
(u_n)$ such that $x_n\to x$ and $y_n\to y$ as $n\to \infty$. Since
$\omega (u_n)\leq \omega (u_{n+1})$, we have $x_n\leq y_{n+1}$ and
$y_n\leq x_{n+1},\, \forall n\geq 1$. Letting $n\to \infty$, we then
have $x\leq y$ and $y\leq x$, and hence $x=y$. This implies that
$\varpi$ is a singleton, that is, $\varpi=\{\alpha\}$. By the
invariance of $\varpi$ for the semiflow, we see that $\alpha$ is an
equilibrium. Since $u\leq \omega (u_n)\leq v,\, \forall n\geq 1$, it
follows that $\alpha\in [u,v]_X$.
\end{proof}

For a continuous-time semiflow $\{Q_t\}_{t\ge 0}$, we need the
following definition of traveling waves.
\begin{definition}\label{DefContCont}
$\psi(x+ct)$ with $\psi\in \mc{C}$ is said to be a traveling wave
with speed $c\in\mathbb{R}$ of the continuous-time semiflow
$\{Q_t\}_{t\ge0}$ if $Q_t[\psi](x)=\psi(x+ct),\forall x\in\R, t\ge
0$.
 we say that $\psi$ connects $0$ to $\beta$ if
$\psi(-\infty)=0$ and $\psi(+\infty)=\beta$.
\end{definition}
\begin{theorem}\label{ThContCont}
Assume that for each $t>0$, the map $Q_t$ satisfies assumptions
(A1),(A3)-(A5) with $E=\Sigma_t$, and the time-one map $Q_1$
satisfies (A6) with $E=\Sigma$. Then there exists $c\in\mathbb{R}$
such that $\{Q_t\}_{t\ge 0}$ admits a non-decreasing traveling wave
with speed $c$ and connecting $0$ to $\beta$ .
\end{theorem}
\begin{proof}
Let $e_0,e_\beta$ and $\delta$ be chosen as in \ref{delta0},
\ref{deltabeta} and \ref{delta}, respectively. We proceeds with
three steps.

Firstly, we show that there exists $s_k\downarrow 0$ such that each
discrete semiflow $\{Q_{s_k}^n\}_{n\ge 0}$ admits two nondecreasing
traveling waves $\psi_{\pm,s_k}(x+c_{\pm,s_k}t)$ with $c_{-,s_k}\ge
c_{+,s_k}$ and $\psi_{\pm,s_k}$ has the following properties:
\begin{equation*}
0<\psi_{-,s_k}(0)\le \delta e_0\quad \text{and}\quad  \beta-\delta
e_\beta \le \psi_{+,s_k}(0) <\beta,
\end{equation*}
but
\begin{equation*}
\psi_{-,s_k}(0)\not\in [0,\delta e_0]]_{\mc{X}}\quad \text{and}\quad
\psi_{+,s_k}(0)\not\in[[\beta-\delta e_\beta,\beta]_{\mc{X}}.
\end{equation*}
Indeed, since for each $s>0$ the map $Q_s$ satisfies (A1)-(A5), from
the first two steps of the proof for Theorem \ref{ThDisCont},  we
see that for the discrete semiflow $\{(Q_s)^n\}_{n\ge 0}$, there
exists two nondecreasing traveling waves $\phi_{\pm,
s}(x+c_{\pm,s}t)$ with the following properties:
\begin{enumerate}
\item[(1)] $\phi_{-,s}$ connects $0$ to some $\alpha_{-,s}\in
E_s\setminus \{0\}$ and $\phi_{+,s}$ connects some $\alpha_{+,s}\in
E_s\setminus \{\beta\}$ to $\beta$;

\item[(2)] $\alpha_{-,s}$ and $\alpha_{+,s}$ are ordered and $c_{-,s}\ge
c_{+,s}$.
\end{enumerate}
By a similar argument as in \cite[Theorem 1.3.7]{ZhaoBook}, it then
follows that both $\alpha_{\pm,s}$ have a subsequence
$\alpha_{\pm,s_k}$ which tends to an equilibrium of the semiflow as
$s_k\to 0$, say the limit $\alpha_-$ and $\alpha_+$, respectively.
Since $\alpha_{-,s}$ and $\alpha_{+,s}$ are ordered, it follows from
Proposition \ref{BistabilitySemiflow} that there are only three
possibilities for the relation of $\alpha_-$ and $\alpha_+$:
\begin{equation*}
\text{(i) $\beta=\alpha_-\ge\alpha_+$; (ii) $\alpha_-\ge\alpha_+=0$;
and (iii) $\alpha_-=\alpha_+\in E\setminus \{0,\beta\}$.}
\end{equation*}
If $\alpha_-=\beta$, then for sufficiently large $k$ we can define
\begin{equation*}
a_{s_k}:=\sup\{x\in\mathbb{R}:\phi_{-,s_k}(x)\in[0,\delta
e_0]_{\mc{X}}\},\quad
b_{s_k}:=\inf\{x\in\mathbb{R}:\phi_{-,s_k}(x)\in[\beta-\delta
e_\beta,\beta]_{\mc{X}}\}.
\end{equation*}
And hence, $\psi_{-,s}(x):=\psi_s(x+a_s)$ and
$\psi_{+,s}(x):=\psi_s(x+b_s)$ are the required traveling waves. If
$\alpha_+=0$, then for sufficient large $k$ we can define
\begin{equation*}
a_{s_k}:=\sup\{x\in\R:\phi_{+,s_k}(x)\in[0,\delta e_0]_{\mc{X}}\},\,
b_{s_k}:=\inf\{x\in\R:\phi_{+,s_k}(x)\in[\beta-\delta
e_\beta,\beta]_{\mc{X}}\}.
\end{equation*}
And hence, $\psi_{-,s_k}(x):=\phi_{+,s_k}(x+a_{s_k})$ and
$\psi_{+,s_k}(x):=\phi_{+,s_k}(x+b_{s_k})$ are the required
traveling waves. If $\alpha_-=\alpha_+\in E\setminus \{0,\beta\}$,
then by Lemma \ref{IslatedEquilibria} we have $\alpha_-=\alpha_+\in
E\setminus \{[0,\delta e_0]_\mc{X}\cup [\beta-\delta
e_\beta,\beta]_\mc{X}\}$. Consequently, for sufficiently large $k$
we can define
\begin{equation*}
a_{s_k}:=\sup\{x\in\R:\phi_{-,s_k}(x)\in[0,\delta e_0]_{\mc{X}}\},\,
b_{s_k}:=\inf\{x\in\R:\phi_{-,s_k}(x)\in[\beta-\delta
e_\beta,\beta]_{\mc{X}}\}.
\end{equation*}
And hence, $\psi_{-,s_k}(x):=\phi_{-,s_k}(x+a_{s_k})$ and
$\psi_{+,s_k}(x):=\phi_{-,s_k}(x+b_{s_k})$ are the required
traveling waves.

Secondly, we show that there exists a subindex, still denoted by
$s_k$, such that $\psi_{\pm,s_k}\to \psi_{\pm}$ in $\mc{C}_\beta$
and $\f{1}{s_k}c_{\pm, s_k}\to c_\pm\in\R$. Indeed, for each
$s_k>0$, there exists an integer $m_k>0$ such that $m_ks_k>2$. Then
\begin{equation}\label{PhiSConvergent}
\psi_{-,s_k}=T_{m_kc_{s_k}}\circ Q_{m_ks_k}[\psi_{-,s_k}]=Q_2\circ
Q_{m_ks_k-2}\circ T_{m_kc_{s_k}}[\psi_{-,s_k}]\in Q_1\circ
Q_1[\mc{C}_\beta].
\end{equation}
Clearly, the compactness of $Q_1$ implies that the set $Q_1\circ
Q_1[\mc{C}_\beta]$ is precompact in $\mc{C}_\beta$. Thus, there
exists subsequence, still denoted by $s_k$, and nonincreasing
functions $\psi_-,\psi_+\in \mc{C}_\beta$ with $0<\psi_-(0)\le
\delta e_0$ and $\beta-\delta e_\beta \le \psi_+(0) <\beta$ such
that $\psi_{-,s_k}\to \psi_-$ and $\psi_{+,s_k}\to \psi_+$ in
$\mc{C}_\beta$. Also we claim that $\psi_{\pm,s_k}(\pm\infty)$ all
exist. Indeed, from \eqref{PhiSConvergent} we see that there exists
$\phi_{s_k}\in\mc{C}_\beta$ such that $ Q_1\circ Q_1[\phi_{s_k}]\to
\psi_-$. Note that $\{Q_1[\phi_{s_k}]\}_{k\ge 1}$ also has a
convergent subsequence with the limit $\phi\in\mc{C}_\beta$. And
hence, by the uniqueness of limit we have $Q_1[\phi]=\psi_-$. Note
that $\psi_-(k)=Q_1[\phi](k)=Q_1[\phi(\cdot+k)](0)$ and
$\{Q_1[\phi(\cdot+k)]\}_{k\ge 1}$ has a convergent subsequence. It
then follows that $\psi_-(\pm\infty)$ exist because $\psi_-$ is
nonincreasing. Similarly, $\psi_+(\pm\infty)$ exist. Also, we have
\begin{equation}\label{Connecting0orBeta}
\psi_-(-\infty)\le \psi_-(0)\le \delta e_0\quad \text{and}\quad
\psi_+(+\infty)\ge \psi_+(0)\ge \beta-\delta e_\beta,
\end{equation}
but
\begin{equation*}
\psi_{-}(0)\not\in [0,\delta e_0]]_{\mc{X}}\quad \text{and}\quad
\psi_{+}(0)\not\in[[\beta-\delta e_\beta,\beta]_{\mc{X}}.
\end{equation*}
Consequently, by the monotonicity of $\psi_\pm$, we have
\begin{equation}\label{relation}
\psi_-(x)\not\in [0,\delta e_0]]_{\mc{X}},\forall x>0\quad
\text{and}\quad \psi_{+}(x)\not\in[[\beta-\delta
e_\beta,\beta]_{\mc{X}},\forall x<0.
\end{equation}
Since $\psi_-$ and $\psi_+$ are the limits of the sequence of
monotone functions with different translations, respectively, we can
employ the same arguments as in the second step of the proof of
Theorem \ref{ThDisCont} to show that $\psi_-(+\infty)$ and
$\psi_+(-\infty)$ are ordered.

To prove that $\f{1}{s_k}c_{\pm, s_k}$ have convergent subsequences,
we only need to prove that $\f{1}{s_k}c_{-, s_k}$ is bounded above
and $\f{1}{s_k}c_{+, s_k}$ is bounded below because $c_{-,s_k}\ge
c_{+,s_k}$. Assume, for the sake of contradiction, that some
subsequence, still say $\f{1}{s_k}c_{-,s_k}$, tends to $+\infty$.
Note that for each $s>0$ there exists $n_s\in \mathbb{Z}^+$ such
that the integer part of $\f{1}{s}$, denoted by $\langle
\f{1}{s}\rangle$, equals $n_s$ and $\f{1}{n_s+1}<s\le \f{1}{n_s}$.
Hence, $s\langle \f{1}{s}\rangle\to 1$ as $s\to 0$. It then follows
that
\begin{equation*}
\lim_{k\to \infty}\langle \f{1}{s_k}\rangle c_{-,s_k}=\lim_{k\to
\infty} \f{1}{s_k}c_{-,s_k}\times s_k\langle
\f{1}{s_k}\rangle=\lim_{k\to \infty} \f{1}{s_k}c_{-,s_k}=+\infty.
\end{equation*}
Thus, using the first observation in \eqref{relation}, we have
\begin{eqnarray}\label{BistabilityInequ1}
Q_1[\delta e_0]&&\ge
Q_1[\psi_-(-\infty)]=Q_1[\psi_-(-\infty)](0)=\lim_{x\to
-\infty}Q_1[\psi_-(\cdot+x)](0)\nonumber\\
&&=\lim_{x\to -\infty}Q_1[\psi_-](x)=\lim_{x\to -\infty}\lim_{k\to
\infty}
(Q_{s_k})^{\langle\f{1}{s_k}\rangle }[\psi_{-,s_k}](x)\nonumber\\
&&=\lim_{x\to -\infty}\lim_{k\to \infty} \psi_{-,s_k}(x+\langle
\f{1}{s_k}\rangle c_{s_k})\ge  \lim_{x\to -\infty}\lim_{y\to
+\infty}\lim_{k\to \infty}
\psi_{-,s_k}(y)\nonumber\\
&&=\lim_{y\to +\infty}\psi_-(y)=\psi_-(+\infty)\not\in [0,\delta
e_0]]_{\mc{X}},
\end{eqnarray}
which contradicts the fact that $Q_1[\delta e_0]\ll \delta e_0$.
Similarly, if $\f{1}{s_k}c_{+, s_k}\to -\infty$, then the second
observation in \eqref{relation} implies that
\begin{eqnarray}\label{BistabilityInequ2}
Q_1[\beta-\delta e_\beta]&&\le
Q_1[\psi_+(+\infty)]=Q_1[\psi_+(+\infty)](0)=\lim_{x\to
+\infty}Q_1[\psi_+(\cdot+x)](0)\nonumber\\
&&=\lim_{x\to +\infty}Q_1[\psi_+](x)=\lim_{x\to +\infty}\lim_{k\to
\infty}
(Q_{s_k})^{\langle\f{1}{s_k}\rangle }[\psi_{+,s_k}](x)\nonumber\\
&&=\lim_{x\to +\infty}\lim_{k\to \infty} \psi_{+,s_k}(x+\langle
\f{1}{s_k}\rangle c_{+,s_k})\le  \lim_{x\to +\infty}\lim_{y\to
-\infty}\lim_{k\to \infty}
\psi_{+,s_k}(y)\nonumber\\
&&=\lim_{y\to -\infty}\psi_+(y)=\psi_+(-\infty)\not\in
[[\beta-\delta e_\beta,\beta]_{\mc{X}},
\end{eqnarray}
which contradicts the fact that $Q_1[\beta-\delta e_\beta]\gg
\beta-\delta e_\beta$.
 Consequently, $\f{1}{s_k}c_{\pm, s_k}$ are bounded.

Finally, we show that either $\psi_-(x+c_-t)$ or $\psi_+(x+c_+t)$
established in the second step is a traveling wave connecting $0$ to
$\beta$. Indeed, for any $t>0$, there exists $m_k\in\mathbb{Z}$ and
$r_k\in [0,s_k)$ such that $t=m_k s_k-r_k$. Clearly, $r_k\to 0$ as
$k\to \infty$. Then we have
\begin{eqnarray*}
&&Q_t[\psi_{\pm}]=\lim_{k\to\infty}Q_{t+r_k}[\psi_{\pm,s_k}]=\lim_{k\to\infty}
Q_{m_ks_k}[\psi_{\pm,s_k}]=\lim_{k\to\infty}
\psi_{\pm,s_k}(\cdot+m_kc_{\pm,s_k})\\
&&=\lim_{k\to\infty}
\psi_{\pm,s_k}\left(\cdot+(t+r_k)\f{1}{s_k}c_{\pm,s_k}\right)=\psi_{\pm}(\cdot+c_\pm
t),
\end{eqnarray*}
where the last equality follows from Proposition
\ref{SequenceInC}(2). From the equality
$Q_t[\psi_{\pm}]=\psi_{\pm}(\cdot+ct),\forall t\ge0$, we see that
$\psi(\pm\infty)$ are equilibria. Recall that $\psi_-(-\infty)\le
\delta e_0\le \psi(+\infty)$ and $\psi_+(+\infty)\ge \beta-\delta
e_\beta\ge\psi_+(-\infty)$. It then follows that
$\psi_-(-\infty)=0,\psi_+(+\infty)=\beta$, and there are only three
possibilities for $\psi_-(+\infty)$ and $\psi_+(-\infty)$:
\begin{enumerate}
\item[(i)]$\beta=\psi_-(+\infty)>\psi_+(-\infty)$;

\item[(ii)]$\psi_-(+\infty)>\psi_+(-\infty)=0$;

\item[(iii)] $\psi_-(+\infty)=\alpha=\psi_+(-\infty)$ for some $\alpha\in
\Sigma\setminus \{0,\beta\}$.
\end{enumerate}
Since the time-one map $Q_1$ satisfies (A6) with $E=\Sigma$, we can
employ the same arguments as in the proof of Lemma
\ref{MiniWaveSpeedDSF} to exclude the possibility (iii). Thus,
either (i) or (ii) holds, and hence, we complete the proof.
\end{proof}

\subsection{Continuous-time semiflows in a discrete habitat}

In this case, time $\mc{T}=\mathbb{R}^+$ and habitat $\mc{H}=\Z$.
Let $\beta\gg0$ be an equilibrium of the semiflow $\{Q_t\}_{t\ge
0}$. We start with the definition of traveling waves for this case.

\begin{definition}\label{DefContDis}
$\psi(i+ct)$ with $\psi\in \mc{B}_\beta$ is said to be a traveling
wave with speed $c\in\mathbb{R}$ of the continuous-time semiflow
$\{Q_t\}_{t\ge0}$ if $Q_t[\psi](i)=\psi(i+ct),\forall i\in\Z,t\ge
0$. Clearly, $\psi$ is continuous if $c\neq 0$.
\end{definition}

For each $t>0$, define $\tilde{Q}_t:\mc{B}_\beta\to \mc{B}_\beta$ by
$\tilde{Q}_s[\phi](x)=Q_s[\phi(\cdot+x)](0)$. Then it is easy to see
the following result holds.
\begin{lemma}
$\{\tilde{Q}_t\}_{t\ge 0}$ has the following properties:
\begin{enumerate}
\item[(i)]$\tilde{Q_0}[\phi]=\phi,\forall \phi\in
\mc{B}$.
\item[(ii)]$\tilde{Q}_t\circ
\tilde{Q}_s[\phi]=\tilde{Q}_{t+s}[\phi],\, \forall t,s\ge 0,\,
\phi\in \mc{B}$.
\item[(iii)]For fixed $x\in\R$, if $t_n\to
t$ and $\phi_n(i+x)\to \phi(i+x)$ in $\mc{X}$ for any
$i\in\mathbb{Z}$, then $\tilde{Q}_{t_n}[\phi_n](x)\to
\tilde{Q}_{t}[\phi](x)$ in $\mc{X}$.
\end{enumerate}
\end{lemma}

We combine the ideas in the proofs of Theorems \ref{ThDisDis} and
\ref{ThContCont} to prove the following result for continuous-time
semiflows in a discrete habitat $\Z$.
\begin{theorem}\label{ThContDis}
Let $\mc{X}=C(M, \R^d)$. Assume that for each $t>0$, the map $Q_t$
satisfies (A1), (A3)-(A5) with $E=\Sigma_t$, and that the time-one
map $Q_1$ satisfies (A6) with $E=\Sigma$. Then there exists
$c\in\mathbb{R}$ such that $\{Q_t\}_{t\ge 0}$ admits a
non-decreasing traveling wave with speed $c$ and connecting $0$ to
$\beta$.
\end{theorem}
\begin{proof}
Let $\delta,e_0,e_\beta$ be chosen as in
\eqref{delta},\eqref{delta0} and \eqref{deltabeta}. We proceeds with
three steps.

Firstly, since for any $s>0$ the map $Q_s$ satisfies assumptions
(A1)-(A5), it then follows from the proof of Theorems \ref{ThDisDis}
and \ref{ThContCont} that there exists $s_k\downarrow 0$ such that
$\{(Q_{s_k})^n\}_{n\ge 0}$ admits two nondecreasing traveling waves
$\tilde{\psi}_{\pm,s_k}(x+c_{\pm,s_k}n)$ with $c_{-,s_k}\ge
c_{+,s_k}$, that is, there exists countable subset $\Theta_k$ such
that
\begin{equation*}
\tilde{Q}_{s_k}[\tilde{\psi}_{\pm,s_k}](x)=\tilde{\psi}_{\pm,s_k}(x+c_{\pm,s_k}),\quad
\forall x\in \R\setminus \Theta_k.
\end{equation*}
Furthermore, $\tilde{\psi}_{-,s_k}$ is left continuous and
$\tilde{\psi}_{+,s_k}$ is right continuous with the following
properties:
\begin{equation*}
0<\tilde{\psi}_{-,s_k}(0)\le \delta e_0\quad \text{and}\quad
\beta-\delta e_\beta \le \tilde{\psi}_{+,s_k}(0) <\beta,
\end{equation*}
but
\begin{equation*}
\tilde{\psi}_{-,s_k}(0)\not\in [0,\delta e_0]]_{\mc{X}}\quad
\text{and}\quad \tilde{\psi}_{+,s_k}(0)\not\in[[\beta-\delta
e_\beta,\beta]_{\mc{X}}.
\end{equation*}

Secondly, we show that for the above sequence $s_k$, there exist a
countable set $\tilde{\Gamma}\subset \R$ and a subsequence, still
denoted by $s_k$, such that $\f{1}{s_k}c_{\pm,s_k}\to c_\pm\in\R$
and that $\tilde{\psi}_{\pm,s_k}(x)$ converges in $\mc{X}$ for all
$x\in\R\setminus \tilde{\Gamma}$. Indeed, let
$\Theta=\cup_{k=1}^\infty\Theta_k$. Hence, $\Theta$ is countable and
\begin{equation*}
\tilde{Q}_{s_k}[\tilde{\psi}_{\pm,s_k}](x)=\tilde{\psi}_{\pm,s_k}(x+c_{\pm,s_k}),\quad
\forall k\ge 1, x\in \R\setminus \Theta.
\end{equation*}
From Proposition \ref{SetProp2}, we see that there exists another
countably dense set $\Gamma\subset\R$ such that $\Gamma\cap
\Theta=\emptyset$. By the same arguments as in the proof of Theorem
\ref{ThDisDis}, we can show that
\begin{equation*}
\tilde{\psi}_-(x):=\lim_{y\in\Gamma,y\uparrow x}\lim_{k\to
\infty}\tilde{\psi}_{-,s_k}(y),\quad \forall x\in \R,
\end{equation*}
and
\begin{equation*}
\tilde{\psi}_+(x):=\lim_{y\in\Gamma,y\downarrow x}\lim_{k\to
\infty}\tilde{\psi}_{+,s_k}(y),\quad \forall x\in \R,
\end{equation*}
are well-defined and all $\tilde{\psi}_\pm(\pm \infty)$ exist.
Furthermore, $\tilde{\psi}_-(+\infty)$ and $\tilde{\psi}_+(-\infty)$
are ordered in $\mc{X}$ and
\begin{equation*}
\tilde{\psi}_-(-\infty)\le \psi_-(0)\le \delta e_0\quad
\text{and}\quad \tilde{\psi}_+(+\infty)\ge \tilde{\psi}_+(0)\ge
\beta-\delta e_\beta,
\end{equation*}
but
\begin{equation*}
\tilde{\psi}_{-}(0)\not\in [0,\delta e_0]]_{\mc{X}}\quad
\text{and}\quad \tilde{\psi}_{+}(0)\not\in[[\beta-\delta
e_\beta,\beta]_{\mc{X}}.
\end{equation*}
Further, $\tilde{\psi}_\pm(x^\pm)$ exist for all $x\in\R\setminus
\Gamma$. Hence, it follows from Theorem \ref{HellyTh} that there
exists a countable subset $\tilde{\Gamma}$ of $\R$ such that
\begin{equation}\label{Eq1}
\tilde{\psi}_{\pm,s_k}(x)\to \tilde{\psi}_\pm(x),\quad \forall x\in
\R\setminus \tilde{\Gamma}.
\end{equation}
By similar arguments as in the second step of the proof of Theorem
\ref{ThContCont}, we can show that $\f{1}{s_k}c_{\pm, s_k}$ are
bounded.

Finally, we prove that either $\tilde{\psi}_-(x+c_-t)$ or
$\tilde{\psi}_+(x+c_+t)$ is a nondecreasing traveling wave
connecting $0$ to $\beta$. Indeed, from \eqref{Eq1} and Proposition
\ref{SetProp1}, we see that there exists a countable subset
$\Gamma_1$ of $\R$ such that
\begin{equation*}
\tilde{\psi}_{\pm,s_k}(i+x)\to \tilde{\psi}_\pm(i+x),\quad \forall
i\in\Z, x\in \R\setminus \Gamma_1.
\end{equation*}
Hence, for any $x\in \R\setminus \Gamma_1$ and $t>0$, we have
\begin{eqnarray}\label{Eq:Limitation}
\tilde{Q}_t[\tilde{\psi}_-](x)&&=Q_t[\tilde{\psi}_-(\cdot+x)](0)=\lim_{k\to
\infty}Q_{t+r_k}[\tilde{\psi}_{-,s_k}(\cdot+x)](0)\nonumber\\
&&=\lim_{k\to \infty}Q_{m_k
s_k}[\tilde{\psi}_{-,s_k}(\cdot+x)](0)=\lim_{k\to
\infty}\tilde{Q}_{m_ks_k}[\tilde{\psi}_{-,s_k}](x)\nonumber\\
&&=\lim_{k\to\infty}
(\tilde{Q}_{s_k})^{m_k}[\tilde{\psi}_{-,s_k}](x)=\lim_{k\to
\infty} \tilde{\psi}_{-,s_k}(x+m_kc_{-,s_k})\nonumber\\
&&=\lim_{k\to \infty}
\tilde{\psi}_{-,s_k}(x+(t+r_k)\f{1}{s_k}c_{-,s_k}).
\end{eqnarray}
In the case where $c_-=0$, we can choose $x_0$ such that
\begin{equation*}
Q_t[\tilde{\psi}_-(x_0+\cdot)](i)=\lim_{k\to \infty}
\tilde{\psi}_{-,s_k}(x+(t+r_k)\f{1}{s_k}c_{-,
s_k})=\tilde{\psi}_-(x_0+i),\forall i\in\mathbb{Z}.
\end{equation*}
In the case where $c_-\neq 0$, we know that there exists a countable
subset $\Gamma_2$ of $\R$ such that
\begin{equation*}
\tilde{Q}_t[\tilde{\psi}_-](x)=\lim_{k\to \infty}
\tilde{\psi}_{-,s_k}(x+(t+r_k)\f{1}{s_k}c_{-,s_k})=\tilde{\psi}_-(x+c_-t),\quad
\forall x\not\in \Gamma_1, x+c_-t\in\Gamma_2.
\end{equation*}
Without loss of generality, we assume that $c_->0$. For any
$y\in\R$, we can choose $x_0\in\R$ and $t_0\ge 0$ such that
$x_0+c_-t_0=y$ and $\tilde{\psi}_-(x)$ is continuous at $x=x_0+i$
for all $i\in\Z$. Now one can find $x_{\pm,k}\in \R\setminus
\Gamma_1$ and $t_{\pm, k}\to t_0$ with $y_{\pm ,k}:=x_\pm
+c_-t_{\pm,k}\in \R\setminus \Gamma_2$ such that $y_{-,k}\uparrow y$
and $y_{+,k}\downarrow y$. Note that
\begin{equation*}
\tilde{\psi}_-(y^-):=\lim_{k\to
\infty}\tilde{\psi}_-(y_{-,k})=\lim_{k\to \infty}
Q_{t_{-,k}}[\tilde{\psi}_-(\cdot+x_{-,k})](0)=Q_{t_0}[\tilde{\psi}_-(\cdot+x_0)](0)
\end{equation*}
and
\begin{equation*}
\tilde{\psi}_-(y^+):=\lim_{k\to
\infty}\tilde{\psi}_-(y_{+,k})=\lim_{k\to \infty}
Q_{t_{+,k}}[\tilde{\psi}_-(\cdot+x_{+,k})](0)=Q_{t_0}[\tilde{\psi}_-(\cdot+x_0)](0).
\end{equation*}
Thus, $\tilde{\psi}_-(x)$ is continuous in $x\in \R$. And hence,
again by Proposition \ref{Convergence} and the equality
\eqref{Eq:Limitation}, we have
$\tilde{Q}_t[\tilde{\psi}_-](x)=\tilde{\psi}_-(x+c_-t)$ for all
$x\in \R$ and $t\ge 0$. Therefore, $\tilde{\psi}_-(x+c_-t)$ is a
traveling wave connecting $0$ to some $\alpha_-\in
\Sigma\setminus\{0\}$. Similarly, we can construct the traveling
wave $\tilde{\psi}_+(x+c_+t)$ connecting some $\alpha_+\in
\Sigma\setminus\{\beta\}$. Besides, $\alpha_-$ and $\alpha_+$ are
ordered. Now the rest of the proof is essentially the same as in the
proof of Theorem \ref{ThContCont}.
\end{proof}

\section{Semiflows in a periodic habitat}
A typical example of evolution systems in a periodic habitat is
\begin{equation}\label{Eq:MediaPer}
u_t=(d(x)u_x)_x+f(u),\,  t>0,x\in\R,
\end{equation}
where $d(x)$ is a positive periodic function of $x\in \R$. Under the
assumption that $f$ has exactly three ordered zeros $0<a<1$ and
$f'(0)<0, f'(a)>0,f'(1)<0$, Xin \cite{Xin} employed perturbation
methods to obtain the existence of spatially periodic traveling wave
$V(x+ct,x)$ with $V(-\infty,\cdot)=0$ and $V(+\infty,\cdot)=1$
provided that $d(x)$ is sufficiently closed to a positive constant
in certain sense (see also \cite{XinReview}). For a general positive
periodic function $d(x)$, the existence of such a traveling wave
remains open. We will revisit this problem in subsection 6.3.

A map $Q:\mc{E}\to \mc{E}\subset \mc{C}$ is said to be spatially
periodic with a positive period $r\in \mc{H}$ if $Q\circ
T_r=T_r\circ Q$, where $T_r$ is the $r$-translation operator.
Similarly, a semiflow $\{Q_t\}_{t\in \mc{T}}$ on $\mc{E}\subset
\mc{C}$ is said to be spatially periodic with a positive period
$r\in \mc{H}$ if $Q_t\circ T_r=T_r\circ Q_t$ for all $t\ge 0$.
\begin{definition}\label{TwPerMedia}
\begin{enumerate}
\item[(i)] An $r$-periodic function $\beta(x)$ is said to be an $r$-periodic
steady state of the map $Q$ (semiflow $\{Q_t\}_{t\in \mc{T}}$) if
$Q[\beta]=\beta (Q_t[\beta]=\beta,\forall t\in \mc{T})$.
\item[(ii)]$V(x+ct,x)$ is said to be a spatially $r$-periodic traveling wave with
speed $c$ of the semiflow $\{Q_t\}_{t\in \mc{T}}$ if
$Q_t[V(\cdot,\cdot)](x)=V(x+ct,x)$ and $V(\cdot,x)$ is $r$-periodic
in $x$. Besides, we say that $V(\xi,x)$ connects $0$ to $\beta(x)$
if $\lim_{\xi\to -\infty}\|V(\xi,x)\|_{\mc{X}}=0$ and $\lim_{\xi\to
+\infty}\|V(\xi,x)-\beta(x)\|_{\mc{X}}=0$ uniformly for $x\in
\mc{H}$.
\end{enumerate}
\end{definition}

Motivated by \cite[Section 5]{LiangZhaoJFA}, we can regard a
spatially periodic semiflow on $\mc{E}\subset\mc{C}$ as a spatially
homogeneous semiflow on another phase space. For any positive $h\in
\mc{H}$, define $[0,h]_{\mc{H}}:=\{l\in \mc{H}:0\le l\le h\}$. We
use $\mc{Y}$ to denote $C([0,r]_{\mc{H}},\mc{X})$ and $\mc{S}$ to
denote the set of all bounded functions from $r\Z$ to $\mc{Y}$.
Clearly, $\mc{Y}$ can be regarded as a subspace of $\mc{S}$. Let
$\mc{Y}^+=C([0,r]_{\mc{H}},\mc{X}^+)$ and $\mc{S}^+$ be the set of
all bounded functions from $r\Z$ to $\mc{Y}^+$. We equip $\mc{Y}$
with the norm $\|u\|_{\mc{Y}}=\max\{\|u(x)\|_{\mc{X}}:x\in
[0,r]_{\mc{H}}\}$ and $\mc{S}$ with the compact open topology. Thus,
$\mc{Y}$ is a Banach lattice with the norm $\|\cdot\|_{\mc{Y}}$ and
the cone $\mc{Y}^+$.

Let
\begin{equation*}
\mc{K}:=\{f\in \mc{S}:f(ri)(r)=f(r(i+1))(0),\forall i\in \Z\}.
\end{equation*}
It is easy to see that
\begin{equation*}
\mc{K}\cap \mc{Y}=\{f\in \mc{S}:f(ri)\equiv f(rj) \,  \text{and}\,
f(ri)(0)=f(ri)(r),\forall i,j\in \Z\}.
\end{equation*}
For any $\phi\in \mc{C}$, define $\tilde{\phi}\in \mc{S}$ by
\[
\tilde{\phi}(ri)(y)=\phi(ri+y), \, \, \forall i\in \Z, \, y\in
[0,r]_{\mc{H}}.
\]
Then we have the following observation.
\begin{lemma}\label{SetEquivalent}
For any $f\in \mc{K}$, there exists a unique $\phi_f\in \mc{C}$ such
that $\tilde{\phi_f}=f$. Further, if $f\in \mc{K}\cap \mc{Y}$, then
$\phi_f$ is $r$-periodic.
\end{lemma}
\begin{proof}
For any $x\in \mc{H}$, we can find $i\in \Z$ and $y\in[0,r]_\mc{H}$
such that $x=ri+y$. It is easy to see that such decomposition of $x$
is unique when $x\in \mc{H}\setminus r\Z$ and is in two possible
ways when $x\in r\Z$. More precisely, when $x\in r\Z$, it can be
decomposed into either $x=r(i+1)+0$ or $x=ri+r$ for some $i\in \Z$.
Note that $f(r(i+1))(0)=f(ri)(r)$. It then follows that
$\phi_f(x)=\phi_f(ri+y):=f(ri)(y)$ is a well-defined function in
$\mc{C}$. Clearly, $\tilde{\phi_f}=f$. If $f(ri)= u, \forall i\in
\Z$, then $\phi_f(ri)= u, \forall i\in \Z$, which implies that
$\phi_f$ is $r$-periodic.
\end{proof}

If we define $F:\mc{C}\to \mc{K}$ by $F(\phi)=\tilde{\phi}$, then
$F$ is a homeomorphism between $\mc{C}$ and $\mc{K}$. Let $\beta(x)$
be a strongly positive $r$-periodic steady state of the semiflow
$\{Q_t\}_{t\ge 0}$. With a little abuse of notation, we use
$\mc{C}_\beta$ to denote the set $\{\phi\in \mc{C}:0\le \phi\le
\beta\}$. Now we can define a semiflow $\{P_t\}_{t\in \mc{T}}$ on
$\mc{K}_{\tilde{\beta}}:=\{f\in \mc{K}:\, 0\leq  f\leq
\tilde{\beta}\}$ by
\begin{equation}\label{DefinitionOfPt}
P_t[f]=F\circ Q_t[\phi_f],\quad \forall f\in
\mc{K}_{\tilde{\beta}},\, t\in \mc{T}.
\end{equation}
Clearly, $P_t\circ F=F\circ Q_t,\forall t\in \mc{T}$, which implies
that semiflows $\{Q_t\}_{t\in \mc{T}}$ and $\{P_t\}_{t\in \mc{T}}$
are topologically conjugate. Moreover, $\{P_t\}_{t\in \mc{T}}$ is
spatially homogeneous and $\tilde{\beta}$ is its equilibrium. Thus,
we see that the semiflow $\{Q_t\}_{t\in \mc{T}}$ on $\mc{C}_\beta$
has a spatially $r$-periodic traveling wave if the semiflow
$\{P_t\}_{t\in\mc{T}}$ on $\mc{K}_{\tilde{\beta}}$ has a traveling
wave. Before stating the main result, we first introduce the
bistability assumption. Let $\beta(x)\gg 0$ be an $r$-periodic
steady state of the semiflow $\{Q_t\}_{t\in \mc{T}}$. Assume that
$0$ is a trivial steady state. Define
\begin{equation*}
\Pi_{\beta}:=\{\phi\in \mc{C}: \, \phi(x)=\phi(x+r),\, 0\le
\phi(x)\le \beta(x),\,  \forall x\in \mc{H}\}.
\end{equation*}

As in Definition \ref{DefStabFix}, we can define the strong
stability of periodic steady states for a map $Q$ in the space of
periodic functions.
\begin{definition}\label{DefStabPer}
A steady state $\alpha\in\Pi_\beta$ is said to be strongly stable
from below for the map $Q:\Pi_\beta\to \Pi_\beta$ if there exist a
positive number $\delta_\alpha^+$ and a strongly positive element
$e_\alpha^+\in \Pi_\beta$ such that
\begin{equation}\label{Eq:DefStabPer}
Q[\alpha -\eta e_\alpha^+]\gg \alpha - \eta e_\alpha^+,\,\forall
\eta\in (0,\delta_\alpha^+].
\end{equation}
The strong instability from below is defined by reversing the
inequality \eqref{Eq:DefStabPer}. Similarly, we can define strong
stability (instability) from above.
\end{definition}

We need the following bistability assumption on the spatially
$r$-periodic map $Q$.
\begin{enumerate}
\item[(A5$''$)]({it Bistability}) $0$ and $\beta\gg0$ are two strongly stable
$r$-periodic steady states from above and below, respectively, for
$Q:\Pi_\beta\to \Pi_\beta$,  and the set of all intermediate
$r$-periodic steady states are totally unordered in $\Pi_\beta$.
\end{enumerate}

We note that a sufficient condition for the non-ordering property of
all intermediate $r$-periodic steady states is: $Q:\Pi_\beta\to
\Pi_\beta$ is eventually strongly monotone and all intermediate
fixed points are strongly unstable from both above and below.

\begin{theorem}\label{ThMediaPeri}
Let $\mc{X}=C(M,\R^d)$. Assume that for any $t>0$, the map $Q_t$
satisfies (A2)-(A4) and the bistability assumption (A5$''$).
Further, assume that the map $P_1:=FQ_1F^{-1}$ satisfies assumption
(A6) with $\mc{C}$ and $\beta$ replaced by $\mc{K}$ and
$\tilde{\beta}$, respectively. Then the spatially $r$-periodic
semiflow $\{Q_t\}_{t\in \mc{T}}$ admits an $r$-periodic traveling
wave $V(x,x+ct)$. Besides, $V(x,\xi)$ is nondecreasing in $\xi$ and
connecting $0$ to $\beta(x)$.
\end{theorem}
\begin{proof}
Let $t\ge 0$ be fixed and $P_t$ be defined as in
\eqref{DefinitionOfPt}. Then it is easy to see that $P_t$ satisfies
(A1)-(A5) with $\mc{C}_\beta$ replaced by $\mc{K}_{\tilde{\beta}}$.
From Theorems \ref{ThDisDis} and \ref{ThContDis},  we see that
$\{P_t\}_{t\in \mc{T}}$ admits a traveling waves $U(x+ct)$ with $U$
connecting $0$ to $\tilde{\beta}$. By the definitions of traveling
waves in a discrete habitat (see Definitions \ref{DefDisDis} and
\ref{DefContDis}), we can find $x_0\in \R$ such that $g:=U
(\cdot+x_0)\in \mc{K}_{\tilde{\beta}}$ and
$P_t[g](ri)=U(ri+ct+x_0),\forall i\in \mathbb{Z}$. By Lemma
\ref{SetEquivalent}, we can find $\psi, h_t \in \mc{C}$ such that
$\tilde{\psi}=g$ and $\tilde{h}_t=U(\cdot+ct+x_0)$, and hence,
$P_t[\tilde{\psi}]=\tilde{h_t}$. By the topological conjugacy of
$Q_t$ and $P_t$, we have $Q_t[\psi]=h_t$. Note that
$\tilde{\psi}=g=U(\cdot+x_0)=\tilde{h_0}$. It then follows from
Lemma \ref{SetEquivalent} that $\psi=h_0$. If $c=0$, then we obtain
$Q_t[\psi]=h_t\equiv h_0=\psi$, which implies that $\psi$ is a
traveling wave with speed zero. If $c\neq 0$, then we define
$V(\xi,x):=h_{\f{\xi-x}{c}}(x)$. Consequently,
\begin{equation*}
V(x+ct,x)=h_t(x)=Q_t[\psi](x)=Q_t[h_0](x)=Q_t[V(\cdot,\cdot)](x),\forall
x\in\mc{H},t\ge 0.
\end{equation*}
This completes the proof.
\end{proof}

To finish this section, we remark that the bistability structure can
be obtained for equation \eqref{Eq:MediaPer} under appropriate
conditions so that the existence result in \cite{Xin,XinReview} is
improved (see the details in subsection 6.3). Further, Theorem
\ref{ThMediaPeri} with $\mc{H}=Z$ and $\mc{X}=\R$ can be used to
rediscover the existence result in \cite{ChenGuoWu} for one
dimensional lattice equation under the bistability assumption.

\section{Semiflows with weak compactness}

In assumption (A4) of section 2, we assume that $Q:\mc{C}_\beta\to
\mc{C}_\beta$ is compact with respect to the compact open topology.
In this section, we establish the existence of bistable waves under
some weaker compactness assumptions.

Let $\tau>0$ be a fixed number. It is well known that the time-$t$
solution map of time-delayed reaction-diffusion equations such as
\begin{equation}\label{DelayEquation}
\f{\partial u}{\partial t}=\f{\partial^2 u}{\partial
x^2}+f(u(t,x),u(t-\tau,x)),
\end{equation}
is compact with respect to the compact open topology if and only if
$t>\tau$, where the phase space $\mc{C}$ is chosen as $C(\R,
C([-\tau,0],\R))$. The first purpose of this section is to show that
our results are still valid for this kind of evolution equations by
introducing an alternative assumption (A4$'$).

In order to state this assumption, we need some notations for
time-delayed evolution systems. Let $\tau\in \mc{T}$ be a positive
number, $\mc{F}$ be a Banach lattice with the positive cone
$\mc{F}^+$ having non-empty interior, $\beta\in Int(\mc{F}^+)$, and
$\mc{X}_\beta=C([-\tau,0],\mc{F}_\beta)$. For any
$\phi\in\mc{C}_\beta$, we can regard it as an element in
$C([-\tau,0]\times\mc{H},\mc{F}^+)$. For any subset $B$ of
$[-\tau,0]\times\mc{H}$, we define $\phi|_B$ as the restriction of
$\phi$ on $B$.
\begin{enumerate}
\item[(A4$'$)]({\it Compactness}) There exists
$s\in(0,\tau]$ such that
\begin{enumerate}
\item[(i)]$Q[\phi](\theta,x)=\phi(\theta+s,x)$ whenever $\theta+s\le 0$.
\item[(ii)]For any
$\epsilon \in (0,s)$, the set $Q[\mc{C}_\beta]|_{
[-s+\epsilon,0]\times \mc{H}}$ is precompact.
\item[(iii)]For any subset $\mc{J}\subset \mc{C}_\beta$ with $\mc{J}(0,\cdot)\subset
C(\mc{H},\mc{Y}_\beta)$ being precompact, the set
$Q[\mc{J}]|_{[-s,0]\times \mc{H}}$ is precompact.
\end{enumerate}
\end{enumerate}

This assumption was motivated by \cite[Assumption
(A6$'$)]{LiangZhao}. Let us use equation \eqref{DelayEquation} to
explain (A4$'$). For any $t>\tau$, one can directly verify that the
solution map $Q_t$ satisfies (A4) by rewriting \eqref{DelayEquation}
as an integral form (see, e.g., \cite{Wu}); and for any $t\in (0,
\tau]$, one can show that $Q_t$ satisfies (A4$'$) (i) and (ii) by
the same arguments. For (A4$'$) (iii), we provide a proof below.

Let $T(0)=I$, and for any $t\in(0,\tau]$, let $T(t)$ be the time-$t$
map of the heat equation $u_t=\Delta u$. Then \eqref{DelayEquation}
can be written as the following form:
\begin{equation*}
u(t,x;\phi)=T(t)\phi(x)+\int_0^tT(t-s)f(u(s,u(s-\tau)))(x)ds,
\end{equation*}
and hence, $Q_t[\phi](\theta,x)=u(t+\theta,x)$. Note that for any
$\phi\in \mc{C}_\beta$, $T(t)\phi\to \phi$ with respect to the
compact open topology as $t\to 0$. It then follows from the
triangular inequality and the absolute continuity of integrals that
for any compact subset $\mc{H}_1\subset \R$, the set
$Q_t[\mc{J}]|_{[-t,0]\times \mc{H}_1}$ is equi-continuous, and
hence, $Q[\mc{J}]|_{[-t,0]\times \R}$ is precompact in
$\mc{C}_\beta$.

\begin{lemma}\label{DelayEqCompactness}
Let $A_\xi,\xi\ge 1$, be defined as in section 3 and $\beta\in
Int(\mc{F}^+)$. Assume that $Q:\mc{C}_\beta\to \mc{C}_\beta$
satisfies (A4$'$). Then there exists an integer $m_0$ such that
$\cup_{\xi\in [1,1+\delta]}(Q\circ
A_\xi)^{m_0}[\mc{C}_\beta]\subset\mc{C}_\beta$ is precompact when
$\mc{H}=\R$, and $\cup_{\xi\in [1,2]}(\tilde{Q}\circ
A_\xi)^{m_0}[\mc{B}_\beta](x)\subset\mc{X}_\beta$ is precompact for
any $x\in \R$ when $\mc{H}=\Z$.
\end{lemma}
\begin{proof}
We only prove the case where $\mc{H}=\R$ since the proof for
$\mc{H}=Z$ is essentially similar. Let $s$ and $\tau$ be defined in
(A4$'$). For such $s$ and $\tau$, there exists $m_0\in \mathbb{N}$
such that $s\in(\f{1}{m_0+1}\tau,\f{1}{m_0}\tau]$. By assumption
(A4$'$)(i), we see that for any $\xi\ge 1$ and $\phi_0\in
\mc{C}_\beta$,
\begin{eqnarray*}
\phi_1^\xi(\theta,x):=Q\circ A_\xi[\phi_0](\theta,x)=
\begin{cases}
\phi_0(\theta+s,\xi x),& \theta+s\le 0\\
Q[\phi_0(\xi\cdot)](\theta,x), &\theta+s>0,
\end{cases}
\end{eqnarray*}
This implies that for any $\xi\ge 1$ and $\epsilon<
s-\f{1}{m_0+1}\tau$,
\begin{equation*}
\cup_{\xi\in[1,2]}Q\circ A_\xi[\mc{C}_\beta]|_{
[-s+\epsilon,0]\times \R} \subset Q[\mc{C}_\beta]|_{
[-s+\epsilon,0]\times \R}.
\end{equation*}
Since $Q[\mc{C}_\beta]|_{ [-s+\epsilon,0]\times \R}$ is precompact,
as assumed in (A4$'$(ii)), it then follows that
$\cup_{\xi\in[1,2]}Q\circ A_\xi[\mc{C}_\beta](0,\cdot)\subset
C(\mathbb{R},\mc{Y}_\beta)$ is precompact. By (A4$'$)(iii) and
similar arguments as above, we have
\begin{eqnarray*}
\phi_2^{\xi}(\theta,x)&:=&Q\circ A_\xi[\phi_1^\xi](\theta,x)=
\begin{cases}
\phi_1^\xi(\theta+s,\xi x),& \theta+s\le 0\\
Q[\phi_1^\xi(\xi\cdot)](\theta,x), &\theta+s>0
\end{cases}
\nonumber\\
&=&
\begin{cases}
\phi_0(\theta+2s,\xi^2x),& \theta+2s\le 0\\
Q[\phi_0(\xi\cdot)](\theta+s,\xi x), &0< \theta+2s\le s\\
Q[\phi_1^\xi(\xi\cdot)](\theta,x), &\theta+s>0,
\end{cases}
\end{eqnarray*}
This implies that $\cup_{\xi\in[1,2]}(Q\circ
A_\xi)^2[\mc{C}_\beta]|_{ [-2s+\epsilon,0]\times \mathbb{R}}$ is
precompact. Consequently, $\cup_{\xi\in[1,2]}(Q\circ
A_\xi)^2[\mc{C}_\beta](0,\cdot)\subset C(\mathbb{R},\mc{Y}_\beta)$
is compact. By induction, we have
\begin{eqnarray*}
\phi_{m_0+1}^{\xi}(\theta,x)&&:=Q\circ
A[\phi_{m_0}^{\xi}](\theta,x)=
\begin{cases}
\phi_{m_0}^{\xi}(\theta+s,\xi x),& \theta+s\le 0\\
Q[\phi_{m_0}^{\xi}(\xi\cdot)](\theta,x), &\theta+s>0
\end{cases}
\nonumber\\
&&=\cdot\cdot\cdot\nonumber\\
&&=
\begin{cases}
Q[\phi_0(\xi\cdot)](\theta+(m_0+1)s,\xi^{m_0} x), &0< \theta+(m_0+1)s\le s\\
Q[\phi_1^\xi(\xi\cdot)](\theta+m_0s,\xi^{m_0-1}x), &0<\theta+m_0s\le s\\
\cdot\cdot\cdot\\
Q[\phi_{m_0-1}^{\xi}(\xi\cdot)](\theta+s,\xi x),&0<\theta+s\le s\\
Q[\phi_{m_0}^{\xi}](\theta,x),&\theta+s>0.
\end{cases}
\end{eqnarray*}
This implies that $\cup_{\xi\in[1,2]}(Q\circ
A_\xi)^{m_0}[\mc{C}_\beta]$ is precompact in $\mc{C}_\beta$.
\end{proof}

\begin{theorem}\label{Delay}
All results in Theorems \ref{ThDisCont}-\ref{ThContDis} and
\ref{ThMediaPeri} are valid if we replace (A4) with (A4$'$).
\end{theorem}
\begin{proof}
Following the proof of theses theorems, we only need to modify the
parts where we use the compactness assumption (A4). At these parts,
by Lemma \ref{DelayEqCompactness} we can easily complete the proof.
\end{proof}

Note that the solution maps of the integro-differential equation
\begin{equation*}
u_t=J* u-u+f(u)
\end{equation*}
satisfy neither (A4) nor (A4$'$). The second purpose of this section
is to modify our developed theory in such a way that it applies to
these integro-differential systems.

Let $\mc{M}$ denote the set of all nondecreasing functions from
$\mathbb{R}$ to $\mc{X}$ and $\beta\in \mc{X}^+$. We equip $\mc{M}$
with the compact open topology. Assume that $Q$ maps $\mc{M}_\beta$
to $\mc{M}_\beta$. Let $E$ denote the set of fixed point of $Q$
restricted on $\mc{X}_\beta$. Suppose that $0$ and $\beta$ are in
$E$. We impose the following assumptions on $Q$:
\begin{enumerate}
\item[(B1)]({\it Translation Invariance}) $T_y\circ Q [\phi]=Q\circ T_y [\phi],
\forall \phi\in\mc{M}_\beta,y\in\R$.
\item[(B2)]({\it
Continuity}) $Q:\mc{M}_\beta\to \mc{M}_\beta$ is continuous in the
sense that if $\phi_n \to \phi$ in $\mc{M}_\beta$, then
$Q[\phi_n](x)\to Q[\phi](x)$ in $\mc{X}_\beta$ for almost all $x\in
\R$.
\item[(B3)]({\it Monotonicity}) $Q$ is order preserving in the sense that $Q[\phi]\ge
Q[\psi]$ whenever $\phi\ge \psi$ in $\mc{M}_\beta$.
\item[(B4)]({\it Weak Compactness}) For any fixed $x\in \R$, the set
$Q[\mc{M}_\beta](x)$ is precompact in $\mc{X}_\beta$.
\item[(B5)]({\it
Bistability}) Fixed points $0$ and $\beta$ are strongly stable from
above and below, respectively, for the map $Q:\mc{X}_\beta\to
\mc{X}_\beta$, and the set $E\setminus\{0,\beta\}\subset
\mc{X}_\beta$ is totally unordered.
\item[(B6)] ({\it Counter-propagation}) For each $\alpha\in E\setminus\{0,\beta\}$,
$c_-^*(\alpha,\beta)+c_+^*(0,\alpha)>0$.
\end{enumerate}

Comparing assumptions (A1)-(A6) and (B1)-(B6), one can find that the
assumptions of translation invariance, monotonicity, bistability and
counter-propagation are the same. The difference lies in the
assumptions of continuity and compactness. Clearly, compactness
assumption (B4) is much weaker than (A4).

\begin{theorem}\label{DisWeakComp}
Let $\mc{X}=C(M, \R^d)$ and assume that $Q:\mc{M}_\beta\to
\mc{M}_\beta$ satisfies (B1)-(B6). Then there exists $c\in \R$ and
$\psi\in \mc{M}_\beta$ connecting $0$ to $\beta$ such that
$Q[\psi](x)=\psi(x+c)$ for all $x\in \R$.
\end{theorem}
\begin{proof}
Combining the proofs of Theorems \ref{ThDisCont} and \ref{ThDisDis},
we can obtain the result. More precisely, one can repeat the proof
of Theorem \ref{ThDisCont} except for the parts where the
compactness assumption (A4) are used. For these parts, one use the
idea in Theorem \ref{ThDisDis}, where $\tilde{Q}$ has the same
compactness property as $Q$.
\end{proof}

In the rest of this section, we say $\{Q_t\}_{t\ge 0}$ is a semiflow
on $\mc{M}_\beta$ provided that $Q_0=I$; $Q_t\circ Q_s=Q_{t+s},
\forall t,s>0$; and $Q_{t_n}[\phi_n](x)\to Q_t[\phi](x)$ in
$\mc{X}_\beta$ for almost all $x\in \R$ whenever $t_n\to t$ and
$\phi_n\to \phi$ in $\mc{M}_\beta$.

\begin{theorem}\label{ContWeakComp}
Let $\mc{X}=C(M, \R^d)$. Assume that $\{Q_t\}_{t\ge 0}$ is a
semiflow on $\mc{M}_\beta$, and for any $t>0$, the map $Q_t$
satisfies (B1) and (B3)-(B6). Then there exist $c\in \R$ and
$\psi\in \mc{M}_\beta$ connecting $0$ to $\beta$ such that
$Q_t[\psi](x)=\psi(x+ct)$ for all $x\in \R$.
\end{theorem}
\begin{proof}
As in the proof of Theorem \ref{DisWeakComp}, we can prove the
conclusion by combing the proofs of Theorems \ref{ThContCont} and
\ref{ThContDis}.
\end{proof}

Similarly, we can define $\omega$-time periodic semiflows on
$\mc{M}_\beta$ and then obtain the following result.
\begin{theorem}
Let $\mc{X}=C(M, \R^d)$. Assume that $\{Q_t\}_{t\ge 0}$ is an
$\omega$-time periodic semiflow on $\mc{M}_\beta$. Let $\beta(t)$ be
a strongly positive periodic solution of $\{Q_t\}_{t\ge 0}$
restricted on $\mc{X}_\beta$. Further, assume that the Poincar\'e
map $Q_\omega$ satisfies (B1) and (B3)-(B6) with $\beta=\beta(0)$.
Then there exist $c\in \R$ and $\phi(t,x)$ with $\phi(t,-\infty)=0$
and $\phi(t,+\infty)=\beta(t)$ such that $Q_t[\psi](x)=\psi(t,x+ct)$
for all $x\in \R$. Besides, $\phi(t,\cdot)\in \mc{M}_\beta$ and
$\phi(t,\cdot)$ is $\omega$-periodic in $t\ge 0$.
\end{theorem}

\section{Applications}
In this section, we apply the obtained abstract results to four
kinds of monotone evolution systems: a time-periodic
reaction-diffusion system, a parabolic system in a cylinder, a
parabolic equation with variable diffusion, and a nonlocal and
time-delayed reaction-diffusion equation.

\subsection{A time-periodic reaction-diffusion system}
Consider the time-periodic reaction-diffusion system
\begin{equation}\label{Application1}
\f{\partial u}{\partial t}=A\Delta u+f(t,u),\quad x\in\mathbb{R},
\end{equation}
where $u=(u_1,\cdot\cdot\cdot,u_n)^T$,
$A=diag\{d_1,\cdot\cdot\cdot,d_n\}$ with each $d_i>0$ and
$f=(f_1,\cdot\cdot\cdot,f_n)^T$ is $\omega$-periodic in $t\ge 0$
(i.e., $f(t,\cdot)=f(t+\omega,\cdot)$). The existence of periodic
bistable traveling waves of \eqref{Application1} with $n=1$ was
proved in \cite{AlikakosBatesChen}. Here we generalize this result
to the case $n\ge 1$.

Let $f\in C^1(\mathbb{R}_+\times\mathbb{R}^n,\mathbb{R}^n)$. In
order to apply Theorem \ref{ThTimePeri} to system
\eqref{Application1}, we choose
$\mc{C}:=C(\mathbb{R},\mathbb{R}^n),\mc{X}:=\mathbb{R}^n$, and
$\mc{E}\subset \mc{C}$ to be the set of all bounded functions from
$\R$ to $\R^n$. Using the solution maps $\{T(t)\}_{t\ge 0}$ of the
heat equation $\f{\partial u}{\partial t}=A \Delta u$, we write
\eqref{Application1} as the following integral form:
\begin{equation}\label{IntegralFormAppl1}
u(t;\phi)=T(t)\phi+\int_{0}^t T(t-s) f(s,u(s;\phi))ds.
\end{equation}
Define $Q_t[\phi]:=u(t;\phi),\forall \phi\in \mc{E}$. Let $0$ and
$\beta\gg 0$ be two fixed points of the Poincar\'e map $Q_\omega$ in
$\mc{X}$, and let $E$ be the set of all spatially homogeneous fixed
points of $Q_\omega$ in $\mc{X}_\beta$. We impose the following
assumptions:
\begin{enumerate}
\item[(C1)]The Jacobian matrix $\D_uf(t,u)$ is cooperative and
irreducible for all $t\ge 0$ and $u\ge 0$.
\item[(C2)]The spatially homogeneous system $u'=f(t,u)$ is of bistable type, that is,
$0$ and $\beta$ are two stable fixed points of $Q_\omega$ in the
sense that $s(\f{d}{du}Q_\omega[0])<0$ and
$s(\f{d}{du}Q_\omega[\beta])<0$, and any $\alpha\in E\setminus
\{0,\beta\}$ is a unstable in the sense that $s(\f{d}{du}
Q_\omega[\alpha])>0$, where $s(M)$ is the stability modulus of the
matrix $M$ defined by $s(M)=\max\{{\rm Re}\lambda:\text{$\lambda$ is
an eigenvalue}\}$.
\end{enumerate}

\begin{theorem}
Assume that (C1)-(C2) hold, and let $\beta (t)$ be the periodic
solution of $u'=f(t,u)$ with $\beta (0)=\beta$. Then there exists
$c\in\mathbb{R}$ such that \eqref{Application1} admits a
time-periodic traveling wave $U(t,x+ct)$ connecting $0$ to
$\beta(t)$.
\end{theorem}
\begin{proof}
It is easy to see that the discrete semiflow $\{Q_\omega^n\}_{n\ge
1}$ on $\mc{C}_\beta$ satisfies (A1)-(A5) with $Q=Q_\omega$. Next we
show that (A6) holds with $Q=Q_\omega$.

Note that for any $\alpha\in E\setminus \{0,\beta\}$,
$\{Q_\omega^n\}_{n\ge 1}:[\alpha,\beta]_{\mc{C}}\to
[\alpha,\beta]_{\mc{C}}$ performs a monostable dynamics, where
$\alpha$ is unstable and $\beta$ is stable. By the theory developed
in \cite{LiangZhao}, it follows that $Q_\omega$ admits leftward and
rightward spreading speeds $c_-^*(\alpha,\beta)$ and
$c_+^*(\alpha,\beta)$. Since $Q_\omega$ is reflectively invariant,
we further have
$c_-^*(\alpha,\beta)=c_+^*(\alpha,\beta):=c^*(\alpha,\beta)$, which
is called the spreading speed of this monostable subsystem. Note
that $\{Q_\omega^n\}_{n\ge 1}:[0,\alpha]_{\mc{C}}\to
[0,\alpha]_{\mc{C}}$ also performs a monostable dynamics, where $0$
is stable and $\alpha$ is unstable. Similarly, this monostable
subsystem also admits a spreading speed $c^*(0,\alpha)$. Let $M_t$
be the solution map of the linearized system of \eqref{Application1}
at the periodic solution $\alpha(t):=Q_t[\alpha]$:
\begin{equation}\label{linearizedApp1}
\f{\partial u}{\partial t}=A\Delta u+\D_uf(t,\alpha(t))u.
\end{equation}
By a similar argument as in the proof of \cite[Lemma
4.1]{WeinbLewisLi}, we see that for each $t>0$, there exists a
strongly positive vector $\eta\in \mathbb{R}^n$ such that
\begin{equation}\label{LinearContro1}
Q_t[u]\ge M_t[u]\quad \text{whenever}\quad u\in
[\alpha,\alpha+\eta]_{\mc{C}}
\end{equation}
and
\begin{equation}\label{LinearContro2}
Q_t[u]\le M_t[u]\quad \text{whenever}\quad u\in
[\alpha-\eta,\alpha]_{\mc{C}}.
\end{equation}
Let $\rho(\mu)$ be the principle Floquet multiplier of the following
linear periodic cooperative and irreducible system
\begin{equation}\label{HomoLinearizedApp1}
\f{dv}{dt}=[\mu^2A+\D_u f(t,\alpha(t))]v.
\end{equation}
Let $v(t,w)$ be the solution of \eqref{HomoLinearizedApp1}
satisfying $v(0,w)=w\in\mathbb{R}^n$. It is easy to see that
$u(t,x)=e^{-\mu x}v(t,w)$ is the solution of linear periodic system
\eqref{linearizedApp1}. Define $\Phi(\mu):=\ln \rho(\mu)/\mu$. From
\cite[Theorem 3.10]{LiangZhao} and inequalities
\eqref{LinearContro1}-\eqref{LinearContro2}, we then have
\begin{equation}\label{upperboundApp1}
c^*(\alpha,\beta)\ge \inf_{\mu>0} \Phi(\mu)\quad \text{and}\quad
c^*(0,\alpha)\ge \inf_{\mu>0} \Phi(\mu).
\end{equation}
Now we prove that $\Phi(+\infty)=+\infty$. Let
$\lambda(\mu)=\f{1}{\omega}\ln \rho(\mu)$. By the Floquet theory, it
then follows that there exists a positive $\omega$-periodic function
$\xi(t):=(\xi_1(t),\cdot\cdot\cdot,\xi_n(t))^T$ such that
$v(t):=e^{\lambda(\mu)t}\xi(t)$ is a solution of
\eqref{linearizedApp1}. In particular, we have
\begin{equation*}
\xi_1'(t)=(\mu^2-\lambda(\mu))\xi_1(t)+\sum_{i=1}^n\f{\partial}{\partial
u_i}f_1(t,\alpha(t))\xi_i(t).
\end{equation*}
Dividing $\xi_1(t)$ in both sides and integrating the above equality
from $0$ to $\omega$ gives
\begin{equation*}
0=(\mu^2-\lambda(\mu))\omega+\int_0^\omega
\sum_{i=1}^n\f{\partial}{\partial u_i}f_1(t,\alpha(t))\xi_i(t)/\xi_1
(t)dt,\quad \forall \mu>0.
\end{equation*}
Since the matrix $\D_uf(t,\alpha(t))$ is cooperative and $\xi(t)$ is
positive, we obtain
\begin{equation*}
0\ge (\mu^2-\lambda(\mu))\omega +\int_0^\omega \f{\partial}{\partial
u_1}f_1(t,\alpha(t)dt.
\end{equation*}
This implies that
\begin{equation*}
\Phi(\mu)= \f{\omega\lambda(\mu)}{\mu}  \ge \mu \omega
+\int_0^\omega \f{\partial}{\partial u_1}f_1(t,\alpha(t)dt,
\end{equation*}
and hence, $\Phi(+\infty)=+\infty$. By \cite[Lemma 3.8]{LiangZhao},
we then have $\inf_{\mu>0} \Phi(\mu)>0$. Thus, the assumption (A6)
with $Q=Q_\omega$ holds. Consequently, Theorem \ref{ThTimePeri}
completes the proof.
\end{proof}

\subsection{A reaction-diffusion-advection system in a cylinder}
In this subsection, we consider the following system
\begin{equation}\label{Application2}
\begin{cases}
\f{\partial u}{\partial t}=A \f{\partial^2 u}{\partial x^2}+ B
\Delta_y u+E(y)\f{\partial u}{\partial
x}+f(u), & x\in\mathbb{R}, y\in\Omega\subset \mathbb{R}^{m-1},t>0, \\
\f{\partial u}{\partial \nu}=0, & \text{on}\,
(0,+\infty)\times\mathbb{R}\times
\partial \Omega,
\end{cases}
\end{equation}
where $A,B$ are positively definite diagonal $n\times n$ matrix, $E$
is diagonal matrix of smooth functions of $y$, $\Omega$ is a bounded
and convex open subset in $\mathbb{R}^{m-1}$ with smooth boundary
$\partial \Omega$, $\Delta_y=\sum_{i=1}^{m-1}\partial^2/\partial
y_i^2$, and $\nu$ is the outer unit normal vector to
$\partial\Omega\times \mathbb{R}$.

The existence of bistable traveling waves for \eqref{Application2}
with $n=1$ was obtained in \cite{BereNiren}. Here we extend this
result to the case $n\ge 2$. Assume that $f\in
C^1(\mathbb{R}^n,\mathbb{R}^n)$ satisfies the following two
conditions:
\begin{enumerate}
\item[(D1)]The Jacobian matrix $\D f(u)$ is cooperative
and irreducible for all $u\ge 0$.
\item[(D2)]$f$ is of bistable type in the sense that it has exactly
three ordered zeros: $0< a< \beta$ and
$s(Df(0))<0,s(Df(a))>0,s(Df(\beta))<0$.
\end{enumerate}

\begin{theorem}
Assume that (D1)-(D2) hold. Then there exists $c\in\mathbb{R}$ such
that system \eqref{Application2} admits a traveling wave connecting
$0$ to $\beta$ with speed $c$.
\end{theorem}

\begin{proof}
In order to employ Theorem \ref{ThContCont}, we choose
$\mc{X}:=C(\bar{\Omega},\mathbb{R}^n)$ and
$\mc{C}:=C(\mathbb{R},\mc{X})$ with the standard cones $\mc{X}^+$
and $\mc{C}^+$, respectively. Let $G(t,x,y,w)$ be the Green function
of the linear equation
\begin{equation}\label{DefineGreenFun}
\begin{cases}
\f{\partial u}{\partial t}=A \f{\partial^2 u}{\partial x^2}+ B
\Delta_y u+E(y)\f{\partial u}{\partial
x}, & x\in\mathbb{R}, y\in\Omega,t>0, \\
\f{\partial u}{\partial \nu}=0, & \text{on}\,
(0,+\infty)\times\mathbb{R}\times
\partial \Omega.
\end{cases}
\end{equation}
Then the solution of \eqref{DefineGreenFun} with initial value
$u(0,\cdot)=\phi(\cdot)\in \mc{C}$ can be expressed as
\begin{equation*}
u(t,x,y;\phi)=\int_{\mathbb{R}}\int_{\Omega}
G(t,x-z,y,w)\phi(z,w)dwdz.
\end{equation*}
Define $T(t)\phi=u(t,\cdot;\phi),\forall \phi\in \mc{C}_\beta$.
Using the constant variation formula, we write \eqref{Application2}
subject to $u(0,\cdot)=\phi(\cdot)\in \mc{C}_\beta$ as an integral
equation
\begin{equation}\label{IntegralFormAppl2}
u(t,x,y;\phi)=T(t)[\phi](x,y)+\int_{0}^t T(t-s)f(u(s,x,y))ds.
\end{equation}
By the linear operators theory, we see that for any $\phi\in
\mc{C}_\beta$, system \eqref{Application2} has a unique solution
$u(t;\phi)$ with $u(0;\phi)=\phi$, which exists globally on
$[0,+\infty)$. Define $Q_t[\phi]:=u(t,\phi)$. Then $\{Q_t\}_{t\ge
0}$ is a subhomogeneous semiflow on $\mc{C}_\beta$ (see
\cite[Section 5.3]{LiangZhao}). Also, assumption (D1) assures that
the semiflow $\{Q_t\}_{t\ge 0}$ restricted on $\mc{X}_\beta$ is
strongly monotone (see \cite{Smith}). Further, it is easy to see
that $Q_t,t\ge 0$, satisfies assumption (A1)-(A4). Since the domain
$\Omega$ is convex, it follows from the result in \cite{Wein1985}
that any non-constant steady state of the $x$-independent system
\begin{equation*}
\begin{cases}
\f{\partial u}{\partial t}= B \Delta_y u+f(u), & y\in,t>0, \\
\f{\partial u}{\partial \nu}=0, & \text{on}\, (0,+\infty)\times
\partial \Omega
\end{cases}
\end{equation*}
is linearly unstable. This then implies that  $Q_t$ satisfies
(A5$'$). Now it remains to show that (A6) holds for $Q_1$.

For each $x$-independent steady state $\alpha=\alpha(y)$ in
$[[0,\beta]]_{\mc{X}}$, system \eqref{Application2} performs a
monostable dynamics on $[\alpha,\beta]_{\mc{C}}$. To better
understand the dynamics of this subsystem, we make a transform
$g(u,y):=f(u+\alpha(y))$. Then its dynamics is equivalent to that of
the following system on $[0,\beta-\alpha]_{\mc{C}}$:
\begin{equation}\label{Monostable1}
\begin{cases}
\f{\partial u}{\partial t}=A \f{\partial^2 u}{\partial x^2}+ B
\Delta_y u+E(y)\f{\partial u}{\partial
x}+g(u,y), & x\in\mathbb{R}, y\in\Omega,t>0, \\
\f{\partial u}{\partial \nu}=0, &
\text{on}\,(0,+\infty)\times\mathbb{R}\times
\partial \Omega.
\end{cases}
\end{equation}
System \eqref{Monostable1} has exactly two $x$-independent steady
state $S_1:=0$ and $S_2:=\beta-\alpha\gg 0$. By the theory developed
in \cite{LiangZhaoJFA}, it follows that \eqref{Monostable1} has a
leftward spreading speed $\hat{c}_-^*$ in a strong sense. Let
$c_-^*(\alpha,\beta)$ be defined as in \eqref{defcmiunsstar} with
$Q=Q_1$. We then have $c_-^*(\alpha,\beta)\geq \hat{c}_-^*$.

To verify (A6) for $Q_1$, we first estimate the speed $\hat{c}_-^*$.
Consider the linearized system of \eqref{Monostable1} at equilibrium
$S_1$:
\begin{equation}\label{Linearized1}
\begin{cases}
\f{\partial u}{\partial t}=A \f{\partial^2 u}{\partial x^2}+ B
\Delta_y u+E(y)\f{\partial u}{\partial
x}+\f{\partial g(0,y)}{\partial u}u, & x\in\mathbb{R}, y\in\Omega,t>0, \\
\f{\partial u}{\partial \nu}=0, &
\text{on}\,(0,+\infty)\times\mathbb{R}\times
\partial \Omega.
\end{cases}
\end{equation}
Suppose $u(t,x,y):=e^{\mu x} \eta(t,y)$ is a solution of
\eqref{Linearized1}, then $\eta(t,y)$ satisfies the
$\mu$-parameterized linear parabolic equation
\begin{equation}\label{MuEq1}
\begin{cases}
\f{\partial v}{\partial t}=B \Delta_y u+[\mu^2 A+\mu
E(y)+\f{\partial g(0,y)}{\partial u}]v, & y\in\Omega,t>0, \\
\f{\partial v}{\partial \nu}=0, & \text{on}\,(0,+\infty)\times
\partial \Omega.
\end{cases}
\end{equation}
Let $\lambda^+(\mu)$ be the principle eigenvalue of the elliptic
problem:
\begin{equation}\label{MuEigenvalue+}
\begin{cases}
\lambda v=B \Delta_y v+[\mu^2 A+\mu
E(y)+\f{\partial g(0,y)}{\partial u}]v, & y\in\Omega, \\
\f{\partial v}{\partial \nu}=0, & \text{on}\,\partial \Omega.
\end{cases}
\end{equation}
By the theory in \cite[Section 3]{LiangZhao}, it follows that
$\hat{c}_-^*\ge \inf_{\mu>0} \f{\lambda^+(\mu)}{\mu}$, and
$\lambda^+(\mu)$ is convex. Then it is easy to see from
\eqref{MuEigenvalue+} that $\lim_{\mu\to
+\infty}\f{\lambda^+(\mu)}{\mu}=+\infty$ and $\lim_{\mu\to
0^+}\f{\lambda^+(\mu)}{\mu}=+\infty$, and hence,
$\f{\lambda^+(\mu)}{\mu}$ attains its infimum at some
$\mu_1\in(0,+\infty)$.

Similarly, system \eqref{Application2} performs a monostable
dynamics on $[0,\alpha]_{\mc{C}}$. To better understand the dynamics
of this subsystem, we make a transform $h(u,y):=-f(\alpha(y)-u)$.
Then its dynamics is equivalent to that of the following system on
$[0,\beta-\alpha]_{\mc{C}}$:
\begin{equation*}
\begin{cases}
\f{\partial u}{\partial t}=A \f{\partial^2 u}{\partial x^2}+ B
\Delta_y u+E(y)\f{\partial u}{\partial
x}+h(u,y), & x\in\mathbb{R}, y\in\Omega,t>0, \\
\f{\partial u}{\partial \nu}=0, &
\text{on}\,(0,+\infty)\times\mathbb{R}\times
\partial \Omega.
\end{cases}
\end{equation*}
By the same arguments, such system have a rightward spreading speed
$\hat{c}_+^*$, and  we have $\hat{c}_+^*\geq c_+^*(0,\alpha)$. Also,
by the same procedure as above,we define $\lambda^-(\mu)$ as the
principle eigenvalue of the following elliptic problem:
\begin{equation}\label{MuEigenvalue-}
\begin{cases}
\lambda v=B \Delta_y v+[\mu^2 A+\mu
E(y)+\f{\partial g(0,y)}{\partial u}]v, & y\in\Omega, \\
\f{\partial v}{\partial \nu}=0, & \text{on}\,\partial \Omega.
\end{cases}
\end{equation}
It then follows that $\hat{c}_+^*\ge \inf_{\mu>0}
\f{\lambda^-(\mu)}{\mu}$, and $\f{\lambda^-(\mu)}{\mu}$ attains its
infimum at some $\mu_2\in(0,+\infty)$. Clearly,
$\lambda^-(\mu)=\lambda^+(-\mu)$.

From assumption (D2), we see that $S_1\,(S_2)$ is a linearly
unstable \, (stable) steady state of the $x$-independent system
\begin{equation}\label{xIndependent1}
\begin{cases}
\f{\partial u}{\partial t}=B\Delta_y u+g(u,y),& y\in\Omega,t>0,\\
\f{\partial u}{\partial \nu}=0, & \text{on}\,(0,+\infty)\times
\partial \Omega.
\end{cases}
\end{equation}
More precisely, letting $\lambda_0$ be the principle eigenvalue of
the following elliptic problem:
\begin{equation}\label{EigenvalueAtAlpha}
\begin{cases}
\lambda u=B\Delta_y u+u\f{\partial g}{\partial u}(0,y),& y\in\Omega,\\
\f{\partial u}{\partial \nu}=0, & \text{on}\,
\partial \Omega,
\end{cases}
\end{equation}
then $\lambda_0>0$. Obviously, equations \eqref{MuEigenvalue+} and
\eqref{MuEigenvalue-} with $\mu=0$ both become equation
\eqref{EigenvalueAtAlpha}, and hence, $\lambda^+
(0)=\lambda_0=\lambda^-(0)>0$.

With the information above, now we can show that $Q_1$ satisfies
(A6). Let $\theta=\f{\mu_2}{\mu_1+\mu_2}$. Note that $\theta
\mu_1+(1-\theta)(-\mu_2)=0$. It then follows that
\begin{eqnarray}\label{VeriAppl2}
c_-^*(\alpha,\beta)+c_+^*(0,\alpha)
&& \ge \f{\lambda^+(\mu_1)}{\mu_1}+\f{\lambda^+(-\mu_2)}{\mu_2}\nonumber\\
&& =\f{\mu_1+\mu_2}{\mu_1\mu_2}[\theta
\lambda^+(\mu_1)+(1-\theta)\lambda^+(-\mu_2)]\nonumber\\
&& \ge \f{\mu_1+\mu_2}{\mu_1\mu_2} \lambda^+(\theta
\mu_1+(1-\theta)(-\mu_2))\nonumber\\
&& = \f{\mu_1+\mu_2}{\mu_1\mu_2}
\lambda^+(0)=\f{\mu_1+\mu_2}{\mu_1\mu_2} \lambda_0>0.
\end{eqnarray}
Consequently, Theorem \ref{ThContCont} completes the proof.
\end{proof}

\subsection{A parabolic equation with periodic diffusion}

In this subsection, we study the existence of spatially periodic
traveling waves of the parabolic equation
\begin{equation}\label{Application4}
u_t=(d(x)u_x)_x+f(u),\, \,  t>0,x\in\R,
\end{equation}
where $f(u)=u(1-u)(u-a),a\in(0,1)$, and $d(x)$ is a positive,
$C^1$-continuous, and $r$-periodic function on $\R$ for some real
number $r>0$.

For any $\phi\in C(\R,[0,1])$, equation \eqref{Application4} admits
a unique solution $u(t;\phi)$ with $u(0;\phi)=\phi$. Define
$Q_t:C(\R,[0,1])\to C(\R,[0,1])$ by $Q_t[\phi]=u(t;\phi)$. It then
follows that $\{Q_t\}_{t\ge 0}$ is a continuous, compact and
monotone semiflow on $C(\R,[0,1])$ equipped with the compact open
topology. Let $C_{per}(\R,[0,1])$ be the set of all continuous and
$r$-periodic functions from $\R$ to $[0,1]$. Then the semiflow
$\{Q_t\}_{t\ge 0}$ restricted on $C_{per}(\R,[0,1])$ is strongly
monotone. Choosing $\mc{H}=\R$ and $\mc{X}=\R$ in Theorem
\ref{ThMediaPeri}, one can easily verify that $\{Q_t\}_{t\ge 0}$
satisfies assumptions (A2)-(A4). If \eqref{Application4} admits the
bistability structure, then Proposition \ref{BistabilitySemiflow}
implies (A5$''$) and a similar argument as in the previous section
shows that (A6) also holds. Thus, we focus on finding sufficient
conditions on $d(x)$ under which \eqref{Application4} admits the
bistability structure.

Let $\bar{u}$ be an $r$-periodic steady state of
\eqref{Application4}. As in \cite{BerHamelRoq}, we define
$\lambda_1(\bar{u},d)$ as the largest number such that there exists
a function $\phi>0$ which satisfies
\begin{equation}\label{CharacteristicEq}
\begin{cases}
(d\phi_x)_x+f'(\bar{u})\phi=\lambda_1(\bar{u},d)\phi,\,x\in \R\\
\text{$\phi$ is $r$-periodic and $\|\phi\|_{\infty}=1$.}
\end{cases}
\end{equation}
We call $\lambda_1(\bar{u},d)$ the principle eigenvalue of
$\bar{u}$, and $\phi$ the corresponding eigenfunction. We say
$\bar{u}$ is linearly unstable if $\lambda_1(\bar{u},d)>0$, and
linearly stable if $\lambda_1(\bar{u},d)<0$. Define
\begin{equation*}
C_{per}^1:=\{\psi\in C^1(\R,\R):\, \, \psi(x)=\psi(x+r),\,\forall
x\in \R\}
\end{equation*}
with the $C^0$-norm induced topology. We say $\psi\in C_{per}^1$ has
the property (P) if every possible non-constant $r$-periodic steady
state of \eqref{Application4} with $d=\psi$ is linearly unstable,
that is, if the equation \eqref{Application4} with $d=\psi$ does not
admit any non-constant $r$-periodic steady state $\bar{u}$ such that
$\lambda_1(\bar{u},\psi)\le 0$. Define
\begin{equation*}
Y:=\{\psi\in C_{per}^1:\, \,  \psi(x)>0\, \text{and}\,\text{$\psi$
has the property (P)}\}.
\end{equation*}

\begin{lemma}\label{ConstantInY}
Any positive constant function is in $Y$.
\end{lemma}
\begin{proof}
Let $d(x)\equiv \bar{d}$ be given. If \eqref{Application4} has no
non-constant $r$-periodic steady state, we are done. Let $\bar{u}$
be a non-constant $r$-periodic steady state of \eqref{Application4}.
We need to prove $\lambda_1(\bar{u},\bar{d})> 0$. Assume, for the
sake of contradiction, that $\lambda_1(\bar{u},\bar{d})\le 0$. Let
$\phi$ be the positive eigenfunction associated with
$\lambda_1(\bar{u},\bar{d})$. Define $M:=\max_{0\le x\le
r}\{\f{|\bar{u}_x|}{\phi}\}$ and $\psi(x,t):=e^{-\gamma
t}(\f{|\bar{u}_x|^2}{\phi}-M^2\phi)$. It is easy to see that
$\psi(t,x)\le 0$ for all $x$ and $t$. Let
$\xi:=\f{|\bar{u}_x|^2}{\phi}$ and $\eta:=M^2\phi$. Then we have
\begin{equation*}
\xi_x=(|\bar{u}_x|^2\phi^{-1})_{x}=2\bar{u}_x
\bar{u}_{xx}\phi^{-1}-|\bar{u}_x|^2\phi^{-2}\phi_x,
\end{equation*}
and
\begin{equation*}
\xi_{xx}=2\phi^{-3}[\bar{u}_{xx}\phi-\bar{u}_x\phi_x]^2+\phi^{-3}[2\bar{u}_x\bar{u}_{xxx}
\phi^{2}-|\bar{u}_x|^2\phi\phi_{xx}].
\end{equation*}
Note that
\begin{equation*}
0=[\bar{d}\bar{u}_{xx}+f(\bar{u})]_x=\bar{d}\bar{u}_{xxx}+f'(\bar{u})\bar{u}_x.
\end{equation*}
It then follows that
\begin{eqnarray*}
&&e^{\gamma t}\left(\psi_t-\bar{d}\psi_{xx}+[\gamma-f'(\bar{u})]\psi\right)\\
&&=-[\bar{d} \xi_{xx}+f'(\bar{u})\xi]+\lambda_1(\bar{u},\bar{d}) \eta\\
&&=-2\bar{d}\phi^{-3}[\bar{u}_{xx}\phi-\bar{u}_x\phi_x]^2-\bar{d}
\phi^{-3}[2\bar{u}_x\bar{u}_{xxx}\phi^{2}-|\bar{u}_x|^2\phi\phi_{xx}]
-f'(\bar{u})|\bar{u}_x|^2\phi^{-1}+\lambda_1(\bar{u},\bar{d})\eta\\
&&\le
\bar{d}|\bar{u}_x|^2\phi^{-2}\phi_{xx}+f'(\bar{u})\phi^{-1}|\bar{u}_x|^2
-2f'(\bar{u})|\bar{u}_x|^2\phi^{-1}-2\bar{d}\phi^{-1}\bar{u}_x\bar{u}_{xxx}+\lambda_1(\bar{u},\bar{d})\eta\\
&&=\lambda_1(\bar{u},\bar{d})\xi+\lambda_1(\bar{u},\bar{d})\eta
-2\bar{u}_x\phi^{-1}[f'(\bar{u})\bar{u}_x+\bar{d}\bar{u}_{xxx}]\\
&&=\lambda_1(\bar{u},\bar{d})[\xi+\eta].
\end{eqnarray*}
Hence, $\psi_t-\bar{d}\psi_{xx}+[\gamma-f'(\bar{u})]\psi\le 0$
because $\lambda_1(\bar{u},\bar{d})\le 0$.

Since $\bar{u}$ is not a constant and $\psi(t,x)$ is $r$-periodic in
$x\in \R$, we can choose $x_0$ such that
$\psi(x_0,t)=\psi(x_0+r,t)=\min_{x\in\R}\psi(x,t)<0$, and hence,
$\psi_x|_{x=x_0}=\psi_x|_{x=x_0+r}=0$. Thus, $\psi(t,x)$ with $x\in
[x_0,x_0+r]$ satisfies the following equation
\begin{equation}\label{Eq:NeumannBV}
\begin{cases}
\psi_t-\bar{d}\psi_{xx}+[\gamma-f'(\bar{u})]\psi\le 0,x\in(x_0,x_0+r),\\
\psi_x|_{x=x_0}=\psi_x|_{x=x_0+r}=0,
\end{cases}
\end{equation}
and $\psi(t,x)$ attains its maximum $0$ at $(x^*,t)$ with
$x^*\in(x_0,x_0+r)$. By the strong maximum principle, we see that
$\psi(t,x)\equiv 0$, which implies that $\bar{u}_x/\phi$ is a
constant. Since $\bar{u}_x/\phi$ is $r$-periodic, it then follows
that $\bar{u}_x\equiv 0$, and hence, $\bar{u}$ is a constant, a
contradiction.
\end{proof}

\begin{remark}
By the proof above, it follows that the conclusion of Lemma
\ref{ConstantInY} is valid for any $f\in C^1$.
\end{remark}

\begin{lemma}\label{YOpenInX}
$Y$ is open in $C_{per}^1$.
\end{lemma}
\begin{proof}
Clearly, Lemma \ref{ConstantInY} implies that $Y\ne \emptyset$. Let
$d^*\in Y$ be given. We need to show that $d^*$ is an interior point
of $Y$. Assume, for the sake of contradiction, that there is a
sequence of points $d_n\in C_{per}^1\setminus Y$ such that $d_n\to
d^*$ in $C_{per}^1$ as $n\to \infty$. Then \eqref{Application4} with
$d=d_n$ admits a non-constant $r$-periodic steady state $u_n$ with
the principle eigenvalue $\lambda_1(u_n,d_n)\le 0$. Using the
transformation $v_n=d_n(u_n)_x$, we see that $(u_n,v_n)$ is a
periodic solution of the following ordinary differential system:
\begin{equation}
\begin{cases}
(u_n)_x=v_n/d_n,\\
(v_n)_x=-f(u_n).
\end{cases}
\end{equation}
By elementary phase plane arguments, it then follows that
\begin{equation}\label{estmiate1}
0\le \inf_{x\in \R} u_n(x)\le a\le \sup_{x\in \R} u_n(x)\le 1,\quad
\forall n\geq 1,\, \, x\in \R.
\end{equation}
Thus, the sequence of functions $((u_n)_x,(v_n)_x)$ is uniformly
bounded and equicontinuous, and hence, $(u_n,v_n)$ has a uniformly
convergent subsequence, still denoted by $(u_n,v_n)$. Let
$(u^*,v^*)$ be the limiting function of $(u_n,v_n)$. Then $u^*$ is
an $r$-periodic steady state of \eqref{Application4} with $d=d^*$.
It is easy to see from \eqref{estmiate1} that $u^*$ is not the
constant function $0$ or $1$.

Let $\phi_n$ be the positive eigenfunction associated with
$\lambda_1(u_n,d_n)$. Then
\begin{equation}\label{eigensequence}
(d_n(\phi_n)_x)_x+f'(u_n(x))\phi_n=\lambda_1(u_n,d_n) \phi_n.
\end{equation}
Dividing both sides of (\ref{eigensequence}) by $\phi_n$ and
integrating from $0$ to $r$, we obtain
\begin{equation}\label{negativeintegral}
\int_0^r \f{d_n[(\phi_n)_x]^2}{\phi_n^2}dx+\int_0^r
f'(u_n(x))dx=\lambda_1(u_n,d_n)r \le 0.
\end{equation}
Since $f\in C^1$ and $f'(a)>0$, we see that $u^*$ cannot be the
constant $a$. Otherwise, the uniform convergence of $u_n$ to $a$
implies that $f'(u_n(x))>0$ for all $x\in [0,r]$ and sufficiently
large $n$, which contradicts (\ref{negativeintegral}). Thus, $u^*$
is a non-constant $r$-periodic function. Since $d^*\in Y$, we  have
$\lambda_1(u^*,d^*)>0$.

Note that $u_n\to u^*$ in $C(\R,\R)$ and $d_n\to d^*$ in
$C_{per}^1$. By the variational characterization of the principal
eigenvalue $\lambda_1(u_n,d_n)$ (see, e.g., Eq. (5.2) of
\cite{BerHamelRoq}), it then follows that $0\ge\lambda_1(u_n,d_n)\to
\lambda_1(u^*,d^*)>0$, a contradiction.
%
%\begin{equation}
%-\lambda_1(u_n)=\min_{\rho\in H_{per}^1,\rho\not\equiv
%0}\f{\int_{0}^r \psi_n(x)\rho_x^2-f'(u_n(x))\rho^2}{\int_0^r
%\rho^2},
%\end{equation}
%where $H_{per}^1:=\{\rho\in H_{loc}^1(R)\text{ such that $\rho$ is
%$r$-periodic}\}$.
\end{proof}

The following counter-example shows that the parabolic equation
(\ref{Application4}) admits no bistability structure in the general
case of periodic function $d(x)$.

\begin{lemma}
Let either $f(u)=u(1-u^2)$, or $f(u)=u(1-u)(u-1/2)$.  Then there
exists a positive function $d\in C_{per}^1$ such that
\eqref{Application4} admits a pair of linearly stable, non-constant,
and $r$-periodic steady states.
\end{lemma}
\begin{proof}
We only consider the case where $f(u)=u(1-u^2)$ since the other one
can be obtained under appropriate scalings. Our proof is based on
the main result in \cite[Theorem 3]{FuscoHale}. Without loss of
generality, we assume that $r=4$. In what follows, we use some
notations of \cite{FuscoHale}.

Let $l\in(0,1)$ be fixed and $c^0$ be the step function on $[-1,1]$
defined by
\begin{equation}
c^0(x)=\begin{cases} 1,& x\in [-1,-l]\cup(l,1],\\
0,& x\in (-l,l].
\end{cases}
\end{equation}
Define $D:=\{(x,y): x=\pm l, y\in [0,1]\}\cup$ graph of $c^0$. By
\cite[Theorem 3]{FuscoHale}, it then follows that for any positive
even function $c\in C^1([-1,1],\R^+)$ which is sufficiently closed
to $c^0$ (in the sense that the distance between $D$ and the graph
of $c$ is small enough), the following Neumann boundary problem
\begin{equation}
\begin{cases}
u_t=(cu_x)_x+u(1-u^2),\, \, x\in(-1,1)\\
u_x(t,\pm 1)=0
\end{cases}
\end{equation}
admits an odd increasing steady state $u_c$ which is linearly
stable. That is, there exist $\lambda_1<0$ and $\phi>0$ such that
\begin{equation}
\begin{cases}
(c\phi_x)_x+f'(u_c)\phi=\lambda_1\phi,\, \, x\in(-1,1)\\
\phi_x(\pm 1)=0.
\end{cases}
\end{equation}
In particular, we can choose $c$ such that $c_x(-1)=c_x(1)=0$. Since
$c$ is even and $f$ is odd, we see that $v_c(x):=u_c(-x)$ is also a
steady state, and $\lambda_1$ is the corresponding eigenvalue with
the positive eigenfunction $\phi(-x)$.

Now we can construct a linearly stable $4$-periodic steady state of
\eqref{Application4}. Define two $4$-periodic functions:
\begin{equation*}
\tilde{d}(x)=\begin{cases} c(x),&x\in [-1,1]\\
c(2-x),& x\in (1,3)
\end{cases}
\quad\text{and}\quad
w_1(x)=\begin{cases} u_c(x),&x\in [-1,1]\\
u_c(2-x),& x\in (1,3).
\end{cases}
\end{equation*}
Then $w_1(x)$ is a $4$-periodic steady state of \eqref{Application4}
with $d=\tilde{d}$. Let the positive $4$-periodic function $\rho(x)$
be defined by
\begin{equation*}
\rho(x)=\begin{cases} \phi(x),&x\in [-1,1]\\
\phi(2-x),& x\in (1,3).
\end{cases}
\end{equation*}
It follows that $\lambda_1$ and $\rho$ solve the following
eigenvalue problem
\begin{equation*}
\begin{cases}
(\tilde{d}\rho_x)_x+f'(w_1)\rho=\lambda_1\rho,\, \, \,x\in \R\\
\text{$\rho$ is $r$-periodic}.
\end{cases}
\end{equation*}
This implies that $w_1$ is a linearly stable periodic steady state
of \eqref{Application4} with $d=\tilde{d}$. Similarly, so is
$w_2(x):=w_1(x+2)$.
\end{proof}

As a consequence of Theorem \ref{ThMediaPeri}, together with Lemmas
\ref{ConstantInY} and \ref{YOpenInX}, we have the following result
on the existence of bistable traveling waves for
\eqref{Application4}.

\begin{theorem}\label{ThApplication4}
Let $\bar{d}$ be a given positive constant. Then there exists
$\delta_0>0$ such that for any $d\in C_{per}^1$ with
$\|d-\bar{d}\|_{C^0}<\delta_0$, \eqref{Application4} admits a
spatially periodic traveling wave solution $u(t,x):=V(x+ct,x)$ with
some speed $c\in \R$ and connecting $0$ to $1$. Besides, $V(\xi,x)$
is nondecreasing in $\xi$.
\end{theorem}

We remark that Theorem \ref{ThApplication4} is a $C^0$-perturbation
result in $C_{per}^1$, and hence, it improves the existence result
in \cite[Theorem 3.1]{XinReview}, where the $H^s$-perturbation is
used for some $s>2$.

\subsection{A nonlocal and time-delayed reaction-diffusion equation}
Let $\tau>0$ be a fixed real number. Choose
$\mc{X}:=C([-\tau,0],\mathbb{R}), \mc{Y}:=C(\mathbb{R},\mathbb{R})$
and $\mc{C}:=C([-\tau,0],\mc{Y})$. We equip $\mc{X}$ with the
maximum norm, $\mc{Y}$ and $\mc{C}$ with the similar norms as in
\eqref{NormDef}. Define $\mc{Y}_+:=C(\mathbb{R},\mathbb{R}_+)$. Let
$d$ be the metric in $\mc{C}(\mc{Y})$ induced by the norm. We are
interested in bistable traveling waves of the following nonlocal and
time-delayed reaction-diffusion equation:
\begin{equation}\label{Application3}
\begin{cases}
\f{\partial u(t,x)}{\partial t}=\f{\partial^2 u(t,x)}{\partial
x^2}+f(u_t)(x),&\quad  t>0, x\in\mathbb{R}\\
u_0=\phi\in \mc{C}, &\quad \theta\in[-\tau,0],
\end{cases}
\end{equation}
where $f: \mc{C}\to \mc{Y}$ is Lipschitz continuous and for each
$t\ge 0$, $u_t\in \mc{C}$ is defined by
\begin{equation*}
u_t(\theta,x):=u(t+\theta,x),\quad \forall
\theta\in[-\tau,0],x\in\mathbb{R}.
\end{equation*}
If the functional $f$ takes the form
$f(\phi)(x)=F(\phi(0,x),\phi(-\tau,x))$, then \eqref{Application3}
becomes a local and time-delayed reaction-diffusion equation:
\begin{equation} \label{localE}
\f{\partial u(t,x)}{\partial t}=\f{\partial^2 u(t,x)}{\partial
x^2}+F(u(t,x),u(t-\tau,x)).
\end{equation}
The bistable traveling waves of (\ref{localE}) were studied in
\cite{Schaaf}. If $f(\phi)(x)=-d \phi(0,x)+
\int_{\mathbb{R}}b(\phi(-\tau,y))k(x-y)\di y$, then
\eqref{Application3} becomes a nonlocal and time-delayed
reaction-diffusion equation:
\begin{equation}\label{nonlocalE}
\f{\partial u(t,x)}{\partial t}=\f{\partial^2 u(t,x)}{\partial
x^2}-d u(t,x)+ \int_{\mathbb{R}}b(u(t-\tau,y))k(x-y)\di y.
\end{equation}
The existence, uniqueness and stability of bistable waves of
(\ref{nonlocalE}) were established in \cite{MaWu07}.

Note that $\mathbb{R}$ can be regarded as a subspace of $\mc{X}$,
and the latter can also be regarded as a subspace of $\mc{C}$.
Define $\bar{f}:\mc{X}\to \mathbb{R}$ by
$\bar{f}(\varphi)=f(\varphi)$ and $\hat{f}:\mathbb{R}\to \mathbb{R}$
by $\hat{f}(\xi)=f(\xi)$.  In order to obtain the existence of
bistable waves for system \eqref{Application3}, we impose the
following assumptions on the functional $f$:
\begin{enumerate}
\item[(E1)]$0<\alpha<\beta$ are three equilibria and there are no other equilibria between
$0$ and $\beta$.
\item[(E2)] The functional $f:\mc{C}_\beta\to \mc{Y}$ is quasi-monotone
in the sense that
\begin{equation*}
\lim_{h\to 0^+}\f{1}{h}
d([\phi(0)-\psi(0)]+h[f(\phi)-f(\psi)];\mc{Y}_+)=0\quad
\text{whenever}\, \, \phi\ge\psi \, \, \text{in} \, \, \mc{C}_\beta.
\end{equation*}
\item[(E3)]Equilibria $0$ and $\beta$ are stable, and $\alpha$ is
unstable in the sense that $\hat{f}'(0)<0,\hat{f}'(\alpha)>0$ and
$\hat{f}'(\beta)<0$.
\item[(E4)]For each $\varphi\in \mc{X}_\beta$, the
derivative $\bar{L}(\varphi):=\D \bar{f} (\varphi)$ of $\bar{f}$ can
be represented as
\begin{equation*}
\bar{L}(\varphi)\chi=a(\varphi)\chi(0)+\int_{-\tau}^0\chi(\theta)\di_\theta
\eta(\varphi):=a(\varphi)\chi(0)+L_1(\varphi)\chi,
\end{equation*}
where $\eta(\varphi)$ is a positive Borel measure on $[-\tau,0]$ and
$\eta(\varphi)([-\tau,-\tau+\epsilon])>0$ for all small
$\epsilon>0$.
\item[(E5)]For any small number $\epsilon>0$, there exits a number
$\delta\in(0,\beta)$ and a linear operator $L_\epsilon:
\mc{C}_\beta\to \mc{Y}$ such that $L_\epsilon \phi\to \D
f(\alpha)\phi,\forall \phi\in \mc{C}_\beta$, as $\epsilon\to 0$ and
that
\begin{equation*}
f(\alpha+\phi)\ge L_\epsilon (\phi),\quad \text{and}\quad
f(\alpha-\phi)\le -L_\epsilon (\phi),\quad \forall \phi\in
\mc{C}_\delta.
\end{equation*}
\end{enumerate}

Using the solution maps $\{T(t)\}_{t\ge0}$ generated by the heat
equation $\f{\partial u(t,x)}{\partial t}=\f{\partial^2
u(t,x)}{\partial x^2}$, we write system \eqref{Application3} as the
integral form
\begin{equation}\label{IntegralFormAppl3}
\begin{cases}
u(t,\cdot)=T(t)\phi(0,\cdot)+\int_0^t T(t-r)f(u_r(\cdot,\cdot))\di r,&\quad t>0\\
u(\theta,\cdot)=\phi(\theta,\cdot),&\quad \theta\in[-\tau,0].
\end{cases}
\end{equation}
Note that traveling waves of system \eqref{IntegralFormAppl3} are
those of system \eqref{Application3}. It then remains to show
\eqref{IntegralFormAppl3} admits a bistable traveling wave.

\begin{theorem}
Under assumption (E1)-(E5), system \eqref{Application3} admits a
nondecreasing traveling wave $\phi(x+ct)$ with $\phi(-\infty)=0$ and
$\phi(+\infty)=\beta$.
\end{theorem}
\begin{proof}
From assumptions (E1)-(E2), we see that system
\eqref{IntegralFormAppl3} generates a monotone semiflow
$\{Q_t\}_{t\ge 0}$ on $\mc{C}_\beta$ with
\begin{equation*}
Q_t[\phi](\theta,x)=u_t(\theta,x;\phi), \quad \forall
(\theta,x)\in[-\tau,0]\times \mathbb{R},
\end{equation*}
where $u(t,x;\phi)$ is the unique solution of system
\eqref{IntegralFormAppl3} satisfying
$u_0(\cdot,\cdot;\phi)=\phi\in\mc{C}_\beta$. By similar arguments as
in section 5, it follows that $Q_t$ satisfies (A4) if $t>\tau$ and
(A4$'$) if $t\in(0,\tau]$.

Let $\bar{Q}_t$ be the restriction of $Q_t$ on $\mc{X}_\beta$.
Denote the derivative $\D\bar{Q}_t[\hat{0}]$ of $\bar{Q}_t$ by
$\bar{M}_{0,t}$, then $\bar{M}_{0,t}$ is the solution map of the
following functional equation:
\begin{equation}\label{xIndependent3}
\f{\di u}{\di t}=\bar{L}(0)u_t=a(0)u(t)+L_1(0)u_t.
\end{equation}
By assumptions (E2) and (E4), it follows that system
\eqref{xIndependent3} admits a principle eigenvalue $s_0$ with an
associated eigenfunctions $v_0:=e^{s_0\theta}$ (see \cite[Theorem
5.5.1]{Smith}). More precisely, $\bar{M}_{0,t}[v_0]=e^{s_0t}v_0$.
Furthermore, \cite[Corollary 5.5.2]{Smith} implies that $s_0<0$
since $\hat{f}'(0)<0$. Therefore, there exists $\delta_0(t)>0$ such
that
\begin{eqnarray*}
\bar{Q}_t[\delta v_0]&&=\bar{Q}_t[0]+\D \bar{Q}_t[0][\delta
v_0]+o(\delta^2)\\
&&=\delta \bar{M}_{0,t}[v_0]+o(\delta^2)\\
&&=\delta e^{s_0t}v_0+o(\delta^2)\\
&&=\delta v_0+\delta [e^{s_0t}-1]v_0+o(\delta^2)\ll \delta v_0,\quad
\forall \delta\in(0,\delta_0(t)].
\end{eqnarray*}
Similarly, there exists $\delta_\alpha(t), v_\alpha$ and
$\delta_\beta(t), v_\beta$ such that
\begin{equation*}
\bar{Q}_t[\beta-\delta v_\beta]\gg \beta-\delta v_\beta, \quad
\forall \delta\in(0,\delta_\beta(t)]
\end{equation*}
and
\begin{equation*}
\bar{Q}_t[\alpha+\delta v_\alpha]\gg \alpha+\delta v_\alpha,\quad
\bar{Q}_t[\alpha-\delta v_\alpha]\ll \alpha-\delta v_\alpha,\quad
\forall  \delta\in(0,\delta_\alpha(t)].
\end{equation*}
Till now, it remains to show (A6) is also true. Indeed, we see from
\cite[Theorem 2.17]{LiangZhao} that the solution semiflows
$\{Q_t\}_{t\ge0}$ restricted on $[0,\alpha]_{\mc{C}}$ and
$[\alpha,\beta]_{\mc{C}}$ admit a spreading speed $c^*(0,\alpha)$
and $c^* (\alpha,\beta)$, respectively. Let $M_\epsilon^t$ be the
solution maps of the linear system
\begin{equation*}
\begin{cases}
u(t,\cdot)=T(t)\phi(0,\cdot)+\int_0^t T(t-r)L_\epsilon (u_r(\cdot,\cdot))\di r,&\quad t>0\\
u(\theta,\cdot)=\phi(\theta,\cdot),&\quad \theta\in[-\tau,0].
\end{cases}
\end{equation*}
Then assumption (E5) guarantees that $Q_t[\phi]\ge
M^\epsilon_t[\phi]$ when $\phi\in \mc{C}_\delta$, where
$\delta=\delta(\epsilon)$ is defined in (E5). Therefore, we see from
\cite[ Theorem 3.10]{LiangZhao} that $c^*(0,\alpha)\ge \bar{c}$ and
$c^* (\alpha,\beta)\ge \bar{c}$, where $\bar{c}$ is positive number
determined by the linearized system of \eqref{Application3} at
$u\equiv \alpha$, and hence, (A6) holds. Consequently, Theorem
\ref{Delay} completes the proof.
\end{proof}

\begin{remark}\label{stabilityremark}
At this moment we are unable to present a general result on the
uniqueness and global attractivity of bistable waves under the
current abstract setting. However, one may use the convergence
theorem for monotone semiflows (see \cite[Theorem 2.2.4]{ZhaoBook})
and the similar arguments as in the proof of \cite[Theorem
10.2.1]{ZhaoBook} and \cite[Theorem 3.1]{XuZhao} to obtain the
global attractivity (and hence, uniqueness) of bistable waves for
four examples in this section.
\end{remark}

\section{Appendix}
In this appendix, we present certain properties of Banach lattices
and countable subsets in $\R$, and some convergence results for
sequences of monotone functions, including an abstract variant of
Helly's theorem.

\begin{proposition}\label{BasicProp}
A Banach lattice $\mc{X}$ has the following properties:
\begin{enumerate}
\item[(1)]For any $u,v\in \mc{X}$ with $v\in \mc{X}^+$,
if $-v\le u\le v$, then $\|u\|_{\mc{X}} \le \|v\|_{\mc{X}}$.
\item[(2)]If $u_k\to u$ and $v_k\to v$ in $\mc{X}$ with $u_k\ge v_k$,
then  $u\ge v$.
\end{enumerate}
\end{proposition}

\begin{proposition}\label{SequenceInC}
The space $\mc{C}$ has the following properties:
\begin{enumerate}
\item[(1)] Let $\phi$ be a monotone function in $\mc{C}$. If $x_k\in\mc{H}$
nondecreasingly tends to $x\in\mc{H}\cup \{+\infty\}$ and
$\lim_{k\to\infty}\phi(x_k)=u\in \mc{X}$, then $\lim_{y \uparrow
x}\phi(y)=u$. The similar result holds if $x_k$ nonincreasingly
tends to $x\in \mc{H}\cup \{-\infty\}$.
\item[(2)]Assume that $h,h_k:\mc{H}\to \mc{H}$ are continuous
and $\phi_k\to \phi$ in $\mc{C}$. If $h_k(x)\to h(x)$ uniformly for
$x$ in any bounded subset of $\mc{H}$, then $\phi_k\circ h_k\to
\phi\circ h$ in $\mc{C}$.
\end{enumerate}
\end{proposition}
Propositions \ref{BasicProp} and \ref{SequenceInC} can be easily
proved. Here we omit the proofs.

\begin{proposition}\label{SetProp1}
Assume that $D$ is a countable subset of $\mathbb{R}$. Then for any
$c\in\mathbb{R}$, there exists another countable subset $A$ of
$\mathbb{R}$ such that $(\R\setminus A)+cm\subset
\mathbb{R}\setminus D, \forall m\in \mathbb{Z}^+$.
\end{proposition}
\begin{proof}
It suffices to show the set $A:=\{x\in\mathbb{R}: \text{there exists
$m$ such that $x+cm\in D$}\}$ is countable. Indeed, we have
$A=\cup_{m=1}^\infty (D-cm)$. This implies $A$ is countable.
\end{proof}

\begin{proposition}\label{SetProp2}
For any countable subset $\Gamma_1$ of $\mathbb{R}$, there exists
another countable $\Gamma_2$ such that it is dense in $\mathbb{R}$
and $\Gamma_1\cap \Gamma_2=\emptyset$.
\end{proposition}
\begin{proof}
Since $\Gamma_1$ is countable and $\cup _{\alpha\in
\mathbb{R}}(\alpha+\mathbb{Q})=\mathbb{R}$, there must exist a
sequence $\alpha_n$ such that $\Gamma_1\subset\cup_{n=1}^\infty
(\alpha_n+\mathbb{Q})$. Note that $\cup_{n=1}^\infty
(\alpha_n+\mathbb{Q})$ is countable. Then we see that there exists
$\alpha\in \mathbb{R}$ such that $\alpha\not\in \cup_{n=1}^\infty
(\alpha_n+\mathbb{Q})$. This means that $\alpha-\alpha_n\not\in
\mathbb{Q}, \forall n\ge 1$, and hence, $(\alpha+\mathbb{Q})\cap
(\alpha_n+\mathbb{Q})=\emptyset,\forall n\ge 1$. Define
$\Gamma_2:=\alpha+\mathbb{Q}$. We then see that $\Gamma_2$ is
countable and dense in $\mathbb{R}$, and $\Gamma_1\cap
\Gamma_2=\emptyset$.
\end{proof}

\begin{proposition}\label{Convergence}
Assume that $f,f_n:\mathbb{R}\to \mc{X}$ are nondecreasing and the
set $D$ is dense in $\mathbb{R}$. If $s_n\to 0$, $f(s)$ is
continuous on $D$ and $f_n(s)\to f(s)$ for every $s\in D$, then
$f_n(s+s_n)\to f(s)$ for every $s\in D$.
\end{proposition}
\begin{proof}
Let $s\in D$ be fixed. For any $\delta >0$, since $D-s$ is dense in
$\mathbb{R}$, we can choose $\delta _+\in (D-s)\cap (0,\delta)$ and
$\delta _-\in (D-s)\cap (-\delta,0)$. Clearly, $s+\delta _+\in D$,
$s+\delta _-\in D$. Thus, there exists an integer $N_{\delta}$ such
that $s+s_n\in (s+\delta _-, s+\delta _+), \, \forall n\geq
N_{\delta}$. Since
\begin{equation*}
f_n(s+\delta _-)-f_n(s+\delta _+)\leq f_n(s+s_n)-f_n(s)\leq
f_n(s+\delta _+)-f_n(s+\delta _-), \, \, \forall n\geq N_{\delta},
\end{equation*}
we have
\begin{equation*}
\|f_n(s+s_n)-f_n(s)\|_{\mc{X}}\leq \|f_n(s+\delta _+)-f_n(s+\delta
_-)\|_{\mc{X}},\,\, \forall n\geq N_{\delta}.
\end{equation*}
It then follows that
\begin{eqnarray*}
\|f_n(s+s_n)-f(s)\|_{\mc{X}}&\leq& \|f_n(s+s_n)-f_n(s)\|_{\mc{X}}+\|f_n(s)-f(s)\|_{\mc{X}}\nonumber \\
&\leq& \|f_n(s+\delta _+)-f_n(s+\delta _-)\|_{\mc{X}}+\|f_n(s)-f(s)\|_{\mc{X}}\nonumber \\
&\leq& \|f(s+\delta _+)-f_n(s+\delta _+)\|_{\mc{X}} +\|f(s+\delta
_+)-f(s+\delta _-)\|_{\mc{X}}\nonumber\\
&\,& \, +\|f_n(s+\delta _-)-f(s+\delta
_-)\|_{\mc{X}}+\|f_n(s)-f(s)\|_{\mc{X}}
\end{eqnarray*}
for all $n\geq N_{\delta}$. Now the pointwise convergence of $f_n$
in $D$ and the continuity of $f$ on $D$ complete the proof.
\end{proof}

To end this section, we prove a convergence theorem for sequences of
monotone functions from $\R$ to the special Banach lattice
$C(M,\R^d)$ defined in section 2, which is a variant of Helly's
theorem \cite[P.165]{Doob} for sequences of monotone functions from
$\R$ to $\R$.

\begin{theorem}\label{HellyTh}
Let $D$ be a dense subset of $\mathbb{R}$ and $f_n, n\ge 1$ be a
sequence of nondecreasing functions from $\R$ to the Banach lattice
$\mc{X}:=C(M,\R^d)$. Assume that
\begin{enumerate}
\item[(i)] for any $s\in D$, $f_n(s)$ is convergent in $\mc{X}$.
\item[(ii)]there exists a countable set $D_1\subset \R$ such that
for any $s\in \R\setminus D_1$, the limits $\lim\limits_{m\to
\infty}\lim\limits_{n\to \infty} f_n(s_{\pm, m})$ exist in $\mc{X}$,
where $s_{-,m} \uparrow s$ and $s_{+,m} \downarrow s$ with $s_{\pm,
m}\in D$.
\end{enumerate}
Then $f_n(s)$ is convergent in $\mc{X}$ almost for all $s\in \R$.
\end{theorem}
\begin{proof}
Due to assumption (ii), we can define $f:\R\to \mc{X}$ by
\begin{equation}
f(s):=
\begin{cases}
\lim\limits_{x\uparrow s}\lim\limits_{x\in D, n\to \infty} f_n(x),&
s\in
\R\setminus D_1\\
\text{any value}, & s\in D_1.
\end{cases}
\end{equation}
We first show that the discontinuous points of $f$ are at most
countable. Define the sets
\begin{equation*}
A:=\{s\in\mathbb{R}\setminus D_1: f(s^-),f(s^+)\,\text{both exits}\}
\end{equation*}
and
\begin{equation*}
B:=\{s\in A: f(s^-)<f(s^+)\}.
\end{equation*}
For any $s\in B$, there exists $x_1\in M $ and $1\le i\le d$ such
that $\left(f(s^-)(x_1)\right)_i<\left(f(s^+)(x_1)\right)_i$. Recall
that $M$ is compact, so there is a countable dense subset $M_1$. It
then follows that there must be $x_2\in M_1$ such that
$\left(f(s^-)(x_2)\right)_i<\left(f(s^+)(x_2)\right)_i$. Therefore,
\begin{equation*}
B=\cup_{i=1}^m \cup_{x\in M_1} \{s\in
A:\left(f(s^-)(x)\right)_i<\left(f(s^+)(x)\right)_i\}.
\end{equation*}
Since for each fixed $i$ and $x$, $\left(f(s)(x)\right)_i$ is a
nondecreasing function from $\mathbb{R}\setminus D_1$ to
$\mathbb{R}$, we know that $ \{s\in
A:\left(f(s^-)(x)\right)_i<\left(f(s^+)(x)\right)_i\}$ is at most
countable, and hence so is the set $B$.

Now we can prove the conclusion. Assume that $s\in\mathbb{R}$ is a
continuous point of $f$. For any $\delta>0$, choose $\delta_-\in
D\cap (s-\delta,s)$ and $\delta_+\in D\cap (s,s+\delta)$. Then we
have
\begin{equation*}
f_n(\delta_-)-f_n(\delta_+) \le f_n(s)-f_n(\delta_-)\le
f_n(\delta_+)-f_n(\delta_-),\,\forall n \ge 1,
\end{equation*}
which, together with Proposition \ref{BasicProp}(2), implies that
\begin{equation}\label{3}
\|f_n(s)-f_n(\delta_-)\|_{\mc{X}}\le
\|f_n(\delta_+)-f_n(\delta_-)\|_{\mc{X}},\,\forall n \ge 1.
\end{equation}
In the other hand, by \eqref{3} and the triangular inequality we
have
\begin{eqnarray*}
\|f_n(s)-f(s)\|_{\mc{X}} &&\le
\|f_n(s)-f_n(\delta_-)\|_{\mc{X}}+\|f_n(\delta_-)-f(s)\|_{\mc{X}}\nonumber\\
&& \le \|f_n(\delta_+)-f_n(\delta_-)\|_{\mc{X}}
+\|f_n(\delta_-)-f(s)\|_{\mc{X}}\nonumber\\
&& \le \|f_n(\delta_+)-f(\delta_+)\|_{\mc{X}}
+\|f(\delta_+)-f(\delta_-)\|_{\mc{X}}\nonumber\\
&&\quad +\|f(\delta_-)-f_n(\delta_-)\|_{\mc{X}}+\|f_n(\delta_-)
-f(\delta_-)\|_{\mc{X}}\nonumber\\
&&\quad+\|f(\delta_-)-f(s)\|_{\mc{X}},\,\forall n \ge 1.
\end{eqnarray*}
Now the pointwise convergence of $f_n$ in $D$ and the continuity of
$f$ at $s$ complete the proof.
\end{proof}

\

\noindent {\bf Acknowledgment. }  J. Fang's research is supported in
part by the NSF of China (grant 10771045) and the Collaborative
Research Groups Program at HIT. X.-Q. Zhao's research is supported
in part by the NSERC of Canada and the MITACS of Canada.

\end{document}